\documentclass[11pt,oneside,reqno]{amsproc}
\usepackage{amsmath, amssymb, amsthm, verbatim, euscript, amscd, textcomp, mathrsfs}
\usepackage{graphicx}
\usepackage{epsfig}
 \usepackage[all,cmtip]{xy}
\usepackage{color}
 \usepackage[fleqn]{cases}
 \usepackage{marginnote}

 \usepackage{tikz}
 \usetikzlibrary{arrows,arrows.meta,bending,positioning}
 \usetikzlibrary{decorations.pathreplacing}
 
\usepackage[dvipsnames]{xcolor}

\usepackage{array}

\newcommand{\hide}[1]{}

\newcolumntype{C}[1]{>{\centering\arraybackslash}S{m{#1}}}

\newlength{\dhatheight}

\usepackage{hyperref}
\hypersetup{colorlinks, linkcolor=red, citecolor=green}

\def\R{\mathbb R}

\def\U{\mathcal U}

\def\cP{\mathbb P}

\def\E{\EuScript E}

\def\Circle{\mathbb S^1}

\def\Lip{\text{Lip}}

\def\be{\begin{eqnarray}}
\def\ee{\end{eqnarray}}

\newtheorem{mainthm}{Theorem}

\newtheorem{theorem}{Theorem}[section]

\newtheorem{prop}[theorem]{Proposition}

\newtheorem{corollary}[theorem]{Corollary}
\newtheorem{lemma}[theorem]{Lemma}
 \theoremstyle{remark}
\newtheorem{remark}[theorem]{Remark}

\theoremstyle{remark}

\renewcommand\R{\mathbb R}

\renewcommand\U{\mathcal U}

\renewcommand\phi{\varphi}
\newcommand\eps{\varepsilon}
\renewcommand\ge{\geqslant}
\renewcommand\le{\leqslant}
\newcommand\ld{\ldots}

\DeclareMathOperator{\Lyap}{\mathscr{L}}

\DeclareMathOperator{\id}{id\!}

\newcommand{\Ho}{{H\"ol\-der} }
\newcommand{\Li}{{Liv\-\v{s}ic} }

\definecolor{RevisionPink}{RGB}{0,0,0}

\newcommand{\rev}[1]{{\color{RevisionPink}#1}}

\DeclareMathOperator{\dist}{dist}

\allowdisplaybreaks[2]
\numberwithin{equation}{section}
\title{Unmarked spectral rigidity of expanding circle maps}

\author[Kostiantyn Drach]{Kostiantyn Drach}
\author[Vadim Kaloshin]{Vadim Kaloshin}

\thanks{\ The authors are supported by ERC Advanced Grant ``SPERIG'' (\#885707). The first author is also partially supported from grants 2021 SGR 00697 (Generalitat de Catalunya), CEX2020-001084-M (Maria de Maeztu Excellence program), PID2023-147252NB-I00 and CNS2025-166633 (AEI)}

\address{\hskip -7mm Universitat de Barcelona, Gran Via de les Corts Catalanes, 585, 08007 Barcelona, Spain \newline Centre de Recerca Matem\`atica, Edifici C, Carrer de l'Albareda, 08193 Bellaterra, Spain}
\email{kostiantyn.drach@ub.edu}
\address{\hskip -7mm Institute of Science and Technology Austria, Am Campus 1, 3400 Klosterneuburg, Austria}
\email{vadim.kaloshin@gmail.com}

\date{2 October 2026}

\begin{document}

\begin{abstract}
For a smooth expanding map $f$ of the circle, its \emph{(unmarked) length spectrum} is defined as the set of logarithms of multipliers of periodic orbits of $f$. This spectrum is analogous to the set of lengths of all closed geodesics on negatively curved surfaces -- the classical length spectrum. In this paper, we prove a length spectral rigidity result for expanding circle maps. Namely, in finite regularity, we show that a \emph{prevalent}, and hence \emph{dense}, set of expanding circle maps $f$ of degree $d \ge 2$ is locally rigid, that is, every nearby isospectral map is smoothly conjugate to $f$. 

The proof introduces a Liv\v sic--Whitney approximation scheme for hyperbolic systems. The approximation step consists of constructing an approximate coordinate change of low regularity using a finite \Li theorem and modifying this change using the Whitney extension theorem. This approximation scheme is reminiscent of a KAM scheme.
\end{abstract}
\maketitle
 
\setcounter{tocdepth}{1} 
\tableofcontents

%%%%%%%%%%%%%%%%%%%%%%%%%%%%%%%%%%%%%%%%%
%%%%%%%%%%%%%%%%%%%%%%%%%%%%%%%%%%%%%%%%%

\section{Introduction}\label{Sec:Intro}
%%%%%%%%%%%%%%%%%%%%%%%%%%%%%%%%%%%%%%%%%
%%%%%%%%%%%%%%%%%%%%%%%%%%%%%%%%%%%%%%%%%

In 1990, Croke \cite{Croke} and Otal \cite{Otal} proved a remarkable result on rigidity of negatively curved metrics in dimension $2$. They showed that a smooth metric $g$ of negative curvature on a closed surface is uniquely defined (up to smooth coordinate changes) by its \emph{marked length spectrum}, i.e., by the lengths of closed geodesics for the metric $g$ `marked' by their respective homotopy types. As it turns out, knowing just the \emph{length spectrum} of $g$, i.e., the set of lengths of all closed geodesics and `forgetting' about their homotopy types is not enough to reconstruct the metric, as the examples of Sunada \cite{Sunada} and Vign\'eras \cite{Vigneras} show. However, the \emph{local (unmarked) length spectral  rigidity} question for \emph{nearby} negatively curved metrics is wide open. In this paper, we study a one-dimensional analog of this question for expanding circle endomorphisms. 

%%%%%%%%%%%%%%%%%%%%%%%%%%%%%%%%%%%%%%%%%%%%%%%%%%%%%%%%
\subsection{Setup and the statements of the main results}
%%%%%%%%%%%%%%%%%%%%%%%%%%%%%%%%%%%%%%%%%%%%%%%%%%%%%%%%

Let $f\colon\Circle\to\Circle$, $\Circle=\mathbb R/\mathbb Z$, be a $C^k$-smooth expanding circle endomorphism of degree $d\ge2$, normalized by $f(0)=0$ and $f'(x)>1$ for all $x \in \Circle$. Here, $k$ is an integer and $C^k$-smooth means that $f$ is $k$ times continuously differentiable. For brevity, write $\E_d^k$ for the class of such maps.

Denote by $\mathbb P_n^f$ the set of all periodic points of (possibly not minimal) period $n$ for $f \in \E_d^k$. The \emph{$\log$-multiplier} of a periodic point $p \in \mathbb P_n^f$ is defined as 
\[
\lambda_n^f(p) := \log \left(f^n\right)'(p).
\]    
The averaging of the $\log$-multiplier $1/n \cdot \log \left(f^n\right)'(p)$ gives the classical \emph{Lyapunov exponent} of the periodic cycle $p \mapsto f(p) \mapsto \ldots \mapsto f^n(p) = p$ (which is independent of the point in the cycle).
 
In \cite{McM}, McMullen calls the $\log$-multiplier of $p$ `the length' of the periodic cycle $p \mapsto f(p) \mapsto \ldots \mapsto f^n(p) = p$, and draws some further analogy between the sets of $\log$-multipliers and the lengths of closed geodesics on hyperbolic surfaces. In \cite{HKS}, Huang, Kaloshin, and Sorrentino established a connection between the Lyapunov exponents and the lengths of certain periodic orbits in the context of convex billiards (see also \cite{BSKL} for a similar connection in the context of dispersing billiards).

Following the analogy outlined above, for each $n \in \mathbb N$, we define the \emph{length (Lyapunov) spectrum at level $n$} as the set $\Lyap_n(f) := \left\{\lambda_n^f(p) : p \in \mathbb P_n^f\right\}$. The union 
\[
\Lyap(f) := \bigcup_{n \in \mathbb N} \Lyap_n(f)
\]
of all these sets yields the \emph{(unmarked) length spectrum} of the expanding circle map $f$. We will use the terms `length spectrum' and `Lyapunov spectrum' interchangeably.

There is a natural \emph{marked} counterpart of the length spectrum. Namely, it is known that any two expanding circle maps $f, g \in \E^k_d$ are topologically conjugate via an orientation-preserving homeomorphism $\phi \colon \mathbb S^1 \to \mathbb S^1$ as follows:
\[
\phi\circ g=f\circ\phi,\qquad\phi(0)=0.
\] 
This homeomorphism respects the symbolic dynamics and hence provides a natural marking: we say that $p^f \in \mathbb P_n^f$ and $p^g \in \mathbb P_n^g$ are \emph{corresponding periodic points} if $p^f = \phi(p^g)$. We call $\phi$ the \emph{marking conjugacy} and say that $f$ and $g$ have the same \emph{marked length (or Lyapunov) spectra} if $\lambda_n^f(p^f) = \lambda_n^g(p^g)$ for every $n$ and every pair of corresponding periodic points. For $k\ge2$, by the classical result of Shub and Sullivan \cite{SS}, the marked spectrum defines an expanding circle map up to a \emph{smooth} change of coordinates, namely, if $f$ and $g$ have the same marked length spectra, then the marking conjugacy $\phi$ is $C^k$-smooth.

We are interested in the following question: 

\begin{center}
{\it Does the (unmarked) length spectrum of an expanding circle map uniquely define the smooth conjugacy class of the map?}
\end{center}

Similarly to the unmarked length spectrum setup for negatively curved metrics, in general the answer to this question is `no': 

\begin{mainthm}[Arbitrarily close nonconjugate isospectral maps]
\label{Thm:Counter}
For every integer $k\ge1$, $d \ge 2$, and every $\eps>0$, there exist $f, g \in \E_d^k$ such that 
\[
\|f-g\|_{C^k}<\eps \quad \text{ and }\quad \Lyap(f)=\Lyap(g),
\]
but their marking conjugacy $\phi$ is not $C^1$. (Here, $\|\cdot \|_{C^k}$ denotes the $C^k$-norm on the circle.) Moreover, $f, g$ can be chosen to be arbitrarily close to the linear map $L_d \colon x \mapsto d\cdot x \pmod 1$. 

Furthermore, there exists $d \ge 2$ so that $f, g$ can be chosen to be not conjugate by any $C^1$-smooth diffeomorphism, orientation preserving or not.
\end{mainthm}

\medskip

Despite the result above, the \emph{local} rigidity is a highly non-trivial question. We say that $g \in \E_d^k$ is \emph{$C^k$-locally spectrally rigid} if there exists $\eps = \eps(g) >0$ such that for all $f \in \E_d^k$ with
\[
\|f-g\|_{C^k}<\eps,\qquad \Lyap(f)=\Lyap(g)
\]
it follows that $f$ and $g$ are $C^k$-smoothly conjugate (i.e., their marking conjugacy is $C^k$-smooth).

Our second main result establishes local rigidity for a \emph{prevalent} set of maps in $\E_d^k$.

Recall the following definition of \emph{prevalence} due to Hunt--Sauer--Yorke~\cite{HSY} (see also \cite{KH}). Let $\mathcal U \subset \E_d^k$ be an open set, where we view $\E_d^k$ embedded into the Banach space $C^k(\Circle, \mathbb R)$ of $C^k$-smooth real-valued $1$-periodic functions fixing the origin and equipped with the norm $\|\cdot\|_{C^k}$. A subset $\mathcal A \subset \mathcal U$ is called \emph{prevalent in $\mathcal U$} if the complement $\U \setminus \mathcal A$ is \emph{shy} in the ambient Banach space, i.e., there exists a compactly supported probability measure $\mu$ on $C^k(\Circle, \mathbb R)$ satisfying 
\[
\mu(u +\mathcal U \setminus \mathcal A) = 0 \qquad \text{for every }\quad u \in C^k(\Circle, \R).
\]
The measure $\mu$ is the product Lebesgue probability measure on the Hilbert cube of Fourier coefficients of $C^k$ functions (see Appendix~\ref{SSec:Random}). It follows that if $\mathcal A$ is prevalent in $\mathcal U$, then it is dense there.
For fixed expansion bounds $1<\Lambda<d<\Omega$, we define $\mathcal U$ as follows:
\begin{equation}\label
{eq:family-U}
 \mathcal U :=\{g\in\E_d^k:\Lambda<g'(x)<\Omega \text{ for all }x\in\Circle\}.
\end{equation}

\begin{mainthm}[Prevalent unmarked length spectral rigidity]
\label{Thm:Prev:Main}
For every $d\ge2$ and $1<\Lambda<d<\Omega$, there is an integer $k_0=k_0(d,\Lambda,\Omega)$ such that for every $k \ge k_0$ the $C^k$-locally spectrally rigid maps are prevalent in $\mathcal U$. In particular, locally rigid maps are dense in $\mathcal U$. 

%The prevalent subset can be chosen to contain no map $C^1$ conjugate to $L_d \colon x\mapsto dx$.
\end{mainthm}

The proof is a combination of the two main results below. The first establishes local rigidity under a \emph{sparsity} assumption on the spectrum. The second shows that this assumption is satisfied by a prevalent set of maps. These results give the quantitative statements behind Main Theorem~\ref{Thm:Prev:Main}.

We say that the spectrum is \emph{$\beta$-sparse} if there is $c_*>0$ such that
\begin{equation}\label{eq:sparsity}
 |\ell_1-\ell_2|\ge c_* \cdot e^{-\beta\max\{\ell_1,\ell_2\}} \quad \text{for all distinct }\ell_1, \ell_2 \in \Lyap(f).
\end{equation}

\begin{mainthm}[Rigidity under sparsity assumption]
\label{Thm:Rig}
\label{Thm:LLSR}
If $g\in\E_d^k$ has a $\beta$-sparse spectrum and
\begin{equation}
\label{eq:rigidity-budget}
 k> \max\left\{2, 1 + 2\beta\frac{\log\Omega}{\log\Lambda}\right\},
\end{equation}
then $g$ is $C^k$-locally spectrally rigid. Here, $1 < \Lambda < g'(x) < \Omega$ are the expansion bounds of $g$.
\end{mainthm}

\begin{remark}
\label{Rem:DDRO}
This theorem is an analog of the Main Theorem in an earlier version of this manuscript \cite{DK}. The main improvement is the explicit upper bound on $\beta$. After \cite{DK} was available on arXiv, DeWitt, Durham, Reber, and O'Hare \cite{DDRO} proved an exciting result saying that $C^k$-smooth negatively curved metrics whose length spectra satisfy an exponential sparsity condition are locally rigid, i.e., nearby length-isospectral metrics are isometric. In the preprint \cite{DK2}, we show that $C^k$ negatively curved metrics with an exponential sparsity condition are prevalent. Together with the result of \cite{DDRO}, it may imply that local length spectral rigidity of $C^k$-smooth negatively curved metrics is prevalent.
\end{remark}

\begin{mainthm}[Prevalence of sparse spectra]
\label{Thm:Prev}
Fix $1<\Lambda<d<\Omega$, an integer $k\ge3$, and an integer $m\ge1$. If
\begin{equation}
\label{eq:existence-budget}
 \beta>
 \left(8(m+1)+\frac{8k+14}{m}\right)\frac{\log\Omega}{\log\Lambda}+\frac{2\log d}{\log\Lambda},
\end{equation}
then the set $\mathcal U$ as in \eqref{eq:family-U} contains a prevalent subset of expanding circle maps with $\beta$-sparse spectra. 
\end{mainthm}

\begin{remark}
    The last two results are combined in the following regime. Theorem~\ref{Thm:Rig} gives $C^k$-local spectral rigidity for $C^k$-smooth maps with $\beta$-sparse spectra for $\beta=O(k)$. On the other hand, Theorem~\ref{Thm:Prev} with $m = \lceil\sqrt{k}\rceil$ gives a prevalent set of maps with $\beta$-sparsity for $\beta > O(\sqrt{k})$. Together, this will yield the conclusion of our Main Theorem~\ref{Thm:Prev:Main}. 
\end{remark}

\begin{remark}[On sparsity]
    Compare the sparsity condition with~\cite{DJ}. Furthermore, for any given $\alpha\in(0,1)$, McMullen~\cite{McMNote} constructed a $C^{1+\alpha}$ expanding circle map close to $L_d(x)=dx\pmod1$ with $\beta$-sparse spectrum for every prescribed $\beta>2\alpha+1$, with no equalities between different cycles. In $C^k$ regularity, the same approach gives the sufficient range $\beta>2k-1$ near the linear map. We improve on this construction in Appendix~\ref{App:Prevalence}, where we prove Theorem~\ref{Thm:Prev}.
\end{remark}

The normal distribution of the entries in the length spectrum $\Lyap_n(f)$, up to a small error of order $n^{-1/4}$, was recently established in~\cite{DFKLL}.

\begin{remark}[On the expansion ratio]
For $g\in\E_d^k$, the quantity $\log \max g'/\log \min g'\ge1$ is the \emph{log-nonlinearity} of $g$. If $\Lambda\le g'\le\Omega$, then $\Lyap_n(g)\subset[n\log\Lambda,n\log\Omega]$.
There are approximately $d^n/n$ primitive cycles of period $n$. When their lengths are distinct, the pigeonhole principle gives two entries at distance at most $Cn^2d^{-n}$. \rev{This counting constraint explains why the sufficient exponent cannot be arbitrarily small. }
\end{remark}

\begin{remark}[Beyond expanding circle maps]
We expect that the methods developed in this paper for establishing prevalent rigidity of expanding circle maps apply to analogous questions on spectral rigidity for Anosov maps on the torus, at least in dimension $2$, and the length spectral rigidity of negatively curved metrics.
\end{remark}

\subsection{Some history of the question}
%%%%%%%%%%%%%%%%%%%%%%%%%%%%%%%%%%%%%%%%%%%%%%%%%%%%%%%%%%%%%%%%%%%%%%%%
%%%%%%%%%%%%%%%%%%%%%%%%%%%%%%%%%%%%%%%%%%%%%%%%%%%%%%%%%%%%%%%%%%%%%%%%

With the exception of certain finite-dimensional families of maps (such as iterative rational maps on the Riemann sphere; see \cite{McMRigidity, China} and references therein), the problem of recovering the smooth conjugacy class of maps from their unmarked multiplier data (often referred to as the length spectrum), and thereby establishing \emph{spectral rigidity}, remains wide open. The current paper makes the first contribution in this direction.

In contrast, a substantial body of literature explores this question within the marked context, i.e., when the periodic data comes with a marking via topological conjugacy or symbolic dynamics. The question in this context is to promote the topological conjugacy to a smooth conjugacy. For expanding maps, the spectral rigidity has been established in various settings: by Shub and Sullivan \cite{SS} for one-dimensional expanding circle maps, by Martens and de Melo \cite{MdM} for unimodal maps, and most recently, by Gogolev and Rodriguez Hertz \cite{GRH} for higher-dimensional expanding maps on closed manifolds. 
The analogous question in the setting of Anosov diffeomorphisms was completely resolved by de la Llave, Marco and Mariy\'on in dimension $2$ \cite{dL1, dL2, dLM}. See also \cite{dL3, Go, KS} and references therein for partial progress for higher-dimensional Anosov diffeomorphisms.

The marked length spectral rigidity for negatively curved surfaces (i.e., recovering the metric by the set of lengths of closed geodesics marked by their homotopy types) was proven in dimension $2$ by Otal \cite{Otal} and Croke \cite{Croke} (see also the lecture note of Wilkinson \cite{Wi} and their recent beautiful book \cite{Wi2}, as well as the most recent progress on the same question for Anosov metrics due to Guillarmou, Lefeuvre, and Paternain \cite{GLP}). See also \cite{Butt1, Butt2} for the results on quantitative marked length spectral rigidity (refining \cite{GLP} and the works of Hamenst\"adt \cite{Ham} and Besson--Courtois--Gallot \cite{BCG} on the marked length spectral rigidity for locally symmetric spaces). For manifolds of dimension $\ge 3$, without further assumptions on the metric (e.g., locally symmetric, conformal) only a perturbative marked length spectral rigidity is known, see the seminal work of Guillarmou and Lefeuvre \cite{GuLe} and references therein.

Techniques towards local length spectral rigidity for hyperbolic metrics can be found in the work of Sharafutdinov \cite{Sharafutdinov}.

Most spectral rigidity results rely on finite \Li\!\!-type theorems (see \cite{Livsic, DynBible} for the classical \Li theorem). The goal of the finite \Li theory is to find solutions to cohomological equations using only finite periodic data; such equations appear naturally when one tries to linearize the conjugacy relation between maps or flows. We refer to the works \cite{GoLe, GuLe, Ka, ZW} on finite \Li theorems. 

Finally, we mention recent results of O'Hare \cite{OH1, OH2} on finite periodic data rigidity for expanding circle maps and area-preserving Anosov maps of the two-torus using methods of thermodynamic formalism.

A very recent progress on unmarked length spectral rigidity for negatively curved metrics under a sparsity assumption was made in~\cite{DDRO}; see Remark~\ref{Rem:DDRO} above. See also \cite{Humbert} for a fresh result on unmarked length spectral rigidity for metrics of \emph{constant} negative curvature in dimension two.

%%%%%%%%%%%%%%%%%%%%%%%%%%%%%%%%%%%%%%%%%%%%%%%%%%%%%%%%%%%%%%%%%%%%%%%%
%%%%%%%%%%%%%%%%%%%%%%%%%%%%%%%%%%%%%%%%%%%%%%%%%%%%%%%%%%%%%%%%%%%%%%%%

\subsection{Strategy of the proof of Main Theorem~\ref{Thm:LLSR}}
The proof is done via a KAM-type iterative scheme. We fix $g$ and successively adjust $f$ by smooth changes of coordinates, so that the adjusted maps converge to $g$. Each step uses only finitely many lengths from the common unmarked spectrum. The main point is to obtain an improvement in the $C^1$ distance while controlling the cost in the available $C^k$ norm. The proof consists of the following steps:

\begin{enumerate}
\item 
\label{Step1}
\emph{Recover finitely many markings from the unmarked spectrum.}
Let $\phi$ be a topological conjugacy (the \emph{marking}) between $f$ and $g$ chosen as $\phi \circ g=f \circ \phi$, $\phi(0) = 0$. If $\eps :=\|f-g\|_{C^1}$ is small, then 
\[
 |\lambda_n^f(\phi x)-\lambda_n^g(x)|\le Cn\eps, \quad x \in \cP_n^g.
\]
Both values belong to the same unmarked spectrum. The $\beta$-sparsity forces the equality $\lambda_n^f(\phi x) =\lambda_n^g(x)$ (the \emph{recovery of the marking}) for $n \le N$ with
\[
N = \left\lfloor\frac{(1-\tau)|\log \eps|}{\beta \log \Omega}\right\rfloor,
\]
where $\tau \in (0,1)$ is a small parameter, $1 < \Lambda < g' < \Omega$ are some uniform bounds over a neighborhood of $g$.
\smallskip

\item 
\emph{Use messengers to establish a quantitative \Li theorem.} 
We introduce a class of connecting periodic orbits which we call \emph{messengers}. 
For an observable $D \in C^k(\Circle, \R)$, the two inverse branches of $g$ fixing $0$ define one linear \emph{barrier} $\mathcal B D$ (see \eqref{eq:uD} for the definition). The application of Rolle's lemma and messengers leads to the following \emph{quantitative \Li theorem with exponential bound:}
\begin{equation}
\label{Eq:Prelim:Li}
\|D-(\mathcal B D)\circ g+\mathcal B D\|_{C^0}
 \le C\bigl(\epsilon_n(D)+\Lambda^{-(k-1)n}\|D\|_{C^{k-1}}\bigr),
\end{equation}
where $\epsilon_n(D)$ is the absolute value of the sum of $D$ along messenger periodic orbits of periods $2n$ and $2n+1$, and $\Lambda>1$ is the lower bound on the expansion $\min g'$. 
\smallskip

\item 
\emph{Apply the Whitney extension theorem (mollification).}
The dynamically natural observable is $D=\log f'\circ\phi-\log g'$, which is a $\theta$-\Ho function for some exponent $\theta \in (0,1)$. By Step~\ref{Step1}, the observable $D$ sums to zero along $g$-periodic orbits of period up to $N$. To apply~\eqref{Eq:Prelim:Li}, we mollify $D$ using the \emph{Whitney extension theorem}. Write
\[
 D=\log f'\circ\phi-\log g'
   =\log f'-\log g'+R,\qquad
 R=\log f'\circ\phi-\log f'.
\]
The difference $\log f'-\log g'$ is $C^{k-1}$, while $R$ is $\theta$-\Ho\!\!. We apply Whitney extension to keep the values of $R$ on a \emph{uniform} grid of size $\eps^\tau$ (\emph{Whitney scale}) and mollify in the complement to obtain $\widetilde R \in C^{k-1}(\Circle)$. Then for the $C^{k-1}$-smooth observable
\[
\widetilde D :=\log f'-\log g'+\widetilde R
\]
one can show that
\[
 \|D-\widetilde D\|_{C^0}\le C\eps^{1+\tau\theta^2},
 \quad
 \|\widetilde D\|_{C^{k-1}}\le C\eps^{-\tau(k-1)}.
\]
Since the sums of $D$ vanish along periodic orbits of period up to $N$, we get
\[
 \epsilon_n(\widetilde D)  \le N\|D-\widetilde D\|_{C^0}, \qquad 2n+1\le N.
\]
\smallskip

\item 
\label{Step4}
\emph{Balance the exponential gain and control the adjustment via interpolation.}
At each step, we keep the two bounds
\[
 \|f-g\|_{C^1}\le\eps,
 \qquad
 \|f\|_{C^k}\le\eps^{-\tau(k-1)}.
\]
Put $u:=\mathcal B\widetilde D$ and $n:=\lfloor(N-1)/2\rfloor$. Applying \eqref{Eq:Prelim:Li} to $\widetilde D$, and adding the regularization error, after a suitable choice of parameters (see Section~\ref{SSec:Choice} for the choices made) we get an exponent $\sigma$ so that
\begin{equation}
\label{Eq:LCor}
  \|D-u\circ g+u\|_{C^0}  \le C|\log\eps|\eps^{1+2\sigma}.
\end{equation}
The assumption~\eqref{eq:rigidity-budget} of Main Theorem~\ref{Thm:LLSR} makes it possible to choose $\sigma>0$. 

The low-norm estimates for $\widetilde D$, translated to the barrier $u = \mathcal B \widetilde D$, give
\[
 \|u\|_{C^0}
 \le C\eps^{1-\theta^2},
 \qquad
 \|u\|_{C^{k-1}}\le C\eps^{-\tau(k-1)}.
\]
Using an interpolation inequality between these two bounds yields, for a specific $\kappa > 0$ (again, see Section~\ref{SSec:Choice}),
\[
 \|u\|_{C^1}
  \le C\eps^\kappa.
\]
\smallskip

\item \emph{Turn the coboundary into a change of coordinates.}
In one inductive step, we start with a pair $f,g$ satisfying
\begin{equation}
\label{Eq:Initial}
\|f - g\|_{C^1} \le \eps, \qquad
 \|f\|_{C^k}
 \le\eps^{-\tau(k-1)},
\end{equation}
for some choices of positive $\tau$, $\sigma$, and $\eps$ sufficiently small. Executing the steps above gives $u \in C^{k-1}(\Circle, \R)$ with estimates as in Step~\ref{Step4}. Set $h(x):={\int_0^x e^{u(t)}\,dt}/{\int_0^1 e^{u(t)}\,dt}$. We show that 
\[
 \|h^{-1}fh-g\|_{C^1}
 \le C\|D-u\circ g+u\|_{C^0} \le \eps^{1+\sigma},
\]
where the last inequality is achieved using \eqref{Eq:LCor} by starting the iteration with $\eps$ sufficiently small. Put $f_+ :=h^{-1}fh$ and $\eps_+ :=\eps^{1+\sigma}$. We obtain
\[
\|f_+ - g\|_{C^1} \le \eps_+, \qquad
 \|f_+\|_{C^k}
 \le\eps_+^{-\tau(k-1)}.
\]
Hence the same two bounds as in \eqref{Eq:Initial} hold for the new pair $f_+,g$, with the smaller error $\eps_+$.

\item \emph{Iterate and recover the remaining derivatives.}
Starting with $f_0=f$, repeat the adjustment:
\[
 f_{j+1}=h_j^{-1}f_jh_j,
 \qquad
 \eps_{j+1}=\eps_j^{1+\sigma},
 \qquad
 \|f_j-g\|_{C^1}\le\eps_j.
\]
The error bounds are $\eps_j=\eps_0^{(1+\sigma)^j}$. In particular, we show that $\|h_j - \id\|_{C^2}$ is uniformly bounded and summable. This makes the products
$\phi_j:=h_0\circ\cdots\circ h_{j-1}$ converge in $C^1$ to a diffeomorphism $\phi$. This limit is a $C^1$ marking, because
\[
 \phi_j\circ f_j=f\circ\phi_j,
 \qquad f_j\longrightarrow g
 \quad\Longrightarrow\quad
 \phi\circ g=f\circ\phi.
\]
The result of Shub--Sullivan~\cite{SS} now gives $\phi\in C^k(\Circle)$.

\medskip

An unorthodox feature of this proof is that the low norm is kept small, while the top norm slowly blows up as $\|f_j\|_{C^k} \le \eps_j^{-\tau(k-1)}$ during the iterative process.

The proof introduces a Liv\v sic--Whitney approximation scheme for hyperbolic systems. The approximation scheme proposed above consists of two key ingredients: messenger-induced finite \Li theorem, and properly chosen Whitney mollification. This approximation scheme is reminiscent of a KAM scheme. In a KAM scheme, at each step one solves a cohomological equation. In our scheme, we apply the aforementioned finite \Li theorem. Then both schemes apply a mollification.

\end{enumerate}

\subsection{The structure of the paper}
The paper is organized as follows. In Section~\ref{Sec:PrelimAnal}, we collect some preliminary facts from analysis, including Rolle's theorem, interpolation inequalities, and some useful estimates. Section~\ref{Sec:PrelimExp} recalls the symbolic dynamics and recovers finitely many markings from exponential sparsity. Section~\ref{Sec:Whitney} gives the Whitney extension (mollification) used for the small nonsmooth part of the observable. In Section~\ref{Sec:QLivsic}, we prove the finite \Li theorem with exponentially small errors using ordinary and hybrid messengers (Figure~\ref{Fig:Hybrid}). Finally, in Section~\ref{Sec:Iteration} we combine all the ingredients together and give an inductive proof of Main Theorem~\ref{Thm:Rig}; we then deduce Main Theorem~\ref{Thm:Prev:Main} as a corollary to Main Theorems~\ref{Thm:Rig} and \ref{Thm:Prev}. Section~\ref{Sec:CounterExample} contains the proof of Theorem~\ref{Thm:Counter}; we present two mechanisms that can ``swap the marking" thus giving counterexamples to the general unmarked length spectral rigidity. The proof of Main Theorem~\ref{Thm:Prev} on the prevalence of expanding circle maps with exponential sparsity satisfying~\eqref{eq:existence-budget}, with all the necessary preliminary facts, is given in Appendix~\ref{App:Prevalence}. This proof is standalone. 

\medskip
\noindent
\textbf{Acknowledgments.} The authors would like to thank Andrey Gogolev for fruitful discussions at the early stages of this project. Also, we thank Sebastian van Strien for several very useful conversations that, in particular, led to the reflection example in Section~\ref{Sec:CounterExample}. Finally, the authors are grateful to Charles Fefferman, Curt McMullen, Ian Melbourne, Carlangelo Liverani, Omri Sarig, Julia Slipantschuk, and Thomas O'Hare,  who kindly answered some of our questions related to the project. We are particularly thankful to Curt McMullen, who shared with us his construction \cite{McMNote}. Finally, we are grateful to several anonymous referees for carefully reading an early version of the manuscript and providing many useful comments that prompted us to substantially improve the paper.

\hide{
====

We thank Jonathan DeWitt, Spencer Durham, James Reber, and Thomas O'Hare whose manuscript \cite{DDRO}, which appeared in Fall 2026, on unmarked spectral rigidity for negatively curved metrics prompted us to look at how to improve our results in the revision of the manuscript. Although we have also explored the idea that extra smoothness can overcome a potentially poor sparsity parameter, their paper is the first published account of this mechanism working. In particular, they explicitly state their finite \Li theorem with an error exponentially small in period \emph{and} in smoothness. Our finite \Li theorem (Theorem~\ref{thm:finite-livsic}) was inspired by their result (although the proof of Theorem~\ref{thm:finite-livsic} can be easily read out in the version of our manuscript from Fall 2025). 

====
}

\medskip
We are grateful to Jonathan DeWitt, Spencer Durham, James Reber, and Thomas O'Hare for their work on unmarked spectral rigidity for negatively curved metrics \cite{DDRO}, which appeared in the fall of 2026. For many years we have tried to show that prevalent expanding circle maps are locally rigid. The problem splits into estimating the sparsity exponent for prevalent maps, and showing that such a sparsity exponent implies local rigidity. \cite{DDRO} encouraged us to revisit the case when in the $C^k$ smothness $k \to \infty$,   

Their finite \Li theorem (Theorem B, first estimate) gives the first finite \Li with an error term that decays exponentially in both the period \emph{and linearly in the degree of smoothness}. In fall 2025, we suggested a version of a \Li theorem with an exponentially small error of order $O_n = \Lambda^{-n}$, where $n$ is the period. Theorem B in \cite{DDRO} inspired us to apply Rolle's lemma to (5.24) (in the numbering of the fall 2025 version of the manuscript) $k$ times to obtain our current \Li theorem (Theorem~\ref{thm:finite-livsic} in the current version). 

It appears that a prevalent set of expanding circle maps has the sparsity parameter $\beta \ge C_{spar} \cdot \sqrt{k}$ (Theorem~\ref{Thm:Prev}), and for maps with $\beta \le C_{rig} \cdot k$ we have local rigidity (Theorem~\ref{Thm:LLSR}); both statements hold for large $k$. Combining these two theorems, we obtain the prevalence of local rigidity (Theorem~\ref{Thm:Prev:Main}). For a reasonable choice of parameters (see Remark~\ref{Rem:Example}), the theorem works for $k = 10000$.

\medskip
\noindent
\textbf{Use of AI tools.} ChatGPT (Astra 6) was used during manuscript revision to assist with exposition, symbolic calculations, and the checking and refinement of selected proof details. In particular, in Appendix~\ref{App:Prevalence}, AI suggested using, in our terminology, the concept of ``extreme neighbours", to formulate the non-clustering Lemma A.1. This lemma is a preparation for proving that sparsity is prevalent. In proving it, AI proposed an implementation of the methods of \cite{KH} to prove this prevalence. The authors take responsibility for the mathematical content and the final text.

%%%%%%%%%%%%%%%%%%%%%%%%%%%%%%%%%%%%%%%%
%%%%%%%%%%%%%%%%%%%%%%%%%%%%%%%%%%%%%%

%%%%%%%%%%%%%%%%%%%%%%%%%%%%%%%%%%%%%%%%%
%%%%%%%%%%%%%%%%%%%%%%%%%%%%%%%%%%%%%%%%%

\section{Preliminaries from analysis}
\label{Sec:PrelimAnal}

%%%%%%%%%%%%%%%%%%%%%%%%%%%%%%%%%%%%%%%%%
\subsection{Some preliminaries from calculus}
%%%%%%%%%%%%%%%%%%%%%%%%%%%%%%%%%%%%%%%%%

For an integer $m\ge0$, the space $C^m(\Circle)$ consists of real-valued functions on the circle $\Circle$ with $m$ continuous derivatives, equipped with the norm
\[
 \|v\|_{C^m}  :=\max_{j \in \{0, \ldots, m\}} \max_{x \in \Circle}\left|v^{(j)}(x)\right|.
\]
For $0<\alpha\le1$, we define
\[
 [v]_\alpha  :=\sup_{\substack{x,y\in\Circle\\x\ne y}}
       \frac{|v(x)-v(y)|}
            {\operatorname{dist}(x,y)^\alpha},
 \qquad
 \|v\|_{C^{m,\alpha}}
 :=\max\bigl\{\|v\|_{C^m},[v^{(m)}]_\alpha\bigr\}.
\]
The space $C^{m,\alpha}(\Circle)$ consists of the functions in $C^m(\Circle)$ for which this norm is finite. We write $C^{m,0}=C^m$ and, for $0<\alpha<1$, $C^{m+\alpha}=C^{m,\alpha}$. In particular, $C^{m,1}$ means that the $m$th derivative is Lipschitz. We use $\|v\|_{\Lip}:=[v]_1$ for the optimal Lipschitz constant. These spaces, equipped with their respective norms, are Banach spaces.

For circle-valued maps, derivatives are computed in angular coordinates. Norms are computed on $[0,1]$ using fixed lifts, and differences of maps of the same degree are computed using compatible lifts. Such differences are real-valued periodic functions. For normalized maps, the lifts are chosen to vanish at $0$.

We will use the following interpolation inequality (see, e.g., \cite[Theorem~2.2.1]{Hamilton}): for real numbers $0<t<s$ (that can be integers), an integer $j \ge 0$, and $v\in C^{s+j}(\Circle)$,
\begin{equation}
\label{eq:interpolation}
 \|v\|_{C^{j+t}}  \le C_{j,s,t} \|v\|_{C^j}^{\,1-t/s} \|v\|_{C^{j+s}}^{\,t/s}
\end{equation}
(here and elsewhere, the subscript in the factor $C_{j,s,t}$ means that it depends on the orders $s$ and $t$).

The following lemma contains some useful estimates to control certain derivatives; it will be used repeatedly in our arguments below. These estimates are classical and follow by a direct computation; for completeness, we outline some arguments.

\begin{lemma}[Useful estimates]
\label{Lem:Computations}
Let $k\ge1$ be an integer and $M\ge1$. For $a,b \in C^k(\Circle)$ with $\|a\|_{C^1},\|b\|_{C^1}\le M$, composition satisfies
\begin{equation}
\label{Eq:Composition}
 \|a\circ b\|_{C^k}  \le C_{k,M} \bigl(\|a\|_{C^k}+\|b\|_{C^k}\bigr).
\end{equation}
If $h \colon \Circle \to \Circle$ is a $C^k$-smooth circle diffeomorphism satisfying $\|h-\id\|_{C^1}\le1/2$, then
\begin{equation}
  \label{Eq:Displ}
 \|h^{-1}-\id\|_{C^k} \le C_k\|h-\id\|_{C^k}. 
\end{equation}
Consequently, for every $C^k$-smooth expanding circle map $f$ with $\|f\|_{C^1}\le M$,
\begin{equation}
\label{eq:tame}
 \|h^{-1}\circ f\circ h\|_{C^k} \le C_{k,M} \bigl(\|f\|_{C^k}+\|h-\id\|_{C^k}\bigr).
\end{equation}
Finally, if $v \colon \Circle \to \R$ is a real-valued $C^k$-smooth function and $\|v\|_{C^0}\le M$, then
\begin{equation}
\label{Eq:Exp}
 \|e^v-1\|_{C^k}\le C_{k,M}\|v\|_{C^k}.   
\end{equation}
\end{lemma}

\begin{proof}
The composition estimate \eqref{Eq:Composition} follows from \cite[Lemma~2.3.4]{Hamilton}.

Below, we will use a version of \eqref{Eq:Composition} keeping track of the size of the outer function as follows. Apply \eqref{Eq:Composition} to $a/\|a\|_{C^1}$ and $b$. When $a\ne0$, this gives
\begin{equation}
\label{Eq:CompisitionMod}
 \|a\circ b\|_{C^k} \le C_{k,M} \bigl(\|a\|_{C^k} +\|a\|_{C^1}\|b\|_{C^k}\bigr).
\end{equation}
The case $a=0$ is immediate.

Let us now estimate the displacement of the inverse. Put $v:=h-\id$\,. Since $\|v\|_{C^1}\le1/2$, we have $h'\ge1/2$, and $h^{-1}$ has a uniformly bounded $C^1$
norm. By \cite[Lemma~2.3.6]{Hamilton}, 
\[
\|h^{-1}\|_{C^k}\le C_k(1+\|v\|_{C^k}).
\]
Using the identity $h^{-1}-\id=-v\circ h^{-1}$ and \eqref{Eq:CompisitionMod}, we obtain
\[
 \begin{aligned}
 \|h^{-1}-\id\|_{C^k}
 &\le C_k\bigl(
       \|v\|_{C^k}
       +\|v\|_{C^1}\|h^{-1}\|_{C^k}
      \bigr)\\
 &\le C_k\bigl(
       \|v\|_{C^k}
       +\|v\|_{C^1}(1+\|v\|_{C^k})
      \bigr)
 \le C_k\|v\|_{C^k},
 \end{aligned}
\]
which proves \eqref{Eq:Displ}. The last inequality uses both $\|v\|_{C^1}\le1/2$ and $\|v\|_{C^1}\le\|v\|_{C^k}$.

Applying \eqref{Eq:Composition} twice now gives
\[
 \|h^{-1}\circ f\circ h\|_{C^k}
 \le C_{k,M}\bigl(
       \|f\|_{C^k}+\|h\|_{C^k}+\|h^{-1}\|_{C^k}
      \bigr).
\]
By the triangle inequality and \eqref{Eq:Displ}, we get
\[
 \begin{aligned}
 \|h\|_{C^k}+\|h^{-1}\|_{C^k}
 &\le 2+\|h-\id\|_{C^k}+\|h^{-1}-\id\|_{C^k}\\
 &\le 2+C_k\|h-\id\|_{C^k}.
 \end{aligned}
\]
Substituting this into the preceding estimate and enlarging
the constant, we obtain
\[
 \|h^{-1}\circ f\circ h\|_{C^k}
 \le C_{k,M}\bigl(
       \|f\|_{C^k}+2+\|h-\id\|_{C^k}
      \bigr).
\]
To remove the additive constant, observe that the degree identity $\int_0^1 f'(x)\,dx=d$ implies $\|f\|_{C^k}\ge\|f'\|_{C^0}\ge d\ge2$. Thus $\|f\|_{C^k}+2\le2\|f\|_{C^k}$. After increasing $C_{k,M}$ once more, we obtain \eqref{eq:tame}.

Finally, we prove \eqref{Eq:Exp}. The uniform norm satisfies
\[
 \|e^v-1\|_{C^0}\le e^M\|v\|_{C^0}.
\]
For $1\le j\le k$, each term in the $j$th derivative of $e^v$ is a constant multiple of $e^v\prod_{i=1}^{p}v^{(j_i)}$, where $1 \le p \le j$, $j_i\ge1$ and $\sum_i j_i=j$. Applying \eqref{eq:interpolation} to each factor gives
\[
 \begin{aligned}
 \prod_{i=1}^{p}\|v^{(j_i)}\|_{C^0}
 &\le C_k
       \|v\|_{C^0}^{p-j/k}
       \|v\|_{C^k}^{j/k}\\
 &\le C_k M^{p-1}\|v\|_{C^k}.
 \end{aligned}
\]
Here we used $j/k\le1$ and $\|v\|_{C^0}\le\|v\|_{C^k}$. Since $e^v\le e^M$, this bounds every derivative and proves \eqref{Eq:Exp}.
\end{proof}

%%%%%%%%%%%%%%%%%%%%%%%%%%%%%%%%%%%%%%%%%%%%%%%
\subsection{Rolle's Theorem and its application}
%%%%%%%%%%%%%%%%%%%%%%%%%%%%%%%%%%%%%%%%%%%%%%%

We say that a continuous function on $\Circle$ has \emph{roots with density $\delta$}, where $0<\delta\le1$, if every closed arc of length $\delta$ contains a root of the function.

\begin{lemma}[Improvement by density]
\label{Lem:AlaRolle}
Let $0<\alpha\le1$ and $g\in C^{0,\alpha}(\Circle)$. Assume that $g$ has roots with density $\delta$. Then
\[
 \|g\|_{C^0}\le \delta^\alpha[g]_\alpha.
\]
In particular, if $g\in C^1(\Circle)$, then $\|g\|_{C^0}\le\delta\|g'\|_{C^0}$.
\end{lemma}

\begin{proof}
For a point $x\in\Circle$, choose a root $y$ of $g$
with $\operatorname{dist}(x,y)\le\delta$. Then
\[
 |g(x)|
 =|g(x)-g(y)|
 \le [g]_\alpha\operatorname{dist}(x,y)^\alpha
 \le \delta^\alpha[g]_\alpha.
\]
Taking the supremum over $x$ proves the first estimate. For $g\in C^1(\Circle)$, apply it with $\alpha=1$ and use $[g]_1\le\|g'\|_{C^0}$.
\end{proof}

Repeated applications of Rolle's theorem give a stronger estimate when more derivatives are available. Let $g\in C^{m,\alpha}(\Circle)$, where $m\ge1$ is an integer and $0<\alpha\le1$, and assume that $g$ has roots with density $\delta$. For $0\le j\le m$, inductively using Rolle's theorem, we get that the derivative $g^{(j)}$ has roots with density $(j+1)\delta$, provided $(j+1)\delta<1$.

Applying Lemma~\ref{Lem:AlaRolle} to the highest derivative therefore gives
\[
 \|g^{(m)}\|_{C^0}  \le (m+1)^\alpha\delta^\alpha[g^{(m)}]_\alpha.
\]
For each $0\le j<m$, the same argument gives
\[
 \|g^{(j)}\|_{C^0} \le (j+1)\delta\|g^{(j+1)}\|_{C^0}.
\]
Combining these inequalities, we obtain
\[
 \|g^{(j)}\|_{C^0}
 \le \frac{m!}{j!}(m+1)^\alpha
       \delta^{m-j+\alpha}[g^{(m)}]_\alpha
 \le \frac{(m+1)!}{j!}
       \delta^{m-j+\alpha}[g^{(m)}]_\alpha,
 \qquad 0\le j\le m.
\]
If $g\in C^{m+1}(\Circle)$, we also have $[g^{(m)}]_1\le\|g^{(m+1)}\|_{C^0}$, which gives the corresponding estimate in terms of the highest derivative.

In the proof below, we will need a version of this argument in which the values of the function are small on a dense finite set, but do not necessarily vanish.

\begin{lemma}[\`A la Rolle]
\label{Lem:AlaRolle2}
\label{lem:sampling}
Let $s>0$ and $g\in C^s(\Circle)$. Let $E\subset\Circle$ be a nonempty finite set such that every closed arc of length $\delta$, where $0<\delta\le1$, contains a point of $E$. Then
\begin{equation}
\label{eq:sampling}
 \|g\|_{C^0} \le C_s\left(\max_{x\in E}|g(x)|+\delta^s\|g\|_{C^s}\right),
\end{equation}
where $C_s$ depends only on $s$. 
\end{lemma}

\begin{proof}
Write $s=m+\alpha$, where $m\ge0$ is an integer and $0<\alpha\le1$. Put $A:=\max_{x\in E}|g(x)|$.

First suppose that $m=0$. The argument of Lemma~\ref{Lem:AlaRolle} applies: for each $x\in\Circle$, choose $y\in E$ with $\operatorname{dist}(x,y)\le\delta$. Then
\[
 |g(x)| \le |g(y)|+|g(x)-g(y)|\le A+\delta^\alpha[g]_\alpha.
\]
This proves the required estimate. When $s=1$, use $[g]_1\le\|g\|_{C^1}$.

Now suppose that $m\ge1$. We first treat $\delta\le(6\cdot2^m)^{-1}$. Fix $x\in\Circle$. We will select points of $E$ near $x$ that are separated from each other. In an angular coordinate system starting at $x$, consider the intervals
\[
 I_i=[x+3i\delta,\;x+(3i+1)\delta],
 \qquad 0\le i<2^m.
\]
They lie in an arc of length at most $1/2$. Each interval contains a point $x_i\in E$, and these points satisfy
\[
 x_{i+1}-x_i\ge2\delta, \qquad |x_i-x|\le3\cdot2^m\delta, \qquad |g(x_i)|\le A.
\]
Thus, although the original set $E$ may contain clusters, the selected points are separated by at least $2\delta$.

Apply the mean value theorem to the disjoint pairs $(x_0,x_1),(x_2,x_3),\ldots$. In each corresponding interval $(x_i, x_{i+1})$, we obtain a point $x^{(1)}_i$ at which
\[
 |g'(x^{(1)}_i)|  \le \frac{2A}{2\delta}  =A\delta^{-1}.
\]
The new points $x^{(1)}_0, x^{(1)}_1, \ldots$ remain separated by at least $2\delta$, because the intervals containing them are separated by at least $2\delta$.

We repeat this procedure, pairing the new points as $(x^{(1)}_0, x^{(1)}_1), (x^{(1)}_2, x^{(1)}_3), \ldots$. After $j$ steps, the points obtained lie in the disjoint intervals
\[
 [x_{2^j i},\;x_{2^j(i+1)-1}], \qquad 0\le i<2^{m-j}.
\]
Consecutive intervals are separated by at least $2\delta$. If, by induction, the values of $g^{(j)}$ at the selected points $x^{(j)}_0, x^{(j)}_1, \ldots$ are bounded by $A\delta^{-j}$ at step $j$, the next application of the mean value theorem gives points $x^{(j+1)}_0, x^{(j+1)}_1, \ldots$ with
\[
 |g^{(j+1)}(x_i^{(j+1)})| \le \frac{2A\delta^{-j}}{2\delta} =A\delta^{-(j+1)}.
\]
Consequently, for each $j\in\{0,\ldots,m\}$, there is a point $y_j$ such that
\begin{equation}
\label{Eq:ApproxRolle}
 |y_j-x|\le C_m\delta, \qquad |g^{(j)}(y_j)|\le A\delta^{-j}, \qquad C_m:=3\cdot2^m.
\end{equation}
We now use these points to estimate the derivatives at $x$. For the highest derivative,
\[
 \begin{aligned}
 |g^{(m)}(x)| &\le |g^{(m)}(y_m)|+|g^{(m)}(x)-g^{(m)}(y_m)|\\
 &\le A\delta^{-m} +(C_m\delta)^\alpha[g^{(m)}]_\alpha.
 \end{aligned}
\]
Since $x$ was arbitrary, this gives the same bound for $\|g^{(m)}\|_{C^0}$.

For $0\le j<m$, integrating $g^{(j+1)}$ between $y_j$ and $x$ and using \eqref{Eq:ApproxRolle}, we obtain
\[
 \begin{aligned}
 |g^{(j)}(x)| &\le |g^{(j)}(y_j)| +|x-y_j|\|g^{(j+1)}\|_{C^0}\\
 &\le A\delta^{-j}+C_m\delta\|g^{(j+1)}\|_{C^0}.
 \end{aligned}
\]
Taking the supremum over $x$ and substituting these bounds successively gives
\[
 \|g\|_{C^0} \le A\sum_{j=0}^{m}C_m^j+C_m^{m+\alpha}\delta^{m+\alpha}[g^{(m)}]_\alpha.
\]
This gives \eqref{eq:sampling} with $C_s := \max\left\{\sum_{j=0}^{m}C_m^j, C_m^{m+\alpha}\right\}$.

Finally, if $\delta>(6\cdot2^m)^{-1}$, then
\[
 \|g\|_{C^0} \le\|g\|_{C^s} \le(6\cdot2^m)^s\delta^s\|g\|_{C^s}.
\]
Increasing $C_s$ if necessary proves \eqref{eq:sampling} for all $0<\delta\le1$.
\end{proof}

%%%%%%%%%%%%%%%%%%%%%%%%%%%%%%%%%%%%%%%%%%%%%%%%%%%%%%%
\section{Preliminaries on expanding circle maps}
\label{Sec:PrelimExp}
%%%%%%%%%%%%%%%%%%%%%%%%%%%%%%%%%%%%%%%%%%%%%%%%%%%%%%%

Throughout this section, $d\ge2$ is fixed, and the maps $f,g\in\E_d^1$ satisfy the common expansion bounds
\begin{equation}
\label{eq:expansion}
 1 < \Lambda\le f'(x),g'(x)\le\Omega.
\end{equation}
These bounds are understood in all statements below; additional regularity assumptions will be specified separately. For a fixed reference map $g$, we choose $\Lambda<\min g'$ and $\Omega>\max g'$, so that the same bounds hold for all sufficiently close maps $f$. 

%%%%%%%%%%%%%%%%%%%%%%%%%%%%%%%%%%%%%%%%%%%%%%%%%%%%%%%
\subsection{Distortion estimates for expanding circle endomorphisms}
%%%%%%%%%%%%%%%%%%%%%%%%%%%%%%%%%%%%%%%%%%%%%%%%%%%%%%%

The following lemma is folklore. 

\begin{lemma}[Distortion with the image length]
\label{Lem:DistortionWithLength}
\label{lem:distortion}
Let $g\in\E_d^1$ have an $\alpha$-\Ho derivative, where $0<\alpha\le1$. Let $n\in\mathbb N$ and let $I\subset\Circle$ be an interval on which $g^n$ is injective, and define $J:=g^n(I)$ to be the image interval. Then
\begin{equation}
\label{Eq:DistortionWithLength}
 \frac{(g^n)'(y)}{(g^n)'(z)}
 \le\exp\left(\frac{[\log g']_\alpha}{\Lambda^\alpha-1}|J|^\alpha\right)
 \qquad\text{for every }y,z\in I.
\end{equation}
In particular, if $g'$ is Lipschitz, the right-hand side is bounded by $\exp\bigl(\|g'\|_{\Lip}|J|/(\Lambda-1)\bigr)$.
\end{lemma}

\begin{proof}
The proof is similar to that of \cite[Lemma~1]{SS}, but we present it for completeness. By expansion, for each $0\le j\le n$,
\[
 |g^j(I)|\le\Lambda^{-(n-j)}|J|.
\]
Indeed, $g^{n-j}$ maps $g^j(I)$ injectively onto $J$, and its derivative is at least $\Lambda^{n-j}$. Hence
\[
 \left|\log g'(g^jy)-\log g'(g^jz)\right|
 \le[\log g']_\alpha\Lambda^{-\alpha(n-j)}|J|^\alpha.
\]
The chain rule now gives
\[
 \begin{aligned}
 \left|\log\frac{(g^n)'(y)}{(g^n)'(z)}\right|
 &\le\sum_{j=0}^{n-1}\left|\log g'(g^jy)-\log g'(g^jz)\right|\\
 &\le[\log g']_\alpha|J|^\alpha
       \sum_{j=0}^{n-1}\Lambda^{-\alpha(n-j)}\\
 &\le\frac{[\log g']_\alpha}{\Lambda^\alpha-1}|J|^\alpha.
 \end{aligned}
\]
Exponentiating proves \eqref{Eq:DistortionWithLength}. In the Lipschitz case, since on an interval $[c, +\infty)$, the logarithm satisfies $|\log(x) - \log(y)| \le 1/c \cdot |x-y|$, we obtain $[\log g']_1\le\Lambda^{-1}\|g'\|_{\Lip}\le\|g'\|_{\Lip}$.
\end{proof}

We will also use the corresponding estimate for inverse branches. Cut the circle at $0$, and let $b\colon[0,1]\to I$ be a full inverse branch of $g^n$. Put $K:=[\log g']_\alpha/(\Lambda^\alpha-1)$. Since $b'(t)=1/(g^n)'(b(t))$, Lemma~\ref{lem:distortion} gives
\[
 e^{-K}\le\frac{b'(t)}{b'(s)}\le e^K.
\]
Integrating in $s$ and using $\int_0^1b'(s)\,ds=|I|$, we obtain $e^{-K}|I|\le b'(t)\le e^K|I|$. Thus the derivative of a full inverse branch is comparable to the length of its image, uniformly in the number of iterates.

%%%%%%%%%%%%%%%%%%%%%%%%%%%%%%%%%%%%%%%%%%%%%%%%%%%%%%%
\subsection{Combinatorics of periodic points}
\label{SSec:Comb}
\label{SSec:Prelim}
%%%%%%%%%%%%%%%%%%%%%%%%%%%%%%%%%%%%%%%%%%%%%%%%%%%%%%%

It is well known that every expanding circle map of degree $d$ is topologically conjugate to the linear map $L_d\colon x\mapsto dx\pmod1$; see \cite[p.~175]{Shub69}. For a given period $n$, the linear map has exactly $d^n-1$ periodic points of period $n$. They are equally spaced, and
\[
 \mathbb P_n^{L_d}
 =\left\{\frac{j}{d^n-1}:0\le j<d^n-1\right\}.
\]
Here, as in the introduction, a point of period $n$ means a point fixed by the $n$th iterate; its least period may be a proper divisor of $n$.

We recall the symbolic construction of the conjugacy, which will also be used later. Cut the circle at $0$ and identify the resulting interval with $[0,1]$. The preimages of $0$ under $g$ divide it into $d$ intervals. Each of these intervals is mapped by $g$ over $[0,1]$. In increasing order, denote the corresponding inverse branches by $b_0,\ldots,b_{d-1}\colon[0,1]\to[0,1]$. They clearly satisfy
\[
 \Omega^{-1}\le b_i'(x)\le\Lambda^{-1},  \qquad \forall x \in [0,1], \qquad b_0(0)=0,\qquad b_{d-1}(1)=1.
\]
For a word $\sigma=s_0\ldots s_{n-1}$ of length $n \ge 1$ in the alphabet $\{0, \ldots, d-1\}$, its \emph{cylinder} is the interval
\[
 I_\sigma:=b_{s_0}\circ\cdots\circ b_{s_{n-1}}([0,1]).
\]
We have 
\begin{equation}
\label{Eq:Cylinder}
\Omega^{-n}\le|I_\sigma|\le\Lambda^{-n}.
\end{equation}

There are $d^n$ words of length $n$, and the corresponding cylinders cover the circle and have disjoint interiors. Since their lengths tend to zero, nested cylinders determine the unique point with a given infinite code. Applying $g$ shifts the code one place to the left.

For a word $\sigma$, a bar denotes an infinite word with periodic repetition: $\overline{\sigma} :=\sigma\sigma\sigma\ldots$. The point with this code (\emph{symbolic address}) belongs to $\mathbb P_n^g$. Every periodic point other than $0$ has a unique symbolic address.

Matching equal codes for $g$ and $f$ gives a unique orientation-preserving homeomorphism $\phi \colon \Circle \to \Circle$ satisfying
\[
 \phi\circ g=f\circ\phi,\qquad\phi(0)=0.
\]
In particular, the sets $\mathbb P_n^g$ and $\mathbb P_n^f$ have natural cyclic orderings, and the orders of corresponding under $\phi$ periodic points agree. 

%%%%%%%%%%%%%%%%%%%%%%%%%%%%%%%%%%%%%%%%%%%%%%%%%%%%%%%
\subsection{Density and sparsity of $\mathbb P_n^g$}
%%%%%%%%%%%%%%%%%%%%%%%%%%%%%%%%%%%%%%%%%%%%%%%%%%%%%%%

The largest gap between consecutive points of $\mathbb P_n^g$ measures their density, while the smallest gap measures their separation. 

\begin{lemma}[Estimates on the gaps in $\mathbb P_n^g$]
\label{Lem:Grid}
\label{lem:mesh}
Let $g\in\E_d^1$. The set $\mathbb P_n^g$ contains exactly $d^n-1$ points, and every gap $I$ between consecutive points satisfies
\begin{equation}
\label{eq:mesh}
 \frac1{\Omega^n-1}\le|I|\le\frac1{\Lambda^n-1}.
\end{equation}
In particular, the largest gap is at most $(1-\Lambda^{-1})^{-1}\Lambda^{-n}$.
\end{lemma}

\begin{proof}
The circle map $x\mapsto g^n(x)-x \pmod1$ has degree $d^n-1$, and its derivative satisfies
\[
 \Lambda^n-1\le(g^n)'(x)-1\le\Omega^n-1.
\]
Thus it is an orientation-preserving covering with exactly $d^n-1$ preimages of $0$. These preimages are precisely the points of $\mathbb P_n^g$.

On a gap $I$ between consecutive preimages, this map makes exactly one full turn: another full turn would produce an additional preimage of $0$ inside $I$. Consequently,
\[
 1=\int_I\bigl((g^n)'(x)-1\bigr)\,dx,
 \qquad
 (\Lambda^n-1)|I|\le1\le(\Omega^n-1)|I|.
\]
This proves \eqref{eq:mesh}. 
\end{proof}

In Appendix~\ref{App:Prevalence}, we will also need to separate points belonging to different periods. Put
\[
 \mathbb Q_p^g:=\bigcup_{1\le n\le p}\mathbb P_n^g.
\]

\begin{lemma}[Separation of several periods]
\label{lem:marking}
Let $g\in\E_d^1$. There is $c>0$, depending only on $d$ and $\Omega$, such that distinct points $x,y\in\mathbb Q_p^g$ satisfy
\[
 \operatorname{dist}(x,y)\ge c\Omega^{-2p}.
\]
\end{lemma}

\begin{proof}
Let $H$ be the normalized conjugacy from $g$ to $L_d$. We first estimate $\operatorname{dist}(Hx,Hy)$ for arbitrary distinct $x,y\in\Circle$.

Let $I$ be a shortest arc joining $x$ and $y$, and put $t:=|I|=\operatorname{dist}(x,y)$. Choose $n\ge0$ with $\Omega^{-(n+1)}<t\le\Omega^{-n}$. By \eqref{Eq:Cylinder}, an interval of length at most $\Omega^{-n}$ meets at most three cylinders of word length $n$. Each image cylinder under $H$ has length $d^{-n}$. Therefore
\[
 \operatorname{dist}(Hx,Hy)\le|H(I)|\le3d^{-n}.
\]
Writing $\alpha:=\log d/\log\Omega$ in this proof, we obtain
\[
 \operatorname{dist}(Hx,Hy)\le3d\,\operatorname{dist}(x,y)^\alpha.
\]
Two distinct periodic points of $L_d$ with periods $n,m\le p$ have circle distance at least $1/((d^n-1)(d^m-1))$, and hence at least $d^{-2p}$. Consequently, for distinct $x,y\in\mathbb Q_p^g$,
\[
 d^{-2p}\le\operatorname{dist}(Hx,Hy)
 \le3d\,\operatorname{dist}(x,y)^\alpha.
\]
Rearranging gives $\operatorname{dist}(x,y) \ge(3d)^{-1/\alpha}d^{-2p/\alpha} =(3d)^{-1/\alpha}\Omega^{-2p}$, as required.
\end{proof}

%%%%%%%%%%%%%%%%%%%%%%%%%%%%%%%%%%%%%%%%%%%%%%%%%%%%%%%
\subsection{Length spectrum}
%%%%%%%%%%%%%%%%%%%%%%%%%%%%%%%%%%%%%%%%%%%%%%%%%%%%%%%

Let $g\in\E_d^k$. Recall that its length spectrum is
\[
 \Lyap(g)=\bigcup_{n\in\mathbb N}\Lyap_n(g),
 \qquad
 \Lyap_n(g)=\{\lambda_n^g(p):p\in\mathbb P_n^g\},
\]
where $\lambda_n^g(p)=\log(g^n)'(p)$ is the $\log$-multiplier at the periodic point $p \in \cP_n^g$.

\begin{lemma}[Basic properties of the length spectrum]
\label{Lem:SimpleProp}
For $g\in\E_d^k$, $k\ge1$, the following properties hold:
\begin{itemize}
\item
$\Lyap_n(g)$ is a nonempty subset of $[n\log\Lambda,n\log\Omega]$;
\item
$\Lyap(g)$ is a nonempty, unbounded subset of $\mathbb R_+$;
\item
$\Lyap_n(g)$, $n\in\mathbb N$, and $\Lyap(g)$ are invariant under $C^1$-smooth  conjugacies;
\item
if $k\ge2$, then $\Lyap_n(g)$ is a singleton for every $n$ if and only if $g$ is $C^k$-smoothly conjugate to $L_d$.
\end{itemize}
\end{lemma}

\begin{proof}
The first and the third items follow by the chain rule. The second item follows from the first. The fourth item is a well-known rigidity result \cite{SS}. 
\end{proof}

%%%%%%%%%%%%%%%%%%%%%%%%%%%%%%%%%%%%%%%%%%%%%%%%%%%%%%%
\subsection{Marked length spectrum and recovery of markings}
\label{SSec:Recovery}
%%%%%%%%%%%%%%%%%%%%%%%%%%%%%%%%%%%%%%%%%%%%%%%%%%%%%%%

Let $f,g\in\E_d^1$, and let $\phi$ be their normalized orientation-preserving conjugacy:
\[
 \phi\circ g=f\circ\phi,\qquad\phi(0)=0.
\]
As explained in Section~\ref{SSec:Comb}, this conjugacy respects symbolic dynamics. We say that $p^f\in\mathbb P_n^f$ and $p^g\in\mathbb P_n^g$ are \emph{corresponding periodic points} if $p^f=\phi(p^g)$. This correspondence provides a marking of periodic orbits, and we call $\phi$ the \emph{marking conjugacy}. The correspondence is a bijection of periodic orbits and preserves least periods. 

\begin{prop}[Closeness of the conjugacy to identity]
\label{Prop:Closeness}
Let $f,g\in\E_d^1$, and let $\phi$ be their marking conjugacy. Then
\[
 \|\phi-\id\|_{C^0}\le\frac{\|f-g\|_{C^0}}{\Lambda-1}.
\]
In particular, the same bound holds for the distance between each pair of corresponding periodic points.
\end{prop}

\begin{proof}
Put $v:=\phi-\id$\!\!. Thus $f(x+v(x))=g(x)+v(g(x))$. By expansion,
\[
 \begin{aligned}
 \Lambda|v(x)|
 &\le|f(x+v(x))-f(x)|\\
 &=|v(g(x))+g(x)-f(x)|\\
 &\le\|v\|_{C^0}+\|f-g\|_{C^0}.
 \end{aligned}
\]
Taking the supremum and subtracting $\|v\|_{C^0}$ gives the required estimate.
\end{proof}

The maps $f$ and $g$ have the same \emph{marked length spectra} if $\lambda_n^f(p^f)=\lambda_n^g(p^g)$ for every $n$ and every pair $p^f, p^g$ of corresponding periodic points. Similarly, they have the same marked length spectra \emph{up to period $N$} if this equality holds for all corresponding periodic points of periods $1\le n\le N$.

We now explain how to recover markings up to a finite period from the unmarked spectrum and the sparsity assumption.

\begin{lemma}[Recovering partial marking]
\label{Lem:Recovering}
\label{lem:finite-marking}
Let $g\in\E_d^2$ have a $\beta$-sparse spectrum, and let $f\in\E_d^1$ satisfy
\[
 \Lyap(f)=\Lyap(g),
 \qquad \|f-g\|_{C^1}\le\eps.
\]
For every fixed $0<\tau<1$, the marked length spectra agree up to period
\[
 N=\left\lfloor\frac{(1-\tau)|\log \eps |}{\beta\log\Omega}\right\rfloor
\]
when $\eps>0$ is sufficiently small. The smallness threshold depends only on $\Lambda,\Omega,\|g''\|_{C^0},\beta,\tau$ and the sparsity constant $c_*$ in \eqref{eq:sparsity}.
\end{lemma}

\begin{proof}
We first estimate the difference between corresponding $\log$-multipliers. Put $D:=\log f'\circ\phi-\log g'$. By Proposition~\ref{Prop:Closeness},
\[
 \begin{aligned}
 |D(x)|
 &\le|\log f'(\phi(x))-\log g'(\phi (x))|+|\log g'(\phi(x))-\log g'(x)|\\
 &\le\frac{\|f'-g'\|_{C^0}}{\Lambda}
      +\frac{\|g''\|_{C^0}}{\Lambda}\|\phi - \id\|_{C^0}\le\frac1\Lambda
      \left(1+\frac{\|g''\|_{C^0}}{\Lambda-1}\right)\eps.
 \end{aligned}
\]
Thus $\|D\|_{C^0}\le C\eps$, with $C$ independent of the period and of higher derivatives of $g$.

For $x\in\mathbb P_n^g$, the conjugacy gives $f^j(\phi (x))=\phi(g^j(x))$. Therefore
\[
 \lambda_n^f(\phi(x))-\lambda_n^g(x)=\sum_{j=0}^{n-1}D(g^j(x)), \qquad
 |\lambda_n^f(\phi(x))-\lambda_n^g(x)|\le Cn\eps.
\]

Both lengths belong to the common spectrum $\Lyap(g)$. By Lemma~\ref{Lem:SimpleProp}, both are at most $n\log\Omega$. If they are distinct, $\beta$-sparsity implies
\[
 \begin{aligned}
 |\lambda_n^f(\phi(x))-\lambda_n^g(x)|
 &\ge c_*\exp\bigl(-\beta\max\{\lambda_n^f(\phi x),\lambda_n^g(x)\}\bigr)\\
 &\ge c_*\Omega^{-\beta n}.
 \end{aligned}
\]
For $n\le N$, the definition of $N$ gives $\Omega^{-\beta n}\ge\eps^{1-\tau}$. Hence distinct corresponding lengths would satisfy $c_*\eps^{1-\tau}\le Cn\eps\le CN\eps$. This is impossible for sufficiently small $\eps$. Thus the corresponding $\log$-multipliers agree for every $1\le n\le N$.
\end{proof}

%%%%%%%%%%%%%%%%%%%%%%%%%%%%%%%%%%%%%%%%%
\subsection{Invariant densities and transfer operators}
\label{SSec:Transfer}
%%%%%%%%%%%%%%%%%%%%%%%%%%%%%%%%%%%%%%%%%

In this subsection, we collect the transfer operator estimates used in the proof of Subsection~\ref{sec:stability}. For background on this method, see, e.g., \cite{Sa, Ruelle89}.

Let $0<\theta<1$ and $g\in\E_d^{1+\theta}$. For a bounded Borel measurable function $v\colon{\Circle}\to\mathbb R$, the \emph{transfer operator} associated with $g$ is defined by
\[
 (\mathcal P v)(x):=\sum_{g(y)=x}\frac{v(y)}{g'(y)}.
\]
This operator is positive and preserves integrals: $\int_{{\Circle}}\mathcal P v(x)\,dx=\int_{{\Circle}}v(x)\,dx$.

By the Ruelle--Perron--Frobenius theorem \cite{Ruelle89}, there is a strictly positive density $\rho\in C^\theta({\Circle})$ satisfying $\mathcal P\rho=\rho$ and $\int_{{\Circle}}\rho(x)\,dx=1$. Moreover, there are constants $C>0$ and $0<q<1$ such that
\begin{equation}
\label{eq:cone-mixing}
 \left\|\mathcal P^n v-\rho\int_{{\Circle}}v(x)\,dx\right\|_{C^0}
 \le Cq^n\|v\|_{C^\theta},
 \qquad v\in C^\theta({\Circle}),\quad n\ge0.
\end{equation}
If $[\log g']_\theta\le M$, the constants can be chosen depending only on the common expansion bounds, $d$, $\theta$, and $M$. This uniformity follows from the real cone contraction estimate in \cite[Section~9]{Rugh10}. We apply that estimate with the metric $\operatorname{dist}(x,y)^\theta$ and the potential $-\log g'$. In this metric, inverse branches contract by $\Lambda^{-\theta}$, and the potential has Lipschitz constant at most $M$. Thus the contraction estimates are uniform under our assumptions.

We will need a consequence of \eqref{eq:cone-mixing} for characteristic functions of intervals.

\begin{lemma}[Transfer estimate]
\label{Lem:Transfer}
Assume $[\log g']_\theta\le M$. There are constants $C>0$ and $0<\gamma<1$, depending only on the common expansion bounds, $d$, $\theta$, and $M$, such that every interval $I\subset{\Circle}$ satisfies
\begin{equation}
\label{eq:indicator-mixing}
 \left\|\mathcal P^n(\mathbf 1_I-|I|)\right\|_{C^0}
 \le C\gamma^n,
 \qquad \forall n\ge0.
\end{equation}
Here $\mathbf 1_I$ denotes the indicator function of $I$, and $|I|$ is its length.
\end{lemma}

\begin{proof}
We can assume $n \ge 1$. We start by estimating $(\mathcal P^n\mathbf 1_J)(x)$ for an interval $J\subset{\Circle}$. Let $I_\sigma$ be an $n$-cylinder. We take the cylinders half-open, so that each point in $g^{-n}(x)$ belongs to exactly one cylinder. Since $g^n$ maps $I_\sigma$ once around the circle, Lemma~\ref{Lem:DistortionWithLength} gives, for every $y\in I_\sigma$,
\[
 1=\int_{I_\sigma}(g^n)'(z)\,dz
 \le C(g^n)'(y)|I_\sigma|.
\]
Thus $1/(g^n)'(y)\le C|I_\sigma|$. Moreover, by \eqref{Eq:Cylinder}, $|I_\sigma|\le\Lambda^{-n}$.

There is one preimage of $x$ in each $n$-cylinder. All but at most two of the cylinders meeting $J$ are contained in $J$. Summing their contributions, we therefore obtain
\[
 (\mathcal P^n\mathbf 1_J)(x)
 =\sum_{\substack{g^n(y)=x\\y\in J}}\frac{1}{(g^n)'(y)}
 \le C\sum_{\sigma:\,I_\sigma\cap J\ne\varnothing}|I_\sigma|
 \le C\bigl(|J|+2\Lambda^{-n}\bigr).
\]
After increasing $C$, this proves
\begin{equation}
\label{eq:interval-weight}
 \|\mathcal P^n\mathbf 1_J\|_{C^0}
 \le C\bigl(|J|+\Lambda^{-n}\bigr).
\end{equation}

We now approximate $\mathbf 1_I$ by a smooth function, so that we can apply \eqref{eq:cone-mixing}. Fix a nonnegative function $\chi\in C_c^\infty((-1,1))$ with integral one. For $0<\eta<1/4$, define $\chi_\eta(t):=\eta^{-1}\chi(t/\eta)$ on $[-1/2,1/2]$ and extend it $1$-periodically. This defines a nonnegative smooth function on ${\Circle}$, supported within distance $\eta$ of $0$, with integral one. Put
\[
 v_\eta(x):=\int_{{\Circle}}\chi_\eta(x-y)\mathbf 1_I(y)\,dy.
\]
Since the kernel is nonnegative and has integral one, we have $0\le v_\eta\le1$. Integration also gives
\[
 \int_{{\Circle}}v_\eta(x)\,dx
 =\int_I\left(\int_{{\Circle}}\chi_\eta(x-y)\,dx\right)dy
 =|I|.
\]
Finally, let us estimate the H\"older norm of $v_\eta$. Differentiating the kernel gives
\[
 \|v_\eta'\|_{C^0}
 \le\|\chi_\eta'\|_{L^1}
 \le C\eta^{-1}.
\]
Together with $0\le v_\eta\le1$, this implies
\[
 |v_\eta(x)-v_\eta(y)|
 \le\min\bigl\{1,C\eta^{-1}|x-y|\bigr\}
 \le C\eta^{-\theta}|x-y|^\theta.
\]
Consequently, $\|v_\eta-|I|\|_{C^\theta}\le C\eta^{-\theta}$. Since this function has integral zero, \eqref{eq:cone-mixing} yields
\begin{equation}
\label{Eq:Cor1}
    \|\mathcal P^n(v_\eta-|I|)\|_{C^0}
 \le Cq^n\eta^{-\theta}.
\end{equation}

It remains to estimate the error introduced by smoothing. Denote by $J_\eta$ the union of the $\eta$-neighborhoods of the two endpoints of $I$. Outside $J_\eta$, the function $\mathbf 1_I$ is constant on the support of the kernel centered at $x$, and hence $v_\eta(x)=\mathbf 1_I(x)$. Since both functions take values in $[0,1]$, we obtain $|\mathbf 1_I-v_\eta|\le\mathbf 1_{J_\eta}$. Positivity of $\mathcal P$ therefore gives
\[
 \left|\mathcal P^n(\mathbf 1_I-v_\eta)\right|
 \le\mathcal P^n|\mathbf 1_I-v_\eta|
 \le\mathcal P^n\mathbf 1_{J_\eta}.
\]
Each endpoint neighborhood has length $2\eta$. Applying \eqref{eq:interval-weight} to these two intervals, we conclude that
\begin{equation}
 \label{Eq:Cor2}
 \left\|\mathcal P^n(\mathbf 1_I-v_\eta)\right\|_{C^0}
 \le C\bigl(\eta+\Lambda^{-n}\bigr).
\end{equation}
Finally, note that $\mathbf 1_I-|I|=(\mathbf 1_I-v_\eta)+(v_\eta-|I|)$. Combining \eqref{Eq:Cor1} and \eqref{Eq:Cor2}, we get
\[
 \left\|\mathcal P^n(\mathbf 1_I-|I|)\right\|_{C^0}
 \le C\bigl(\eta+\Lambda^{-n}+q^n\eta^{-\theta}\bigr).
\]
Choosing $\eta=\frac18q^{n/(1+\theta)}$, we get \eqref{eq:indicator-mixing} with $\gamma=\max\{\Lambda^{-1},q^{1/(1+\theta)}\}<1$.
\end{proof}

%%%%%%%%%%%%%%%%%%%%%%%%%%%%%%%%%%%%%%%%%%%%%%%%%%%%%%%
\subsection{Starting H\"older conjugacy}
%%%%%%%%%%%%%%%%%%%%%%%%%%%%%%%%%%%%%%%%%%%%%%%%%%%%%%%

We will need a quantitative H\"older estimate for the displacement of the marking conjugacy. 

\begin{lemma}[Starting H\"older conjugacy]
\label{Lem:Holder}
\label{lem:holder-marking}
Let $f\in\E_d^1$, $g\in\E_d^2$, and let $\phi$ be their marking conjugacy. If $0<\theta<1$ and $\Omega^\theta<\Lambda$, then $\phi$ is $\theta$-\Ho and
\begin{equation}
\label{eq:marking-small}
 \|\phi-\id\|_{C^\theta}\le C_\theta\|f-g\|_{C^1},
\end{equation}
where $C_\theta$ depends only on $\Lambda,\Omega,\theta$, and $\|g''\|_{C^0}$.
\end{lemma}

\begin{proof}
As above, put $v:=\phi-\id$ and $\eps:=\|f-g\|_{C^1}$. Proposition~\ref{Prop:Closeness} gives
\[
 \|v\|_{C^0}\le\frac{\eps}{\Lambda-1}.
\]
Let us estimate $v(x)-v(y)$ for $x, y \in \Circle$; it suffices to assume that $x$ and $y$ are close so that $\dist(x,y) = |x-y|$. Subtracting the identities $f(x+v(x))=g(x)+v(g(x))$ and $f(y+v(y))=g(y)+v(g(y))$, we obtain
\[
 f(x+v(x))-f(x+v(y))=v(g(x))-v(g(y)) +\int_y^x\bigl(g'(t)-f'(t+v(y))\bigr)\,dt.
\]
To estimate the integrand, we write:
\[
 \begin{aligned}
 |g'(t)-f'(t+v(y))|
 &\le|g'(t)-g'(t+v(y))|+|(g'-f')(t+v(y))|\\
 &\le\|g''\|_{C^0}\|v\|_{C^0}+\eps\\
 &\le\left(1+\frac{\|g''\|_{C^0}}{\Lambda-1}\right)\eps.
 \end{aligned}
\]
Using the lower bound for $f'$, we therefore get
\[
 \Lambda|v(x)-v(y)|
 \le|v(g(x))-v(g(y))|+C\eps|x-y|,
 \qquad C:=1+\frac{\|g''\|_{C^0}}{\Lambda-1}.
\]
Let us apply this inequality successively to the pairs $(g^j(x),g^j(y))$. After $n$ steps,
\[
 |v(x)-v(y)|
 \le2\Lambda^{-n}\|v\|_{C^0}
 +C\eps\sum_{j=0}^{n-1}\Lambda^{-j-1}|g^j(x)-g^j(y)|.
\]
The upper derivative bound gives, for $0 < \theta<1$,
\[
 |g^jx-g^jy|
 \le\min\{\Omega^j|x-y|,1\}
 \le\Omega^{j\theta}|x-y|^\theta.
\]
Letting $n\to\infty$ and summing the geometric series, we obtain
\[
 [v]_\theta
 \le C\eps\sum_{j=0}^{\infty}\Lambda^{-j-1}\Omega^{j\theta}
 =\frac{C\eps}{\Lambda-\Omega^\theta}.
\]
Together with the uniform estimate from Proposition~\ref{Prop:Closeness}, this proves \eqref{eq:marking-small}.
\end{proof}

%%%%%%%%%%%%%%%%%%%%%%%%%%%%%%%%%%%%%%%%%%%%%%%%%%%%%%%%%%%%%%%%%%%%%%%%%%%%%%%%%%%
\section{Whitney extension and mollification}
\label{Sec:Whitney}
%%%%%%%%%%%%%%%%%%%%%%%%%%%%%%%%%%%%%%%%%%%%%%%%%%%%%%%%%%%%%%%%%%%%%%%%%%%%%%%%%%%

Suppose that values of a function are prescribed on a finite set of points of the circle. We will be interested in smooth extensions with explicit control of their norms. This is a simple form of the one-dimensional Whitney extension problem \cite{Whitney} (see also \cite{Sh,Fef} and the references therein).

Here, we only need an extension from a uniform grid. We construct it explicitly by smooth transitions between consecutive prescribed values. The spacing of the grid controls both the approximation error and the size of the higher derivatives.

\begin{lemma}[Whitney extension]
\label{lem:whitney}
Let $v\in C^\theta({\Circle})$, where $0<\theta<1$, and let $x_i=i\eta$, $0\le i\le M$, be a uniform grid with spacing $\eta=1/M$. Here $M\ge1$ is an integer, and $x_M=1$ is identified with $x_0=0$ on the circle. There exists $\widetilde v\in C^\infty({\Circle})$ such that $\widetilde v(x_i)=v(x_i)$ at every grid point and
\begin{align}
 \|\widetilde v-v\|_{C^0}
 &\le C_\theta\eta^\theta\|v\|_{C^\theta},
 \label{eq:W-error}\\
 \|\widetilde v\|_{C^\theta}
 &\le C_\theta\|v\|_{C^\theta},
 \label{eq:W-low}\\
 \|\widetilde v\|_{C^s}
 &\le C_s\eta^{\theta-s}\|v\|_{C^\theta},
 \qquad s>\theta.
 \label{eq:W-high}
\end{align}
The constants are independent of $v$ and the grid spacing.
\end{lemma}

\begin{proof}
Choose a smooth function $\chi:[0,1]\to[0,1]$ satisfying $\chi(0)=0$, $\chi(1)=1$, and $\chi^{(j)}(0)=\chi^{(j)}(1)=0$ for every integer $j\ge1$. On each grid interval $[x_i,x_{i+1}]$, define
\[
 \widetilde v(x)
 :=v(x_i)+\bigl(v(x_{i+1})-v(x_i)\bigr)
       \chi\left(\frac{x-x_i}{\eta}\right).
\]
The endpoint conditions give $\widetilde v(x_i)=v(x_i)$. All positive-order derivatives vanish at the grid points from both sides, so the formulas on consecutive intervals define a smooth function. The same holds at $0=1$, since $v(x_M)=v(x_0)$. Thus $\widetilde v\in C^\infty({\Circle})$. We are left with proving the estimates; the arguments are standard, and we give them for completeness.

We first estimate the approximation error. If $x\in[x_i,x_{i+1}]$, then $|\chi|\le1$ and H\"older continuity give
\[
 \begin{aligned}
 |\widetilde v(x)-v(x)|
 &\le |v(x_i)-v(x)|
       +|v(x_{i+1})-v(x_i)| \le 2\eta^\theta[v]_\theta.
 \end{aligned}
\]
Taking the supremum proves \eqref{eq:W-error}. Moreover, $\widetilde v(x)$ lies between $v(x_i)$ and $v(x_{i+1})$, and hence $\|\widetilde v\|_{C^0}\le\|v\|_{C^0}$.

Let us now estimate the derivatives. On each grid interval, direct differentiation gives, for every integer $j\ge1$,
\[
 \widetilde v^{(j)}(x)
 =\bigl(v(x_{i+1})-v(x_i)\bigr)\eta^{-j}
       \chi^{(j)}\left(\frac{x-x_i}{\eta}\right).
\]
Since $|v(x_{i+1})-v(x_i)|\le\eta^\theta[v]_\theta$, we obtain $\|\widetilde v^{(j)}\|_{C^0} \le C_j\eta^{\theta-j}[v]_\theta$. These bounds hold on the whole circle because the derivatives agree at the grid points. Combining them with the uniform bound and using $\eta\le1$, we conclude that
\[
 \|\widetilde v\|_{C^j}
 \le C_j\eta^{\theta-j}\|v\|_{C^\theta},
 \qquad j\ge1.
\]
This proves \eqref{eq:W-high} for integer orders.

To prove \eqref{eq:W-low}, fix distinct points $x,y\in{\Circle}$ and put $r:=\operatorname{dist}(x,y)$. If $r\le\eta$, integrate the first derivative along the shorter arc joining the points:
\[
 |\widetilde v(x)-\widetilde v(y)|
 \le \|\widetilde v'\|_{C^0}r
 \le C\eta^{\theta-1}[v]_\theta r
 \le C[v]_\theta r^\theta.
\]
If $r\ge\eta$, then we can compare with the original function and use the approximation error proved above:
\[
 \begin{aligned}
 |\widetilde v(x)-\widetilde v(y)|
 &\le |v(x)-v(y)|+2\|\widetilde v-v\|_{C^0} \le [v]_\theta r^\theta+4\eta^\theta[v]_\theta \le 5[v]_\theta r^\theta.
 \end{aligned}
\]
Thus $[\widetilde v]_\theta\le C[v]_\theta$. Together with $\|\widetilde v\|_{C^0}\le\|v\|_{C^0}$, this proves \eqref{eq:W-low}.

It remains to prove \eqref{eq:W-high} for noninteger orders. First suppose $s=j+\alpha$, where $j\ge1$ is an integer and $0<\alpha<1$. Applying \eqref{eq:interpolation} between $C^j$ and $C^{j+1}$, we obtain
\[
 \begin{aligned}
 \|\widetilde v\|_{C^{j+\alpha}}
 &\le C_{j,\alpha}
       \|\widetilde v\|_{C^j}^{1-\alpha}
       \|\widetilde v\|_{C^{j+1}}^\alpha \le C_{j,\alpha}
       \bigl(\eta^{\theta-j}\|v\|_{C^\theta}\bigr)^{1-\alpha}
       \bigl(\eta^{\theta-j-1}\|v\|_{C^\theta}\bigr)^\alpha=C_{j,\alpha}\eta^{\theta-j-\alpha}\|v\|_{C^\theta}.
 \end{aligned}
\]
Since $s=j+\alpha$, this is the required bound.

Finally, suppose $\theta<s<1$. We again distinguish distances smaller and larger than $\eta$. For $r\le\eta$, the first derivative estimate gives
\[
 \frac{|\widetilde v(x)-\widetilde v(y)|}{r^s}
 \le C\eta^{\theta-1}[v]_\theta r^{1-s}
 \le C\eta^{\theta-s}[v]_\theta.
\]
For $r\ge\eta$, the H\"older estimate just proved gives
\[
 \frac{|\widetilde v(x)-\widetilde v(y)|}{r^s}
 \le C[v]_\theta r^{\theta-s}
 \le C\eta^{\theta-s}[v]_\theta.
\]
Therefore $[\widetilde v]_s\le C\eta^{\theta-s}[v]_\theta$. Combining this with the uniform bound proves \eqref{eq:W-high} for the remaining orders.
\end{proof}

\begin{remark}
In applications, the prescribed scale need not be the reciprocal of an integer. Given $0<\eta_0\le1$, choose $M=\lceil\eta_0^{-1}\rceil$ and use the grid with spacing $\eta=1/M$. Then
\[
 \frac{\eta_0}{2}\le\eta\le\eta_0.
\]
Consequently, \eqref{eq:W-error}--\eqref{eq:W-high} also hold with $\eta_0$ in place of $\eta$, after changing the constants.     
\end{remark}

%%%%%%%%%%%%%%%%%%%%%%%%%%%%%%%%%%%%%%%%%%%%%%%%%%%%%%%%%%%%%%%%%%%%%%%%%%%%%%%%%%%%

%%%%%%%%%%%%%%%%%%%%%%%%%%%%%%%%%%%%%%%%%%%%%%%%%%%%%%%%%%%%%%%%%%%%%%%%%%%%%%%%%%%
\section{Messenger-induced finite \Li theorem}
\label{Sec:QLivsic}
%%%%%%%%%%%%%%%%%%%%%%%%%%%%%%%%%%%%%%%%%%%%%%%%%%%%%%%%%%%%%%%%%%%%%%%%%%%%%%%%%%%

A real-valued function $D$ on ${\Circle}$ is called an \emph{observable}. If $D=u\circ g-u$ for a continuous function $u$, then $D$ is a \emph{coboundary}. Its sum along every periodic orbit of $g$ vanishes. The classical \Li theorem gives the converse for H\"older observables. In this section, we prove a finite version of \Li theorem: sums over two prescribed periods control the error in the cohomological equation, with an exponentially small error.

Throughout the first three subsections, fix $g\in\E_d^k$, where $k\ge1$, satisfying the expansion bounds~\eqref{eq:expansion}.

%%%%%%%%%%%%%%%%%%%%%%%%%%%%%%%%%%%%%%%%%%%%%%%%%%%%%%%%%%%%%%%%%%%%%%%%%%%%%%%%%%%
\subsection{The barrier function}
%%%%%%%%%%%%%%%%%%%%%%%%%%%%%%%%%%%%%%%%%%%%%%%%%%%%%%%%%%%%%%%%%%%%%%%%%%%%%%%%%%%

Recall the inverse branches $b_0,b_{d-1}$ from Subsection~\ref{SSec:Comb}. On the interval $[0,1]$, their iterates converge to $0$ and $1$, respectively. These endpoints represent the same point of ${\Circle}$. For a H\"older observable $D$, define the two \emph{one-sided barriers} by
\begin{equation}
\label{eq:one-sided-barriers}
 u^*(x):=\sum_{j=1}^{\infty}\bigl(D(b_*^j(x))-D(0)\bigr), \qquad *\in\{0,d-1\}.
\end{equation}
We can glue these two one-sided barriers into one \emph{barrier function}, denoted by $\mathcal B D$, defined on $\Circle$ as follows. Choose a fixed $C^\infty$-smooth function $\chi:[0,1]\to[0,1]$, equal to one near $0$ and zero near $1$, and put
\begin{equation}
\label{eq:uD}
 (\mathcal B D)(x):=\chi(x)u^0(x)+(1-\chi(x))u^{d-1}(x).
\end{equation}

\begin{lemma}[Preliminary bounds on the barrier]
\label{lem:barrier}
For every $0<t\le k$, the construction $D\mapsto\mathcal B D$ defines a linear map from $C^t({\Circle})$ to itself, with $(\mathcal B D)(0)=0$, and
\begin{equation}
\label{eq:u-bound}
 \|\mathcal B D\|_{C^t}\le C_t\|D\|_{C^t}.
\end{equation}
If $D=v\circ g-v$ for a continuous function $v$, then $\mathcal B D=v-v(0)$. Consequently, the remainder
\[
 r:=D-(\mathcal B D)\circ g+\mathcal B D
\]
satisfies $\mathcal B r=0$ and has the same sums as $D$ on every periodic orbit of $g$.
\end{lemma}

\begin{proof}
We first prove convergence of the series~\eqref{eq:one-sided-barriers}. By the inverse contraction from Subsection~\ref{SSec:Comb},
\[
 \operatorname{dist}(b_*^j(x),0)\le\Lambda^{-j},
 \qquad \|(b_*^j)'\|_{C^0}\le\Lambda^{-j}.
\]
If $0<t<1$, these bounds give on $[0,1]$,
\[
 \|D\circ b_*^j-D(0)\|_{C^0}\le[D]_t\Lambda^{-jt},
 \qquad [D\circ b_*^j]_t\le[D]_t\Lambda^{-jt}.
\]
The sum over $j\ge1$ is bounded by a constant times $[D]_t/(\Lambda^t-1)$. For $t=1$, the uniform estimate follows from the mean value theorem. Thus the series converges in $C^t$ for $0<t\le1$.

For higher regularity, we first estimate the derivatives of the iterated inverse branches. Differentiating $b_*^{j+1}=b_*\circ b_*^j$ at an integer order $a\ge2$, the term containing the highest derivative of $b_*^j$ is $(b_*'\circ b_*^j)(b_*^j)^{(a)}$. Every other term is a constant multiple of
\[
 (b_*^{(q)}\circ b_*^j)\prod_{i=1}^{q}(b_*^j)^{(a_i)},
 \qquad q\ge2,\quad a_i\ge1,\quad \sum_{i=1}^{q}a_i=a.
\]
Assuming inductively that the derivatives of orders below $a$ are bounded by $C_a\Lambda^{-j}$, we therefore obtain
\[
 \|(b_*^{j+1})^{(a)}\|_{C^0}
 \le\Lambda^{-1}\|(b_*^j)^{(a)}\|_{C^0}+C_a\Lambda^{-2j}.
\]
Iterating this inequality gives
\[
 \begin{aligned}
 \|(b_*^j)^{(a)}\|_{C^0}
 &\le\Lambda^{-(j-1)}\|b_*^{(a)}\|_{C^0}
   +C_a\sum_{i=1}^{j-1}\Lambda^{-(j-1-i)}\Lambda^{-2i} \le C_a'\Lambda^{-j}.
 \end{aligned}
\]
This proves the required bounds for every available integer derivative.

For a noninteger order $a+\alpha$, where $a\ge1$ and $0<\alpha<1$, we use the elementary H\"older estimates
\begin{equation}
\label{eq:barrier-holder-products}
 \begin{aligned}
 [AB]_\alpha &\le\|A\|_{C^0}[B]_\alpha+\|B\|_{C^0}[A]_\alpha,\\
 [A\circ b_*^j]_\alpha
 &\le[A]_\alpha\Lambda^{-j\alpha}.
 \end{aligned}
\end{equation}
Apply these inequalities to the same differentiated formula. Derivatives of $b_*^j$ below order $a$ have H\"older seminorm bounded by $C_a\Lambda^{-j}$, since their next derivatives satisfy the integer estimates above. The term containing $[(b_*^j)^{(a)}]_\alpha$ has coefficient at most $\Lambda^{-1}$, and the remaining terms give
\[
 [(b_*^{j+1})^{(a)}]_\alpha
 \le\Lambda^{-1}[(b_*^j)^{(a)}]_\alpha
       +C_{a,\alpha}\Lambda^{-(1+\alpha)j}.
\]
Summing the resulting geometric series, as in the integer case, yields $[(b_*^j)^{(a)}]_\alpha\le C_{a,\alpha}'\Lambda^{-j}$. The chain rule and these estimates now give, for every $1\le t\le k$,
\[
 \|D\circ b_*^j-D(0)\|_{C^t([0,1])}
 \le C_t\Lambda^{-j}\|D\|_{C^t}.
\]
Consequently, both series converge in $C^t([0,1])$, with norms at most $C_t\|D\|_{C^t}$. Multiplication by the fixed cutoff in~\eqref{eq:uD} preserves these estimates.

Let us check the endpoint identification. In a local angular coordinate $x\in(-a,a)$ about $0$, let $b$ be the local inverse of $g$ fixing $0$. For sufficiently small $a>0$,
\[
 b(x)=b_0(x)\quad(x\ge0),
 \qquad b(x)=b_{d-1}(1+x)-1\quad(x\le0).
\]
Thus the two endpoint expressions in~\eqref{eq:uD} are restrictions of the same local function
\[
 \sum_{j=1}^{\infty}\bigl(D(b^j(x))-D(0)\bigr).
\]
By the estimates already proved, this function is $C^t$. Hence the endpoint values and all available derivatives agree, and the value at $0$ is zero. This proves~\eqref{eq:u-bound} on ${\Circle}$. Linearity follows directly from~\eqref{eq:one-sided-barriers} and~\eqref{eq:uD}.

Suppose now that $D=v\circ g-v$ with $v$ continuous. Since $g(0)=0$, we have $D(0)=0$. For either endpoint branch, the partial sums telescope:
\[
 \begin{aligned}
 \sum_{j=1}^{\ell}D(b_*^j(x))
 &=\sum_{j=1}^{\ell}\bigl(v(b_*^{j-1}(x))-v(b_*^j(x))\bigr)=v(x)-v(b_*^\ell(x))
 \longrightarrow_{\ell \to \infty} v(x)-v(0).
 \end{aligned}
\]
Both one-sided barriers are therefore equal to $v-v(0)$, and so is $\mathcal B D$.

Applying this identity to $v=\mathcal B D$, whose value at $0$ is zero, we obtain
\[
 \mathcal B r
 =\mathcal B D-\mathcal B\bigl((\mathcal B D)\circ g-\mathcal B D\bigr)
 =0.
\]
\end{proof}

%%%%%%%%%%%%%%%%%%%%%%%%%%%%%%%%%%%%%%%%%%%%%%%%%%%%%%%%%%%%%%%%%%%%%%%%%%%%%%%%%%%
\subsection{Ordinary and hybrid messengers}
%%%%%%%%%%%%%%%%%%%%%%%%%%%%%%%%%%%%%%%%%%%%%%%%%%%%%%%%%%%%%%%%%%%%%%%%%%%%%%%%%%%

We use the symbolic addresses and cylinders from Subsection~\ref{SSec:Comb}. If $x_-,x_+\in\mathbb P_n^g$ have periodic addresses $\overline{\sigma_-},\overline{\sigma_+}$, and $p,q\ge1$ are integers, the \emph{$(p,q)$-messenger connecting $x_-$ to $x_+$} is the periodic point $y_-$ of period dividing $(p+q) \cdot n$ with address $\overline{\sigma_-^p\sigma_+^q}$. Here $\sigma^p$ denotes $p$ consecutive copies of the word $\sigma$. The messenger follows the itinerary of $x_-$ for $pn$ steps and that of $x_+$ for $qn$ steps. We set $y_+:=g^{pn}(y_-)$; by definition, $y_+$ is the $(q,p)$-messenger connecting $x_+$ to $x_-$.

We obtain estimates for its closest approaches directly from the cylinder bounds. Indeed, $y_-$ and $x_-$ share a prefix of length $pn$, while $y_+$ and $x_+$ share a prefix of length $qn$. By~\eqref{Eq:Cylinder},
\begin{equation}
\label{eq:messenger-close}
 \operatorname{dist}(y_-,x_-)\le\Lambda^{-pn},
 \qquad \operatorname{dist}(y_+,x_+)\le\Lambda^{-qn}.
\end{equation}
These bounds also hold for the distances in $[0,1]$ when the points are represented by the indicated addresses. We will use this observation when applying the endpoint inverse branches.

For a point $x_i\in\mathbb P_n^g$ with address $\overline{\sigma_i}$, the \emph{ordinary $(p,p)$-messenger} from the side $*\in\{0,d-1\}$ of $0 \in {\Circle}$ to $x_i$ is a periodic point $y_i^- \in \cP^{g}_{2pn}$ given by the symbolic code $\overline{*^{pn}\sigma_i^p}$. We also set $y_i^+:=g^{pn}(y_i^-)$; this is an ordinary $(p,p)$-messenger connecting $x_i$ to $0 \in {\Circle}$ from the side $*$.

Whenever $x_i=0$, we choose its address to be $\overline{\sigma_i}=\overline{0}$.

We next sew the incoming half of one ordinary messenger to the outgoing half of another using \emph{hybrid messengers}. We allow the two ordinary messengers to use different sides of $0$. For this construction, we first recall the following shadowing statement.

A list of points $x_0,\ldots,x_{m-1}\in{\Circle}$ is a \emph{periodic $\varepsilon$-pseudo-orbit of length $m$} if
\[
 \operatorname{dist}(g(x_j),x_{j+1})\le\varepsilon,
 \qquad 0\le j<m,\qquad x_m:=x_0.
\]
Repeated points are allowed. The quantities $\operatorname{dist}(g(x_j),x_{j+1})$ are its jumps. An orbit shadows this list if its successive points are close to the corresponding $x_j$.

\begin{lemma}[Shadowing a periodic pseudo-orbit]
\label{lem:shadowing}
Let $x_0,\ldots,x_{m-1}$ be a periodic pseudo-orbit, and put $e_j:=\operatorname{dist}(g(x_j),x_{j+1})$, with cyclic indices. If $\max_j e_j$ is sufficiently small, there exists an orbit $z_j=g^j(z_0)$ with $g^m(z_0)=z_0$ such that, for $0<\theta\le1$,
\begin{equation}
\label{eq:shadowing}
 \max_{0\le j<m}\operatorname{dist}(z_j,x_j)
 \le\frac{\max_j e_j}{\Lambda-1},
 \qquad
 \sum_{j=0}^{m-1}\operatorname{dist}(z_j,x_j)^\theta
 \le\frac{\sum_{j=0}^{m-1}e_j^\theta}{\Lambda^\theta-1}.
\end{equation}
The allowed size of the jumps can be chosen independently of $m$.
\end{lemma}

\begin{proof}
If all jumps $e_j$ vanish, we take $z_j=x_j$. Otherwise, put $R:=\max_j e_j/(\Lambda-1)$, and let $I_j$ be the closed circle interval of radius $R$ centered at $x_j$. We assume that the jumps are sufficiently small that
\[
 R+\max_j e_j<\frac14.
\]
Equivalently, $\max_j e_j<(\Lambda-1)/(4\Lambda)$, so this condition is independent of $m$.

We first construct inverse branches mapping $I_{j+1}$ into $I_j$. Since $g$ is a covering map, there is an inverse branch $b_j$ on the circle interval of radius $1/4$ centered at $g(x_j)$, with $b_j(g(x_j))=x_j$. By the preceding smallness condition, this interval contains $I_{j+1}$. Moreover, \eqref{eq:expansion} gives
\[
 \operatorname{dist}(b_j(y),b_j(w))
 \le\Lambda^{-1}\operatorname{dist}(y,w)
\]
for points in its domain. Consequently, for $y\in I_{j+1}$,
\[
 \begin{aligned}
 \operatorname{dist}(b_j(y),x_j)
 &=\operatorname{dist}(b_j(y),b_j(g(x_j))) \le\Lambda^{-1}\operatorname{dist}(y,g(x_j))\\
 &\le\Lambda^{-1}(R+e_j)
 \le\Lambda^{-1}(R+\max_i e_i)=R.
 \end{aligned}
\]
Thus $b_j(I_{j+1})\subset I_j$.

We now find points satisfying $z_j=b_j(z_{j+1})$ simultaneously. On the product $I_0\times\cdots\times I_{m-1}$, define
\[
 T(y_0,\ldots,y_{m-1})
 :=\bigl(b_j(y_{j+1})\bigr)_{j=0}^{m-1},
 \qquad y_m:=y_0.
\]
The inclusions just proved show that $T$ maps this product into itself. It is a contraction in the maximum metric, because
\[
 \max_j\operatorname{dist}(b_j(y_{j+1}),b_j(w_{j+1}))
 \le\Lambda^{-1}\max_j\operatorname{dist}(y_j,w_j).
\]
The contraction mapping theorem therefore gives a fixed point $(z_0,\ldots,z_{m-1})$. The fixed-point identities imply $g(z_j)=z_{j+1}$, and hence $g^m(z_0)=z_0$. Since $z_j\in I_j$, we also obtain
\[
 \max_j\operatorname{dist}(z_j,x_j)
 \le R=\frac{\max_j e_j}{\Lambda-1},
\]
which is the first estimate in \eqref{eq:shadowing}.

To prove the second estimate, we use the identities $z_j=b_j(z_{j+1})$ and $x_j=b_j(g(x_j))$ to obtain
\[
 \operatorname{dist}(z_j,x_j)
 \le\Lambda^{-1}
 \bigl(\operatorname{dist}(z_{j+1},x_{j+1})+e_j\bigr).
\]
For $0<\theta\le1$, the inequality $(a+b)^\theta\le a^\theta+b^\theta$ therefore yields
\[
 \Lambda^\theta\operatorname{dist}(z_j,x_j)^\theta
 \le\operatorname{dist}(z_{j+1},x_{j+1})^\theta+e_j^\theta.
\]
Summing over $j=0,\ldots,m-1$, 
\[
 (\Lambda^\theta-1)
 \sum_{j=0}^{m-1}\operatorname{dist}(z_j,x_j)^\theta
 \le\sum_{j=0}^{m-1}e_j^\theta,
\]
as was claimed.
\end{proof}

Suppose now that $x_j=g^t(x_i)$, where $1\le t\le n$. Write the address of $x_i$ as $\overline{\sigma_i}=\overline{s_t s_{n-t}}$, where the two blocks have the indicated lengths. Then $\overline{\sigma_j}=\overline{s_{n-t}s_t}$ is the address of $x_j$. We choose sides $*,*'\in\{0,d-1\}$ and take the ordinary $(p,p)$-messengers
\[
 y_i^-:\quad\overline{*^{pn}\sigma_i^p},
 \qquad y_j^-:\quad\overline{(*')^{pn}\sigma_j^p},
 \qquad y_i^+:=g^{pn}(y_i^-),\quad y_j^+:=g^{pn}(y_j^-).
\]
Here $*$ specifies the side of the incoming messenger at $x_i$, and $*'$ specifies the side of the outgoing messenger at $x_j$.

A \emph{hybrid messenger of type $(p,p)$ with sides $(*,*')$} is a point $z_-\in\mathbb P^g_{2pn+t}$ whose orbit shadows these three pieces in their order:
\begin{itemize}
\item the orbit of $y_i^-$ for $pn$ iterates;
\item the orbit of $x_i$ for $t$ iterates;
\item the orbit of $y_j^+$ for $pn$ iterates.
\end{itemize}
We construct it below and denote the two junctions by $z_0:=g^{pn}(z_-)$ and $z_+:=g^t(z_0)$; see Figure~\ref{Fig:Hybrid}. The sides specify the ordinary messenger segments being sewn, and the shadowing comparison takes place on ${\Circle}$.

When $*=*'$, we retain the symbolic construction: $z_-$ is the point with address $\overline{*^{pn}s_t\sigma_j^p}$. We obtain a uniform estimate for its distance from the three pieces. Indeed, the word identity $s_t\sigma_j^p=\sigma_i^p s_t$ shows that $z_-$ and $y_i^-$ share $2pn$ initial symbols. Similarly, $z_+$ and $y_j^+$ share the prefix $\sigma_j^p*^{pn}$, while $z_0$ and $x_i$ share the prefix $\sigma_i^p s_t$ of length $pn+t$. By~\eqref{Eq:Cylinder},
\[
 \begin{aligned}
 \operatorname{dist}(g^\ell(z_-),g^\ell(y_i^-))
 &\le\Lambda^{-(2pn-\ell)}, &&0\le\ell<pn,\\
 \operatorname{dist}(g^\ell(z_0),g^\ell(x_i))
 &\le\Lambda^{-(pn+t-\ell)}, &&0\le\ell<t,\\
 \operatorname{dist}(g^\ell(z_+),g^\ell(y_j^+))
 &\le\Lambda^{-(2pn-\ell)}, &&0\le\ell<pn.
 \end{aligned}
\]
For $0<\theta\le1$, each of the three sums of the $\theta$th powers is bounded by the same geometric series. For example,
\[
 \sum_{\ell=0}^{pn-1}\Lambda^{-(2pn-\ell)\theta}
 =\Lambda^{-pn\theta}\sum_{r=1}^{pn}\Lambda^{-r\theta}
 \le\frac{\Lambda^{-pn\theta}}{\Lambda^\theta-1},
\]
and the central sum has the same form with $r=1,\ldots,t$. Thus their total is at most $3\Lambda^{-pn\theta}/(\Lambda^\theta-1)$.

For $*\ne*'$, we use Lemma~\ref{lem:shadowing}. The list to be shadowed is
\[
 \begin{gathered}
 y_i^-,g(y_i^-),\ldots,g^{pn-1}(y_i^-),\qquad
 x_i,g(x_i),\ldots,g^{t-1}(x_i),\\
 y_j^+,g(y_j^+),\ldots,g^{pn-1}(y_j^+).
 \end{gathered}
\]
All transitions are exact except possibly the three junctions. By~\eqref{eq:messenger-close}, their jumps satisfy
\[
 \begin{aligned}
 \operatorname{dist}(g^{pn}(y_i^-),x_i)
 &=\operatorname{dist}(y_i^+,x_i)\le\Lambda^{-pn},\\
 \operatorname{dist}(g^t(x_i),y_j^+)
 &=\operatorname{dist}(x_j,y_j^+)\le\Lambda^{-pn},\\
 \operatorname{dist}(g^{pn}(y_j^+),y_i^-)
 &=\operatorname{dist}(y_j^-,y_i^-)\le2\Lambda^{-pn}.
 \end{aligned}
\]
The last estimate holds because both $y_i^-$ and $y_j^-$ are within distance $\Lambda^{-pn}$ of $0$ on ${\Circle}$. For sufficiently large $pn$, Lemma~\ref{lem:shadowing} gives the required periodic orbit. We take its initial point as $z_-$ and set $z_0:=g^{pn}(z_-)$ and $z_+:=g^t(z_0)$. The required lower bound on $pn$ is independent of $t$ and of the two sides.

The second estimate in~\eqref{eq:shadowing} bounds the total sum of the $\theta$th powers of the corresponding distances by
\[
 \frac{\Lambda^{-pn\theta}+\Lambda^{-pn\theta}+(2\Lambda^{-pn})^\theta}{\Lambda^\theta-1}
 =\frac{2+2^\theta}{\Lambda^\theta-1}\Lambda^{-pn\theta}.
\]
Since $3\le2+2^\theta$, this also includes the bound for the symbolic hybrid when $*=*'$. We have therefore proved, for every choice of the sides,
\begin{equation}
\label{eq:hybrid-comparison}
 \begin{aligned}
 &\sum_{\ell=0}^{pn-1}\operatorname{dist}(g^\ell(z_-),g^\ell(y_i^-))^\theta
 +\sum_{\ell=0}^{t-1}\operatorname{dist}(g^\ell(z_0),g^\ell(x_i))^\theta\\
 &\qquad+\sum_{\ell=0}^{pn-1}\operatorname{dist}(g^\ell(z_+),g^\ell(y_j^+))^\theta
 \le\frac{2+2^\theta}{\Lambda^\theta-1}\Lambda^{-pn\theta}.
 \end{aligned}
\end{equation}
In particular, the estimate has no factor depending on the length $2pn+t$ of the hybrid orbit.

\begin{figure}[htbp]
\centering
\resizebox{\textwidth}{!}{%
\definecolor{qqzzqq}{rgb}{0.,0.6,0.}
\definecolor{ffqqqq}{rgb}{1.,0.,0.}
\definecolor{qqqqff}{rgb}{0.,0.,1.}
\begin{tikzpicture}[line cap=round,line join=round,>={Stealth},x=1.0cm,y=1.0cm]
\clip(-10.3,-2) rectangle (7.3,1.9);
\draw [line width=1.2pt] (-9.765268033844542,0.)-- (4.675488994250826,0.);
\draw [color=black, line width=1.2pt] (-9.765268033844542,0.)-- ++(-2.5pt,-2.5pt) -- ++(5.0pt,5.0pt) ++(-5.0pt,0) -- ++(5.0pt,-5.0pt);
\draw (-9.765268033844542,0.) node[anchor=north,yshift=-2pt] {$0$};
\draw [fill=qqqqff,color=qqqqff] (-9.02360386015981,0.) circle (2.5pt);
\draw (-9.02360386015981,0.) node[anchor=north,yshift=-2pt] {\textcolor{blue}{$y_i^-$}};
\draw [fill=ffqqqq, color=ffqqqq] (-9.208346747908648,0.) circle (2pt);
\draw (-9.208346747908648,0.) node[anchor=south,yshift=2pt,xshift=1pt] {\textcolor{red}{$z_-$}};
\draw [fill=qqzzqq, color=qqzzqq] (-8.505367722989183,0.) circle (2.5pt);
\draw (-8.5,0.) node[anchor=west,xshift=5pt, yshift=-6pt] {\textcolor{qqzzqq}{$y_j^-$}};
\draw [fill=qqqqff, color=qqqqff] (-1.644045331996723,0.) circle (2.5pt);
\draw (-1.644045331996723,0.) node[anchor=north,yshift=-2pt] {\textcolor{blue}{$y_i^+$}};
\draw [fill=ffqqqq, color=ffqqqq] (-1.9229277623796266,0.) circle (2pt);
\draw (-1.9229277623796266,0.) node[anchor=south west,yshift=4pt,xshift=-4pt] {\textcolor{red}{$g^{np}(z_-)=:z_0$}};
\draw [fill=black] (-1.2,0.) ++(-2.5pt,0 pt) -- ++(2.5pt,2.5pt)--++(2.5pt,-2.5pt)--++(-2.5pt,-2.5pt)--++(-2.5pt,2.5pt);
\draw (-1.2,0.) node[anchor=north,yshift=-1pt] {\textcolor{black}{$x_i$}};
\draw [fill=ffqqqq,color=ffqqqq] (3.8127420091223887,0.) circle (2pt);
\draw (3.8127420091223887,0.) node[anchor=south west,yshift=4pt,xshift=-4pt] {\textcolor{red}{$g^{np+t}(z_-)=:z_+$}};
\draw [fill=qqzzqq,color=qqzzqq] (4.084347541477638,0.) circle (2.5pt);
\draw (4.084347541477638,0.) node[anchor=north,yshift=-6pt] {\textcolor{qqzzqq}{$y_j^+$}};
\draw [fill=black] (4.675488994250826,0.) ++(-2.5pt,0 pt) -- ++(2.5pt,2.5pt)--++(2.5pt,-2.5pt)--++(-2.5pt,-2.5pt)--++(-2.5pt,2.5pt);
\draw (4.675488994250826,0.) node[anchor=north,yshift=-2pt] {\textcolor{black}{$x_j$}};
\draw [-{>[sep=2pt]},color=blue,line width=1pt] (-9.02360386015981,0.) to [bend right] node [midway,below]{\textcolor{blue}{$g^{np}$}}(-1.644045331996723,0.);
\draw [-{>[sep=2pt]},color=qqzzqq,line width=1pt] (4.084347541477638,0.) to [bend right] node [midway,below]{\textcolor{qqzzqq}{$g^{np}$}}(-8.505367722989183,0.);
\draw [-{>[sep=2pt]},color=black,line width=1pt] (-1.2,0.) to [bend right] node [midway,below]{\textcolor{black}{$g^{t}$}}(4.675488994250826,0.);
\draw [-{>[sep=2pt]},color=red] (-9.208346747908648,0.) to [bend right] node [pos=0.75,above]{\textcolor{red}{$g^{np}$}}(-1.9229277623796266,0.);
\draw [-{>[sep=2pt]},color=red] (-1.9229277623796266,0.) to [bend right] node [pos=0.75,above, yshift=-2pt]{\textcolor{red}{$g^{t}$}}(3.8127420091223887,0.);
\draw [-{>[sep=2pt]},color=red] (3.8127420091223887,0.) to [bend right] node [pos=0.75,above]{\textcolor{red}{$g^{np}$}}(-9.208346747908648,0.);
\end{tikzpicture}}
\caption{A $(p,p)$-messenger connecting $0$ to $x_i \in \cP^g_n$ is shown in blue, while a $(p,p)$-messenger connecting $0$ to $x_j = g^t(x_i)$ is shown in green. A hybrid messenger (in red) connects $0$, $x_i$ and $x_j = g^t(x_i)$ by shadowing halves of the orbits of blue and green messengers and the orbit $x_i, g(x_i), \ld, g^t(x_i) = x_j$ (in black).}
\label{Fig:Hybrid}\label{fig:messengers}
\end{figure}

We now take $p=t=1$. The ordinary messengers have period dividing $2n$, and the hybrids used below have period dividing $2n+1$. These points belong to $\mathbb P_{2n}^g$ and $\mathbb P_{2n+1}^g$, respectively, even if their least periods are smaller. Define the finite periodic error by
\begin{equation}
\label{eq:finite-periodic-error}
 \epsilon_n(D):=\max_{m\in\{2n,2n+1\}}\ \max_{x\in\mathbb P_m^g}
       \left|\sum_{j=0}^{m-1}D(g^j(x))\right|.
\end{equation}

The uniform approximation estimate from Lemma~\ref{lem:whitney} controls the change in this finite periodic error. Indeed, for any two continuous observables $D,\widetilde D$, the triangle inequality in each sum gives
\[
 \epsilon_n(\widetilde D)
 \le\epsilon_n(D)+(2n+1)\|\widetilde D-D\|_{C^0}.
\]
We will use this bound when applying the finite \Li theorem after regularization by Lemma~\ref{lem:whitney}.

\begin{lemma}[Barrier and messenger comparison]
\label{lem:messenger}
For $D\in C^1({\Circle})$,
\begin{equation}
\label{eq:messenger}
 \max_{x\in\mathbb P_n^g}
 \left|D(x)-(\mathcal B D)(g(x))+(\mathcal B D)(x)\right|
 \le2\epsilon_n(D)+C\Lambda^{-n}\|D\|_{C^1}.
\end{equation}
For $D\in C^\theta({\Circle})$, where $0<\theta<1$, the same estimate holds with the last term replaced by $C_\theta\Lambda^{-\theta n}\|D\|_{C^\theta}$.
\end{lemma}

\begin{proof}
We first prove the $C^1$ estimate for sufficiently large $n$ (for small $n$ we can just adjust the constant). Fix $x_i\in\mathbb P_n^g$ and a side $*\in\{0,d-1\}$, and take the ordinary $(1,1)$-messenger $y_i^-,y_i^+=g^n(y_i^-)$ connecting $0$ from side $*$ to $x_i$. Reading its first $n$ symbols backwards gives
\[
 g^{n-\ell}(y_i^-)=b_*^\ell(y_i^+),\qquad 1\le\ell\le n.
\]
Let us compare the sum along this incoming half with the one-sided barrier at $x_i$. The points $y_i^+$ and $x_i$ share the prefix $\sigma_i$, so their distance in $[0,1]$ is at most $\Lambda^{-n}$ by~\eqref{Eq:Cylinder}. Inverse contraction and the mean value theorem therefore give
\[
 \begin{aligned}
 \left|\sum_{\ell=1}^{n}D(b_*^\ell(y_i^+))
       -\sum_{\ell=1}^{n}D(b_*^\ell(x_i))\right|
 &\le\|D'\|_{C^0}\Lambda^{-n}\sum_{\ell=1}^{n}\Lambda^{-\ell}\le\frac{\Lambda^{-n}}{\Lambda-1}\|D'\|_{C^0}.
 \end{aligned}
\]
The remaining tail in~\eqref{eq:one-sided-barriers} satisfies
\[
 \left|\sum_{\ell=n+1}^{\infty}\bigl(D(b_*^\ell(x_i))-D(0)\bigr)\right|
 \le\|D'\|_{C^0}\sum_{\ell=n+1}^{\infty}\Lambda^{-\ell}
 =\frac{\Lambda^{-n}}{\Lambda-1}\|D'\|_{C^0}.
\]
Adding these estimates, we obtain
\begin{equation}
\label{eq:incoming}
 \left|\sum_{\ell=0}^{n-1}D(g^\ell(y_i^-))-nD(0)-u^*(x_i)\right|
 \le C\Lambda^{-n}\|D\|_{C^1}.
\end{equation}

The sum along the outgoing half is the complete $2n$-term sum minus the incoming sum. Consequently, \eqref{eq:incoming} implies
\begin{equation}
\label{eq:outgoing}
 \left|\sum_{\ell=0}^{n-1}D(g^\ell(y_i^+))
       -\sum_{\ell=0}^{2n-1}D(g^\ell(y_i^-))+nD(0)+u^*(x_i)\right|
 \le C\Lambda^{-n}\|D\|_{C^1}.
\end{equation}

We now compare the barriers at $x_i$ and $x_j:=g(x_i)$. We choose sides $*,*'\in\{0,d-1\}$ and take the corresponding hybrid messenger of type $(1,1)$, with $t=1$. Thus $z_-\in\mathbb P_{2n+1}^g$, $z_0=g^n(z_-)$, and $z_+=g(z_0)$. Its incoming segment shadows the ordinary messenger $y_i^-$ from side $*$, while its outgoing segment shadows $y_j^+$ from side $*'$.

We decompose the complete sum along the hybrid into these three parts:
\[
 \sum_{\ell=0}^{2n}D(g^\ell(z_-))
 =\sum_{\ell=0}^{n-1}D(g^\ell(z_-))+D(z_0)
       +\sum_{\ell=0}^{n-1}D(g^\ell(z_+)).
\]
By~\eqref{eq:hybrid-comparison} with $p=t=\theta=1$, the sum of the distances to the corresponding ordinary segments and to $x_i$ is at most $4\Lambda^{-n}/(\Lambda-1)$. The mean value theorem therefore gives
\[
 \left|\sum_{\ell=0}^{2n}D(g^\ell(z_-))
       -\left(\sum_{\ell=0}^{n-1}D(g^\ell(y_i^-))
              +D(x_i)+\sum_{\ell=0}^{n-1}D(g^\ell(y_j^+))\right)\right|
 \le\frac{4\Lambda^{-n}}{\Lambda-1}\|D'\|_{C^0}.
 \]
 
We now substitute~\eqref{eq:incoming} for the incoming sum at $x_i$ and~\eqref{eq:outgoing} for the outgoing sum at $x_j$. The terms $nD(0)$ cancel, and we obtain
\begin{multline*}
 \left|D(x_i)+u^*(x_i)-u^{*'}(x_j)
       -\left(\sum_{\ell=0}^{2n}D(g^\ell(z_-))
              -\sum_{\ell=0}^{2n-1}D(g^\ell(y_j^-))\right)\right| \le C\Lambda^{-n}\|D\|_{C^1}.
\end{multline*}
Thus, the barrier difference is controlled by subtracting the complete sum along an ordinary messenger from the complete sum along a hybrid messenger. Each sum has absolute value at most $\epsilon_n(D)$ by~\eqref{eq:finite-periodic-error}. Hence, for every choice of the two sides,
\[
 |D(x_i)+u^*(x_i)-u^{*'}(x_j)|
 \le2\epsilon_n(D)+C\Lambda^{-n}\|D\|_{C^1}.
\]

It remains to combine the one-sided barriers into $\mathcal B D$. For fixed $*'$, we average the expressions inside the last absolute value over $*$, with weights $\chi(x_i)$ for $*=0$ and $1-\chi(x_i)$ for $*=d-1$. These weights are nonnegative and sum to one. By~\eqref{eq:uD} and the triangle inequality, this gives
\[
 |D(x_i)+(\mathcal B D)(x_i)-u^{*'}(x_j)|
 \le2\epsilon_n(D)+C\Lambda^{-n}\|D\|_{C^1}.
\]
We then average over $*'$ with weights $\chi(x_j)$ and $1-\chi(x_j)$. We obtain~\eqref{eq:messenger}, since $x_j=g(x_i)$. The two averages may use different weights, which is why we included hybrids with independently chosen sides.

For $D\in C^\theta({\Circle})$, $0<\theta<1$, the estimate is similar.
\end{proof}

%%%%%%%%%%%%%%%%%%%%%%%%%%%%%%%%%%%%%%%%%%%%%%%%%%%%%%%%%%%%%%%%%%%%%%%%%%%%%%%%%%%
\subsection{Rolle's theorem and the exponential improvement}
%%%%%%%%%%%%%%%%%%%%%%%%%%%%%%%%%%%%%%%%%%%%%%%%%%%%%%%%%%%%%%%%%%%%%%%%%%%%%%%%%%%

The messenger estimate controls the error on $\mathbb P_n^g$. We use the mesh estimate from Lemma~\ref{Lem:Grid} and the approximate Rolle estimate from Lemma~\ref{Lem:AlaRolle2} to pass to the whole circle. Applying the messenger estimate to the remainder itself will then allow us to improve the exponential factor by interpolation.

\begin{theorem}[Finite Liv\v{s}ic with exponential bounds]
\label{thm:finite-livsic}
For every $0<s\le k$ and every $D\in C^s({\Circle})$, the barrier $\mathcal B D$ satisfies, for all $n\ge1$,
\[
 \begin{aligned}
 \|D-(\mathcal B D)\circ g+\mathcal B D\|_{C^0}
 &\le C_s\bigl(\epsilon_n(D)+\Lambda^{-sn}\|D\|_{C^s}\bigr),\\
 \|\mathcal B D\|_{C^s}
 +\|D-(\mathcal B D)\circ g+\mathcal B D\|_{C^s}
 &\le C_s\|D\|_{C^s}.
 \end{aligned}
\]
The constants are independent of $n$ and $D$. 
\end{theorem}

\begin{proof}
Put $r:=D-(\mathcal B D)\circ g+\mathcal B D$. We first record its top-norm bound. Composition with the fixed map $g$ gives $\|u\circ g\|_{C^s}\le C_s\|u\|_{C^s}$. For integer $s$, this follows from~\eqref{Eq:CompisitionMod}; for noninteger $s$, use the chain rule and the product estimate in~\eqref{eq:barrier-holder-products}, together with $[A\circ g]_\alpha\le[A]_\alpha\|g'\|_{C^0}^\alpha$. Therefore Lemma~\ref{lem:barrier} gives
\[
 \|r\|_{C^s}
 \le\|D\|_{C^s}+C_s\|\mathcal B D\|_{C^s}
 \le C_s\|D\|_{C^s}.
\]
Together with~\eqref{eq:u-bound}, this proves the second estimate of the theorem.

We now estimate the $C^0$ norm. By Lemma~\ref{lem:barrier}, $\mathcal B r=0$, and the periodic sums of $r$ and $D$ agree. Thus the definition~\eqref{eq:finite-periodic-error} gives $\epsilon_n(r)=\epsilon_n(D)$. For $s\ge1$, applying Lemma~\ref{lem:messenger} to $r$ yields
\[
 \max_{x\in\mathbb P_n^g}|r(x)|
 \le2\epsilon_n(D)+C\Lambda^{-n}\|r\|_{C^1}.
\]
The left-hand side is the value of $r$ itself because its barrier vanishes.

By~\eqref{eq:mesh}, every closed arc of length
\[
 \delta:=\min\left\{1,\frac{\Lambda^{-n}}{1-\Lambda^{-1}}\right\}
\]
contains a point of $\mathbb P_n^g$. Apply~\eqref{eq:sampling} with $E=P_n^g$ and this value of $\delta$. Since $\delta^s\le C_s\Lambda^{-sn}$, the preceding estimate gives
\begin{equation}
\label{eq:livsic-before-young}
 \|r\|_{C^0}
 \le C_s\epsilon_n(D)+C_s\Lambda^{-n}\|r\|_{C^1}
       +C_s\Lambda^{-sn}\|r\|_{C^s}.
\end{equation}

Suppose first that $s>1$. Using~\eqref{eq:interpolation} between $C^0$ and $C^s$, with intermediate order one, we obtain
\[
 \|r\|_{C^1}
 \le C_s\|r\|_{C^0}^{(s-1)/s}\|r\|_{C^s}^{1/s}.
\]
We will absorb the resulting contribution of the uniform norm into the left-hand side of~\eqref{eq:livsic-before-young}. Young's inequality with conjugate exponents $s/(s-1)$ and $s$, after rescaling its two factors, reads
\[
 ab\le\tfrac12 a^{s/(s-1)}+C_s b^s,
 \qquad a,b\ge0.
\]
Take $a=\|r\|_{C^0}^{(s-1)/s}$ and $b=C_s\Lambda^{-n}\|r\|_{C^s}^{1/s}$. Combining the last two inequalities gives
\[
 C_s\Lambda^{-n}\|r\|_{C^1}
 \le\tfrac12\|r\|_{C^0}+C_s'\Lambda^{-sn}\|r\|_{C^s}.
\]
Substitute this bound into~\eqref{eq:livsic-before-young} and move $\tfrac12\|r\|_{C^0}$ to the left. We obtain
\[
 \|r\|_{C^0}
 \le C_s\bigl(\epsilon_n(D)+\Lambda^{-sn}\|r\|_{C^s}\bigr)
 \le C_s'\bigl(\epsilon_n(D)+\Lambda^{-sn}\|D\|_{C^s}\bigr),
\]
as required. The exponent $s$ appears because Young's inequality raises the factor $\Lambda^{-n}$ to the power $s$.

For $s=1$, both remainder terms in~\eqref{eq:livsic-before-young} already have the form $C\Lambda^{-n}\|r\|_{C^1}$. The top-norm bound proved at the beginning therefore gives the conclusion directly. For $0<s<1$, apply the H\"older part of Lemma~\ref{lem:messenger} to $r$, followed by~\eqref{eq:sampling}. This gives
\[
 \|r\|_{C^0}
 \le C_s\bigl(\epsilon_n(D)+\Lambda^{-sn}\|r\|_{C^s}\bigr)
 \le C_s'\bigl(\epsilon_n(D)+\Lambda^{-sn}\|D\|_{C^s}\bigr).
\]
This proves the theorem for every stated order.
\end{proof}

\hide{
The exact identity $\mathcal B r=0$ is what permits the absorption in this proof: the messenger error involves $\|r\|_{C^1}$, which can be interpolated with the unknown small norm $\|r\|_{C^0}$. At the same time, the periodic error is unchanged, since subtracting a coboundary does not change periodic sums.

\begin{corollary}[Liv\v{s}ic theorem]
\label{cor:livsic}
Let $D$ be a H\"older observable with zero sum on every periodic orbit. Then
\[
 D=(\mathcal B D)\circ g-\mathcal B D.
\]
If, in addition, $D\in C^s({\Circle})$ for $0<s\le k$, then $\mathcal B D\in C^s({\Circle})$ and satisfies~\eqref{eq:u-bound}.
\end{corollary}

\begin{proof}
The hypothesis implies $\epsilon_n(D)=0$ for every $n$. Apply Theorem~\ref{thm:finite-livsic} at a H\"older order $0<s<1$ for which $D\in C^s({\Circle})$, and let $n\to\infty$. We obtain
\[
 \|D-(\mathcal B D)\circ g+\mathcal B D\|_{C^0}=0.
\]
The higher regularity follows from Lemma~\ref{lem:barrier}.
\end{proof}
}

%%%%%%%%%%%%%%%%%%%%%%%%%%%%%%%%%%%%%%%%%%%%%%%%%%%%%%%%%%%%%%%
\subsection{Converting \Li bounds into change of coordinates}
\label{sec:stability}

We use the transfer estimate from Subsection~\ref{SSec:Transfer} to relate the \Li theorem to the actual change of coordinates.

\begin{theorem}[Coboundary control of coordinate change]
\label{prop:stability}
Fix $0<\theta<1$. Let $g\in\E_d^{1+\theta}$ range in a bounded $C^{1+\theta}$ family satisfying the common expansion bounds, and let $f\in\E_d^2$ range in a bounded $C^2$ family. For the orientation-preserving marking $\phi\circ g=f\circ\phi$, normalized by $\phi(0)=0$, we have
\[
 \|\phi-\id\|_{C^0}+\|f-g\|_{C^1}
 \le C\|\log f'\circ\phi-\log g'\|_{C^0}
\]
when the right-hand side is sufficiently small. The constant is uniform in these families.
\end{theorem}

\begin{proof}
Put $D:=\log f'\circ\phi-\log g'$. Let $\mathcal P$ be the transfer operator associated with $g$, as in Subsection~\ref{SSec:Transfer}. Since $[\log g']_\theta\le\Lambda^{-1}[g']_\theta$, the constants in Lemma~\ref{Lem:Transfer} are uniform for our family.

We first estimate the displacement of $\phi$ by comparing the lengths of intervals and their images. Define a probability measure $\nu$ by $\nu(A):=|\phi(A)|$ for Borel sets $A\subset{\Circle}$. For a bounded Borel function $v$, write $\nu(v):=\int_{{\Circle}}v\,d\nu$. Thus
\[
 \nu(v)=\int_{{\Circle}}v(\phi^{-1}(x))\,dx,
 \qquad
 \nu(1)=1,
 \qquad
 \nu(\mathbf 1_I)=|\phi(I)|.
\]

Let us derive the identity satisfied by this measure. The conjugacy gives a bijection between the preimages of $\phi^{-1}(x)$ under $g$ and the preimages of $x$ under $f$, namely $y\mapsto\phi(y)$. Moreover, the definition of $D$ gives
\[
 \frac{e^{-D(y)}}{g'(y)}=\frac{1}{f'(\phi(y))}.
\]
Substituting these relations and then changing variables under $f$, we obtain
\begin{equation}
\label{eq:marked-measure}
 \begin{aligned}
 \nu\bigl(\mathcal P(e^{-D}v)\bigr)
 &=\int_{{\Circle}}\sum_{g(y)=\phi^{-1}(x)}
       \frac{e^{-D(y)}v(y)}{g'(y)}\,dx\\
 &=\int_{{\Circle}}\sum_{f(t)=x}
       \frac{v(\phi^{-1}(t))}{f'(t)}\,dx\\
 &=\int_{{\Circle}}v(\phi^{-1}(t))\,dt
 =\nu(v).
 \end{aligned}
\end{equation}
The change of variables is performed only under $f$, so this calculation requires no differentiability of $\phi$.

Fix an interval $I\subset{\Circle}$ and put $w:=\mathbf 1_I-|I|$. By Lemma~\ref{Lem:Transfer},
\begin{equation}
 \label{Eq:Dec}
 \|\mathcal P^j w\|_{C^0}\le C\gamma^j,
 \qquad j\ge0,
\end{equation}
where $0<\gamma<1$ is uniform in the family. We will use this decay to compare $\nu(I)$ with $|I|$.

Apply \eqref{eq:marked-measure} to $v=\mathcal P^j w$ and subtract $\nu(\mathcal P^{j+1}w)$. By linearity,
\[
 \nu(\mathcal P^j w)-\nu(\mathcal P^{j+1}w)
 =\nu\bigl(\mathcal P((e^{-D}-1)\mathcal P^j w)\bigr).
\]
Summing over $j=0,\ldots,n-1$, the intermediate terms cancel, giving
\[
 \nu(w)-\nu(\mathcal P^n w)
 =\sum_{j=0}^{n-1}\nu\bigl(\mathcal P((e^{-D}-1)\mathcal P^j w)\bigr).
\]

We now bound each summand. Assume $\|D\|_{C^0}\le1$; then $\|e^{-D}-1\|_{C^0}\le e\|D\|_{C^0}$. Positivity of $\mathcal P$ and \eqref{Eq:Dec} then imply
\[
 \begin{aligned}
 \left|\nu\bigl(\mathcal P((e^{-D}-1)\mathcal P^j w)\bigr)\right|
 &\le e \cdot  \|D\|_{C^0} \cdot\,\|\mathcal P^j w\|_{C^0} \cdot \,\nu(\mathcal P1) \le C\|D\|_{C^0} \cdot \gamma^j.
 \end{aligned}
\]
Here we used $\nu(\mathcal P1)\le\|\mathcal P1\|_{C^0}\le d/\Lambda$. Consequently,
\[
 |\nu(w)-\nu(\mathcal P^n w)|
 \le C\|D\|_{C^0} \sum_{j=0}^{n-1}\gamma^j
 \le\frac{C\|D\|_{C^0}}{1-\gamma}.
\]
Also $|\nu(\mathcal P^n w)|\le C\gamma^n\to0$ as $n\to\infty$. Taking the limit and recalling that $\nu(w)=\nu(I)-|I| = |\phi(I)|-|I|$, we obtain $\bigl||\phi(I)|-|I|\bigr|\le C\|D\|_{C^0}$. Applying the last estimate to $I=[0,x]$ proves
\begin{equation}
\label{eq:marked-displacement}
 \|\phi-\id\|_{C^0}\le C\|D\|_{C^0}.
\end{equation}

It remains to estimate the difference between the maps. The identity $f'(\phi(x))=g'(x)e^{D(x)}$ and \eqref{eq:marked-displacement} give
\[
 \begin{aligned}
 |f'(x)-g'(x)|
 &\le |f'(x)-f'(\phi(x))|
       +|f'(\phi(x))-g'(x)|\\
 &\le \|f''\|_{C^0}|\phi(x)-x|
       +g'(x)|e^{D(x)}-1|\\
 &\le C \|D\|_{C^0}.
 \end{aligned}
\]
Thus, $\|f - g\|_{C^1} \le C \|D\|_{C^0}$. Together with \eqref{eq:marked-displacement}, this proves the theorem.
\end{proof}

%%%%%%%%%%%%%%%%%%%%%%%%%%%%%%%%%%%%%%%%%%%%%%%%%%%%%%%%%%%%%%%%%%%%%%%%%%%%%
\section{Proofs of Main Theorems \ref{Thm:Prev:Main} and \ref{Thm:LLSR}}
\label{Sec:Iteration}\label{Sec:Main}
%%%%%%%%%%%%%%%%%%%%%%%%%%%%%%%%%%%%%%%%%%%%%%%%%%%%%%%%%%%%%%%%%%%%%%%%%%%%%

We first prove Theorem~\ref{Thm:Rig}. The reference map $g$ remains fixed, and at each step we replace $f$ by $h^{-1}\circ f\circ h$, where $h$ is a small change of coordinates. We need to improve the $C^1$ distance to $g$, control the $C^k$ norm of the new map, and make the successive changes of coordinates summable. We choose the parameters according to these three requirements. At the end, we combine this rigidity result with Theorem~\ref{Thm:Prev} to prove Main Theorem~\ref{Thm:Prev:Main}.

\subsection{The scales before the iteration}
\label{SSec:Choice}

Let $g\in\E_d^k$ have a $\beta$-sparse spectrum and satisfy~\eqref{eq:rigidity-budget}. Since this inequality is strict, we can choose the common expansion bounds so that $\Lambda<\min g'$, $\max g'<\Omega$, and
\[
 k-1>\max\{1,2\beta a\},\qquad a:=\frac{\log\Omega}{\log\Lambda}.
\]
Every sufficiently $C^1$-close map then satisfies~\eqref{eq:expansion}. Constants below may depend on $g$, its sparsity constant, the expansion bounds, and the parameters fixed in this subsection, but not on the step of the iteration.

Suppose that $f$ has the same unmarked spectrum as $g$ and $\|f-g\|_{C^1}\le\eps$. For a fixed $0<\tau<1$, Lemma~\ref{Lem:Recovering} recovers the marked lengths through period $N=\lfloor(1-\tau)|\log \eps|/(\beta\log\Omega)\rfloor$. The periodic sums in Theorem~\ref{thm:finite-livsic} have periods $2n$ and $2n+1$, so we take $n=\lfloor(N-1)/2\rfloor$. The two integer parts give $n\ge(1-\tau)\log(1/\eps)/(2\beta\log\Omega)-2$, and hence
\begin{equation}
\label{eq:heuristic-rate}
 \Lambda^{-(k-1)n}\le\Lambda^{2(k-1)}\eps^{(1-\tau)(k-1)/(2\beta a)}.
\end{equation}
We will take $\eps$ sufficiently small that $N\ge3$. Since the logarithmic derivative of a $C^k$ map belongs to $C^{k-1}$, we use Theorem~\ref{thm:finite-livsic} at the integer order $k-1$. The inequality $k-1>2\beta a$ ensures that, before the losses below, the exponent in~\eqref{eq:heuristic-rate} is greater than one.

We now choose the H\"older exponent needed for Lemma~\ref{Lem:Holder}. Let
\[
 0<\theta<\min\{1/4,k-2\},\qquad \Omega^\theta<\Lambda.
\]
The marking displacement will be small in $C^\theta$. Composing a $C^\theta$ function with a $C^\theta$ map gives a $C^{\theta^2}$ function, so we use $\theta^2$ for the nonsmooth part of the observable. This allows us to work with a uniform $C^{2+\theta}$ bound for $f$.

The parameter $\tau$ controls both the Whitney scale and the permitted growth of the top norm:
\begin{equation}
\label{eq:scales}
 \frac12\eps^\tau\le\eta\le\eps^\tau,\qquad
 \|f\|_{C^k}\le\eps^{-\tau(k-1)},\qquad
 N=\left\lfloor\frac{(1-\tau)|\log\eps|}{\beta\log\Omega}\right\rfloor.
\end{equation}
Here $\eta=1/\lceil\eps^{-\tau}\rceil$ is a uniform grid spacing, as in Lemma~\ref{lem:whitney}. We choose $0<\tau<1/(k-1)$ sufficiently small that
\begin{equation}
\label{eq:tau-choice}
 \frac{1+\theta}{k-1}+\tau(1+\theta)<1,
 \qquad
 \frac{(1-\tau)(k-1)}{2\beta a}-\tau(k-1)>1+\tau\theta^2.
\end{equation}
The first inequality will give a uniform $C^{2+\theta}$ bound by interpolation. The second will make the exponential remainder in Theorem~\ref{thm:finite-livsic} smaller than the Whitney approximation error. At $\tau=0$, these inequalities reduce to $1+\theta<k-1$ and $k-1>2\beta a$, which we have already arranged.

Finally, put $\sigma:=\tau\theta^2/2$ and $\eps_+:=\eps^{1+\sigma}$. We will obtain an error bounded by $C|\log\eps|\eps^{1+2\sigma}$. For sufficiently small $\eps$, the inequality $C|\log\eps|\eps^\sigma\le1$ absorbs both the constant and the logarithm, giving the desired bound $\eps_+$.

\subsection{One improving step}

\begin{prop}[Inductive step]
\label{prop:step}
With $g,\theta,\tau,\sigma$ fixed above, suppose that $f\in\E_d^k$ has the same unmarked length spectrum as $g$ and satisfies
\[
 \|f-g\|_{C^1}\le\eps,\qquad \|f\|_{C^k}\le\eps^{-\tau(k-1)}.
\]
For all sufficiently small $\eps$, there is a normalized $C^k$ diffeomorphism $h$ such that $f_+:=h^{-1}\circ f\circ h$ satisfies
\[
 \|f_+-g\|_{C^1}\le\eps_+,\qquad
 \|f_+\|_{C^k}\le\eps_+^{-\tau(k-1)},\qquad
 \eps_+=\eps^{1+\sigma}.
\]
Moreover, for a fixed $\kappa>0$ independent of the step,
\[
 \|h-\id\|_{C^2}+\|h^{-1}-\id\|_{C^2}\le C\eps^\kappa.
\]
\end{prop}

\begin{proof}
We first obtain the uniform lower-order bounds needed in the proof. Since $g$ is fixed and $\eps\le1$, we have $\|f-g\|_{C^k}\le C\eps^{-\tau(k-1)}$. Applying~\eqref{eq:interpolation} with $j=1$, $t=1+\theta$, and $s=k-1$, we get
\begin{equation}
\label{eq:f-low}
 \begin{aligned}
 \|f-g\|_{C^{2+\theta}}
 &\le C\|f-g\|_{C^1}^{1-(1+\theta)/(k-1)} \|f-g\|_{C^k}^{(1+\theta)/(k-1)}\\
 &\le C\eps^{1-(1+\theta)/(k-1)-\tau(1+\theta)}.
 \end{aligned}
\end{equation}
The exponent is positive by~\eqref{eq:tau-choice}. Thus $f$ stays in a bounded $C^{2+\theta}$ family, even though its $C^k$ norm is allowed to grow.

Let $\phi\circ g=f\circ\phi$ be the normalized marking. We split its logarithmic derivative error into a smooth part and a small nonsmooth part:
\[
 D:=\log f'\circ\phi-\log g'
   =\log f'-\log g'+R,\qquad
 R:=\log f'\circ\phi-\log f'.
\]
We obtain
\begin{equation}
\label{eq:R-holder}
 \|R\|_{C^{\theta^2}}\le C\eps.
\end{equation}
Indeed, putting $v:=\phi-\id$\,\,, Lemma~\ref{Lem:Holder} gives $\|v\|_{C^\theta}\le C\eps$. The fundamental theorem of calculus gives
\[
 R(x)=v(x)\int_0^1(\log f')'(x+t v(x))\,dt.
\]
By~\eqref{eq:f-low}, the function $(\log f')'=f''/f'$ has a uniformly bounded $C^\theta$ norm. The maps $x\mapsto x+t v(x)$ have uniformly bounded $C^\theta$ seminorms, for $0\le t\le1$. For a $C^\theta$ function $a$ and a $C^\theta$ map $b$, we have
\[
 |a(b(x))-a(b(y))|
 \le[a]_\theta[b]_\theta^\theta\dist(x,y)^{\theta^2}.
\]
Consequently, the integral in the formula for $R$ has a uniformly bounded $C^{\theta^2}$ norm. The product estimate~\eqref{eq:barrier-holder-products}, together with $\|v\|_{C^{\theta^2}}\le C\eps$, proves~\eqref{eq:R-holder}.

We now regularize only $R$. Lemma~\ref{lem:whitney}, with exponent $\theta^2$ and grid spacing $\eta$ from~\eqref{eq:scales}, gives a smooth function $\widetilde R$. Put $\widetilde D:=\log f'-\log g'+\widetilde R$. We will use the following three bounds:
\begin{equation}
\label{eq:Deta}
 \begin{aligned}
 \|D-\widetilde D\|_{C^0}&\le C\eps^{1+2\sigma},\\
 \|\widetilde D\|_{C^{\theta^2}}&\le C\eps^{1-\theta^2},\\
 \|\widetilde D\|_{C^{k-1}}&\le C\eps^{-\tau(k-1)}.
 \end{aligned}
\end{equation}
For the first bound,~\eqref{eq:W-error} and~\eqref{eq:R-holder} give $\|R-\widetilde R\|_{C^0}\le C\eta^{\theta^2}\eps\le C\eps^{1+\tau\theta^2}$. This is the required exponent because $2\sigma=\tau\theta^2$.

For the second bound, we first estimate the smooth difference. The expansion bounds and~\eqref{eq:f-low} give $\|\log f'-\log g'\|_{C^0}\le C\eps$ and $\|\log f'-\log g'\|_{C^1}\le C$. Interpolation~\eqref{eq:interpolation}, with $j=0$, $s=1$, and $t=\theta^2$, yields
\[
 \|\log f'-\log g'\|_{C^{\theta^2}}
 \le C\|\log f'-\log g'\|_{C^0}^{1-\theta^2}
        \|\log f'-\log g'\|_{C^1}^{\theta^2}
 \le C\eps^{1-\theta^2}.
\]
By~\eqref{eq:W-low}, we also have $\|\widetilde R\|_{C^{\theta^2}}\le C\eps$. Adding these estimates proves the second bound in~\eqref{eq:Deta}.

For the top norm, the composition estimate in Lemma~\ref{Lem:Computations}, applied to the logarithm on the fixed interval $[\Lambda,\Omega]$, gives
\[
 \|\log f'-\log g'\|_{C^{k-1}}
 \le C\bigl(\|f\|_{C^k}+\|g\|_{C^k}\bigr)
 \le C\eps^{-\tau(k-1)}.
\]
Here the $C^1$ norms of $f'$ and $g'$ are uniformly bounded by~\eqref{eq:f-low}. Any fixed contribution from the logarithm is absorbed into $\|f\|_{C^k}\ge d$. On the other hand,~\eqref{eq:W-high} gives
\[
 \|\widetilde R\|_{C^{k-1}}
 \le C\eps\eta^{\theta^2-(k-1)}
 \le C\eps^{1+2\sigma-\tau(k-1)}.
\]
This proves the last bound in~\eqref{eq:Deta}. Moreover, the last displayed quantity tends to zero as $\eps\to0$, since $\tau(k-1)<1$. Thus the regularized nonsmooth part has a decaying top norm.

We next estimate the periodic error. Let $N$ be as in~\eqref{eq:scales} and $n=\lfloor(N-1)/2\rfloor$. Lemma~\ref{Lem:Recovering} gives, for $z\in\mathbb P_p^g$ and $p\le N$,
\[
 \sum_{i=0}^{p-1}D(g^i(z))
 =\lambda_p^f(\phi(z))-\lambda_p^g(z)=0.
\]
It follows that $\bigl|\sum_{i=0}^{p-1}\widetilde D(g^i(z))\bigr|\le p\|\widetilde D-D\|_{C^0}$. Since the periods in the definition~\eqref{eq:finite-periodic-error} are at most $2n+1\le N$, we obtain
\[
 \epsilon_n(\widetilde D)
 \le N\|\widetilde D-D\|_{C^0}
 \le C|\log\eps|\eps^{1+2\sigma}.
\]
Only the $C^0$ approximation error enters this estimate; the higher derivatives of $\widetilde D$ do not affect the periodic sums.

Put $u:=\mathcal B\widetilde D$. Applying Theorem~\ref{thm:finite-livsic} at order $k-1$, and then using~\eqref{eq:heuristic-rate} and~\eqref{eq:Deta}, we get
\begin{equation}
\label{eq:step-residual}
 \begin{aligned}
 \|D-u\circ g+u\|_{C^0}
 &\le\|D-\widetilde D\|_{C^0} +C\epsilon_n(\widetilde D) +C\Lambda^{-(k-1)n}\|\widetilde D\|_{C^{k-1}}\\
 &\le C\left(|\log\eps|\eps^{1+2\sigma} +\eps^{(1-\tau)(k-1)/(2\beta a)-\tau(k-1)}\right)\\
 &\le C|\log\eps|\eps^{1+2\sigma}.
 \end{aligned}
\end{equation}
The last inequality is precisely where we use the second condition in~\eqref{eq:tau-choice}.

We also need the correction itself to be small. Lemma~\ref{lem:barrier} and~\eqref{eq:Deta} give $\|u\|_{C^0}\le C\eps^{1-\theta^2}$ and $\|u\|_{C^{k-1}}\le C\eps^{-\tau(k-1)}$. Applying~\eqref{eq:interpolation} with $j=0$, $t=1$, and $s=k-1$, we obtain
\begin{equation}
\label{eq:u-low}
 \begin{aligned}
 \|u\|_{C^1}
 &\le C\|u\|_{C^0}^{(k-2)/(k-1)}\|u\|_{C^{k-1}}^{1/(k-1)}
 \le C\eps^\kappa,\\
 \kappa&:=\frac{(1-\theta^2)(k-2)}{k-1}-\tau>0.
 \end{aligned}
\end{equation}
Indeed, the first inequality in~\eqref{eq:tau-choice} gives $\tau<1/(1+\theta)-1/(k-1)$, and this upper bound is smaller than $(1-\theta^2)(k-2)/(k-1)$ for $0<\theta<1/4$.

Finally, we now turn $u$ into a change of coordinates. Define
\begin{equation}
\label{eq:h}
 h(x):={\displaystyle\int_0^x e^{u(t)}\,dt}/
                 {\displaystyle\int_0^1 e^{u(t)}\,dt}.
\end{equation}
This defines an orientation-preserving $C^k$ diffeomorphism of ${\Circle}$ fixing $0$. The normalization integral lies between $e^{-\|u\|_{C^0}}$ and $e^{\|u\|_{C^0}}$. Since
\[
 h'-1=\frac{e^u-\int_0^1e^{u(t)}\,dt}{\int_0^1e^{u(t)}\,dt},
 \qquad h''=h'u',
\]
we get $\|h'-1\|_{C^0}\le C\|u\|_{C^0}$ and $\|h''\|_{C^0}\le C\|u'\|_{C^0}$. Using~\eqref{eq:u-low}, integrating from $0$, and then using the inverse estimate~\eqref{Eq:Displ}, we obtain
\begin{equation}
\label{eq:step-coordinate}
 \|h-\id\|_{C^2}+\|h^{-1}-\id\|_{C^2}\le C\eps^\kappa.
\end{equation}
For the top norm,~\eqref{Eq:Exp} at order $k-1$ yields
\[
 \|h-\id\|_{C^k}
 \le C\|e^u-1\|_{C^{k-1}}
 \le C\|u\|_{C^{k-1}}
 \le C\eps^{-\tau(k-1)}.
\]
The first inequality follows from the same formula for $h'-1$ and $h(0)=0$. Thus one integration in~\eqref{eq:h} restores the derivative lost when passing from the maps to their logarithmic derivatives.

Let us compare $f$ with $h\circ g\circ h^{-1}$. Their marking is $\phi\circ h^{-1}$, and the chain rule gives
\begin{equation}
\label{eq:exact-marked-residual}
 \bigl(\log f'\circ\phi\circ h^{-1}
       -\log(h\circ g\circ h^{-1})'\bigr)\circ h
 =D-u\circ g+u.
\end{equation}
Indeed, $(h\circ g\circ h^{-1})'(h(x))=g'(x)h'(g(x))/h'(x)$, while $\log h'=u-\log\int_0^1e^{u(t)}\,dt$. The constant cancels in the difference of the two logarithmic derivatives of $h$, proving~\eqref{eq:exact-marked-residual}.

Let us check the uniform bounds required for Theorem~\ref{prop:stability}. Since $g$ is fixed in $C^2$, the chain rule, the mean value theorem, and~\eqref{eq:step-coordinate} give
\[
 \|h\circ g\circ h^{-1}-g\|_{C^1}\le C\|h-\id\|_{C^1}\le C\eps^\kappa,
 \qquad \|h\circ g\circ h^{-1}\|_{C^2}\le C.
\]
The first estimate and the strict expansion bounds for $g$ give the common expansion bounds for $h\circ g\circ h^{-1}$ when $\eps$ is sufficiently small. The second gives a uniformly bounded $C^{1+\theta}$ family. Also, $f$ has a uniform $C^2$ bound by~\eqref{eq:f-low}. We can therefore use Theorem~\ref{prop:stability}. Together with~\eqref{eq:step-residual} and~\eqref{eq:exact-marked-residual}, it gives
\[
 \|f-h\circ g\circ h^{-1}\|_{C^1}
 \le C|\log\eps|\eps^{1+2\sigma}.
\]
Conjugating back costs only a uniform constant in $C^1$. To see this, let $a,b$ have bounded first derivatives. The chain rule gives
\[
 \begin{aligned}
 &(h^{-1}\circ a\circ h-h^{-1}\circ b\circ h)'\\
 &\quad=\bigl((h^{-1})'\circ a\circ h-(h^{-1})'\circ b\circ h\bigr)(a'\circ h)h'
       +((h^{-1})'\circ b\circ h)((a'-b')\circ h)h'.
 \end{aligned}
\]
The mean value theorem and the uniform $C^2$ bounds for $h,h^{-1}$ bound this expression by $C\|a-b\|_{C^1}$; the uniform difference is bounded in the same way. Taking $a=f$ and $b=h\circ g\circ h^{-1}$, we obtain
\[
 \|f_+-g\|_{C^1}
 \le C|\log\eps|\eps^{1+2\sigma}
 \le\eps^{1+\sigma}=\eps_+.
\]
The last inequality holds once $C|\log\eps|\eps^\sigma\le1$.

It remains to check the top norm. By the tame estimate~\eqref{eq:tame} in Lemma~\ref{Lem:Computations},
\[
 \|f_+\|_{C^k}
 \le C\bigl(\|f\|_{C^k}+\|h-\id\|_{C^k}\bigr)
 \le C\eps^{-\tau(k-1)}
 \le\eps_+^{-\tau(k-1)}.
\]
For the last inequality, we require $C\eps^{\tau(k-1)\sigma}\le1$. These conditions only decrease the allowed initial value of $\eps$; all constants have already been fixed. The new map is normalized and expanding, and Lemma~\ref{Lem:SimpleProp} shows that its unmarked spectrum is unchanged. Its marking is $h^{-1}\circ\phi$, since $(h^{-1}\circ\phi)\circ g=f_+\circ(h^{-1}\circ\phi)$. This completes the inductive step.
\end{proof}

\subsection{Convergence and full regularity}

\begin{proof}[Proof of Theorem~\ref{Thm:Rig}]
We choose the expansion bounds and the parameters as above. Next, we fix $\eps_0>0$ sufficiently small for Proposition~\ref{prop:step}, and such that $\eps_0^{-\tau(k-1)}>\|g\|_{C^k}$. A sufficiently small $C^k$ neighborhood of $g$ then ensures that $f_0:=f$ satisfies both starting bounds. We further decrease $\eps_0$ below when needed.

Let us construct successively $f_{j+1}:=h_j^{-1}\circ f_j\circ h_j$ and $\eps_{j+1}:=\eps_j^{1+\sigma}$ using Proposition~\ref{prop:step}. Lemma~\ref{Lem:SimpleProp} guarantees $\Lyap(f_j) = \Lyap(g)$ at each step. Thus
\[
 \|f_j-g\|_{C^1}\le\eps_j,\qquad
 \|f_j\|_{C^k}\le\eps_j^{-\tau(k-1)},\qquad
 \eps_j=\eps_0^{(1+\sigma)^j}.
\]
Moreover,~\eqref{eq:step-coordinate} gives
\[
 \sum_{j\ge0}\bigl(\|h_j-\id\|_{C^2}+\|h_j^{-1}-\id\|_{C^2}\bigr)
 \le C\sum_{j\ge0}\eps_0^{\kappa(1+\sigma)^j}<\infty.
\]
Indeed, $(1+\sigma)^j\ge1+j\sigma$, so the last series is bounded by the geometric series with terms $\eps_0^\kappa(\eps_0^{\kappa\sigma})^j$.

To establish convergence of the coordinate changes, put $\phi_0:=\id$ and $\phi_j:=h_0\circ\cdots\circ h_{j-1}$ for $j\ge1$. These maps satisfy $\phi_j\circ f_j=f\circ\phi_j$. Let $a_j:=C\eps_j^\kappa$, with $C$ large enough to bound the $C^2$ displacements of both $h_j$ and its inverse. The chain rule gives
\[
 \phi_{j+1}'=(\phi_j'\circ h_j)h_j',\qquad
 \phi_{j+1}''=(\phi_j''\circ h_j)(h_j')^2+(\phi_j'\circ h_j)h_j''.
\]
Consequently,
\[
 \begin{aligned}
 \|\phi_{j+1}'\|_{C^0}&\le(1+a_j)\|\phi_j'\|_{C^0},\\
 \|\phi_{j+1}''\|_{C^0}&\le(1+a_j)^2\|\phi_j''\|_{C^0}+a_j\|\phi_j'\|_{C^0}.
 \end{aligned}
\]
Since $\sum_j a_j<\infty$, iteration of these inequalities, starting from $\phi_0=\id$\,\,, yields
\[
 \|\phi_j'\|_{C^0}\le\exp\Bigl(\sum_{i\ge0}a_i\Bigr),\qquad
 \|\phi_j''\|_{C^0}\le\exp\Bigl(3\sum_{i\ge0}a_i\Bigr)\sum_{i<j}a_i.
\]
In particular, both derivative norms are uniformly bounded. By decreasing $\eps_0$, we may also assume $a_j<1/2$. The derivatives then stay uniformly away from zero, because
\[
 \inf\phi_j'\ge\prod_{i<j}(1-a_i)
 \ge\exp\Bigl(-2\sum_{i\ge0}a_i\Bigr)>0.
\]

The mean value theorem and the preceding $C^2$ bound give
\[
 \|\phi_{j+1}-\phi_j\|_{C^1}
 =\|\phi_j\circ h_j-\phi_j\|_{C^1}\le C a_j.
\]
For the derivative, this follows by writing $\phi_{j+1}'-\phi_j'=(\phi_j'\circ h_j-\phi_j')h_j'+\phi_j'(h_j'-1)$; the uniform difference is estimated directly. Hence $\phi_j$ converges in $C^1$ to a normalized diffeomorphism $\phi$. Passing to the limit in $\phi_j\circ f_j=f\circ\phi_j$, we obtain $\phi\circ g=f\circ\phi$. Thus the limit is the marking conjugacy.

The full $C^k$ regularity follows from Shub--Sullivan~\cite{SS}. It also follows directly from Lemma~\ref{lem:barrier}. Indeed, put $v:=\log\phi'$, which is initially continuous. Differentiating the conjugacy gives
\[
 v\circ g-v=\log f'\circ\phi-\log g'=:D.
\]
The exact identity in Lemma~\ref{lem:barrier} gives $v-v(0)=\mathcal B D$. If $\phi\in C^j$ for $1\le j\le k-1$, then $D\in C^j$, so the same lemma gives $v\in C^j$ and therefore $\phi\in C^{j+1}$. Starting at $j=1$ proves $\phi\in C^k$. Since $\phi'$ is positive, its inverse has the same regularity.
\end{proof}

\begin{proof}[Proof of Main Theorem~\ref{Thm:Prev:Main}]
Fix $d,\Lambda,\Omega$ as in the statement and put $m:=\lceil\sqrt{k}\rceil$. We have
\[
 8(m+1)+\frac{8k+14}{m}\le16\sqrt{k}+16+\frac{14}{\sqrt{k}}.
\]
Thus the choice
\[
 \beta:=\frac{(16\sqrt{k}+16+14/\sqrt{k})\log\Omega+2\log d+1}{\log\Lambda}
\]
satisfies~\eqref{eq:existence-budget}. For fixed degree and expansion bounds, this quantity grows at most as a constant times $\sqrt{k}$. Hence there is $k_0=k_0(d,\Lambda,\Omega)$ such that, for every $k\ge k_0$,
\[
 k-1>2\beta\frac{\log\Omega}{\log\Lambda}.
\]
Theorem~\ref{Thm:Prev} provides a prevalent Borel subset of $\mathcal U$ with $\beta$-sparse spectra. Every map in this subset satisfies~\eqref{eq:rigidity-budget}, since $\log\max g'/\log\min g'<\log\Omega/\log\Lambda$. Theorem~\ref{Thm:Rig} therefore gives local rigidity, with a neighborhood depending on the reference map.

Finally, prevalence implies density here. Indeed, if a shy Borel set contained an open ball, the compact support of its witnessing measure could be covered by finitely many translates of that ball. Each translate would have measure zero by the definition of shyness, contradicting that the measure has total mass one. Thus the complement of the prevalent subset contains no nonempty open subset of $\mathcal U$.
\end{proof}

\begin{remark}[An explicit choice of parameters]
\label{Rem:Example}
One sufficient, unoptimized choice is $d=2$, $\Lambda=1.9$, $\Omega=2.1$, $k=10000$, $m=100$, and $\beta=3000$. The lower bound for $\beta$ in~\eqref{eq:existence-budget} is less than $1862$, whereas
\[
 2\beta\frac{\log2.1}{\log1.9}<6936<9999=k-1.
\]
Thus the sparsity and rigidity conditions hold simultaneously at a finite regularity.
\end{remark}

%%%%%%%%%%%%%%%%%%%%%%%%%%%%%%%%%%%%%%%%%%%%%%%%%%%%%%%%%%%%%%%%%%%%%%%
\section{Counterexamples to general rigidity (Proof of Main Theorem~\ref{Thm:Counter})}
\label{Sec:CounterExample}
%%%%%%%%%%%%%%%%%%%%%%%%%%%%%%%%%%%%%%%%%%%%%%%%%%%%%%%%%%%%%%%%%%%%%%%%%

We illustrate two mechanisms proving the assertions of Main Theorem~\ref{Thm:Counter}. The first proves, for every degree $d\ge2$, the existence of arbitrarily close isospectral maps whose marking is not $C^1$; both maps can be chosen arbitrarily close to $L_d$. It reverses the cyclic order and retains a smooth conjugacy in the other orientation. The second proves the stronger conclusion excluding smooth conjugacies in either orientation. We give an explicit example in degree $4$. This example can be easily extended to any composite degree. 

\subsection{First mechanism: reversing the cyclic order}

\begin{figure}[htbp]
\centering
\includegraphics[width=0.5\textwidth,trim={10 10 10 10},clip]{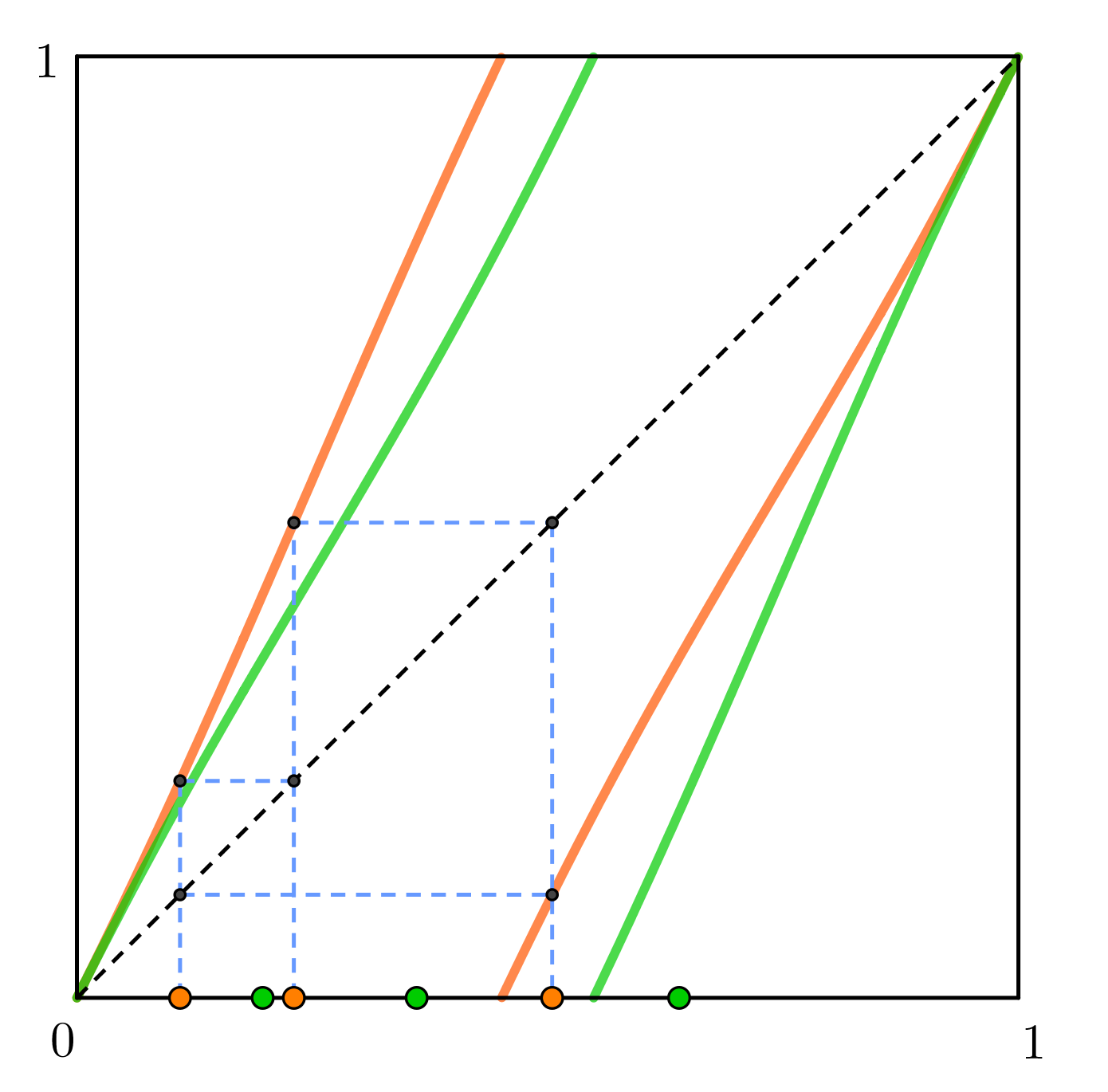}
\caption{Two length isospectral expanding circle maps with no smooth orientation-preserving conjugacy. The graph of $f(x)=2x+0.1\sin^2(\pi x)$ modulo one is shown in orange, and that of $g(x)=-f(-x)=2x-0.1\sin^2(\pi x)$ modulo one in green. The log-multipliers of the points with symbolic sequence $\overline{001}$ are approximately $2.31$ for $f$ (orange points) and $1.90$ for $g$ (green points).}
\label{Fig:Example}
\end{figure}

\begin{proof}[Proof of the first assertions of Main Theorem~\ref{Thm:Counter}]
We give the construction for every $d\ge2$ and every integer $k\ge1$. Pick $f$ to be $C^k$-$\eps/2$-close to the linear map $L_d\colon x\mapsto dx\pmod1$, normalized so that $f(0)=0$. Furthermore, pick $f$ so that it is not odd, that is, $f(-x)\ne-f(x)$ for some $x$. We will specify this choice below so that two orbits with opposite symbolic codes have different multipliers.

Let $R\colon x\mapsto-x$ be the orientation-reversing diffeomorphism of ${\Circle}$, and define
\[
 g:=R\circ f\circ R^{-1}.
\]
Since $f$ is not odd, we have $g\ne f$. Moreover, $R$ commutes with $L_d$, and therefore
\[
 \|g-L_d\|_{C^k}=\|f-L_d\|_{C^k},\qquad
 \|f-g\|_{C^k}\le2\|f-L_d\|_{C^k}<\eps.
\]
In this way, $g$ is an orientation-preserving degree-$d$ expanding map of the circle and $g(0)=0$. By construction, $\Lyap_n(f)=\Lyap_n(g)$ for every $n\in\mathbb N$, since smooth conjugacies preserve multipliers regardless of their orientation; see Lemma~\ref{Lem:SimpleProp}.

We now make the required choice of $f$. For the linear map $L_d$, consider the period-three points
\[
 z:=\frac1{d^3-1},\qquad w:=1-z.
\]
Their symbolic codes are obtained by repeating the words $001$ and $(d-1)(d-1)(d-2)$, respectively. Thus they are opposite: each digit $a\in\{0,\ldots,d-1\}$ is replaced by $d-1-a$. Their orbits are disjoint. Indeed, for $d\ge3$ the orbit $\{z,dz,d^2z\}$ lies in $(0,1/2)$ and its reflection lies in $(1/2,1)$; for $d=2$ their numerators are $\{1,2,4\}$ and $\{3,5,6\}$ over the common denominator $7$.

Choose a smooth periodic function $\chi$ supported in a small interval about $z$ containing no other point of the grid $\{j/(d^3-1):0\le j<d^3-1\}$, with $\chi(z)=0$ and $\chi'(z)=1$. We take
\[
 f(x):=dx+t\chi(x)\pmod1
\]
for sufficiently small $t>0$. This map is expanding, fixes $0$, and is arbitrarily close to $L_d$ in $C^k$ as $t\to0$. It agrees with $L_d$ at every point of the grid, so the two chosen orbits remain unchanged. Only the derivative at $z$ changes along these orbits. Hence their multipliers satisfy
\[
 \lambda:=(f^3)'(z)=d^2(d+t),\qquad
 \mu:=(f^3)'(w)=d^3,\qquad \lambda\ne\mu.
\]
In particular, this choice of $f$ is not odd. For small $t$, the two orbits retain their opposite symbolic codes.

It follows that $\widetilde z:=R(z)$ and $\widetilde w:=R(w)$ are periodic points for $g$ with multipliers $\lambda$ and $\mu$:
\[
 (g^3)'(\widetilde z)=\lambda,\qquad
 (g^3)'(\widetilde w)=\mu.
\]
Furthermore, since $R$ reverses the cyclic order, $\widetilde z$ has the opposite code to $z$, whereas $\widetilde w$ has the same code as $z$. We also have $0<\widetilde w<\widetilde z$.

Since the codes of $z$ and $\widetilde w$ are the same, any orientation-preserving conjugacy fixing $0$ between $f$ and $g$ must map $z$ to $\widetilde w$. This conjugacy cannot be a $C^1$ diffeomorphism because the multipliers at these points are different. The marking conjugacy is therefore not $C^1$. This proves the first assertion of Main Theorem~\ref{Thm:Counter}. The degree-two example in Figure~\ref{Fig:Example} illustrates the same mechanism.
\end{proof}

\subsection{Second mechanism: interchanging noninvertible factors}

Here we construct a pair of expanding circle maps with no $C^1$ conjugacy in either orientation. The mechanism is to interchange two noninvertible factors: $f=a\circ b$ and $g=b\circ a$. The factors induce bijections on the finite periodic sets, preserving their multipliers, even though these bijections need not preserve or reverse their cyclic order. We give the explicit degree-four example; this example can be easily extended to arbitrary composite degree.

\begin{proof}[Proof of the stronger assertion of Main Theorem~\ref{Thm:Counter} in degree $4$]
Choose a real analytic periodic function $\chi$ with
\begin{equation}
\label{eq:counterexample-jets}
 \chi(j/15)=0\quad(0\le j<15),\qquad
 \chi'(j/15)=\begin{cases}1,&j=1,\\0,&j\ne1.\end{cases}
\end{equation}
For example, with $y=x-1/15$, put
\[
 \chi(x)=\frac{\sin(2\pi y)}{2\pi}
 \left(\frac{\sin(15\pi y)}{15\sin(\pi y)}\right)^2.
\]
The quotient is the trigonometric polynomial $15^{-1}(1+2\sum_{j=1}^{7}\cos(2\pi jy))$. Its square is one at $y=0$ and has a double zero at every other point of the grid, proving~\eqref{eq:counterexample-jets}.

We define, modulo one,
\[
 \begin{aligned}
 a_t(x)&:=2x+t\chi(x),& b(x)&:=2x,\\
 f_t(x)&:=a_t(b(x))=4x+t\chi(2x),&
 g_t(x)&:=b(a_t(x))=4x+2t\chi(x).
 \end{aligned}
\]
For small $t>0$, these maps are expanding, fix $0$, and satisfy $\|f_t-g_t\|_{C^k}\le C_k t<\eps$. Both maps tend to $L_4$ in every fixed $C^k$ norm as $t\to0$.

This construction is the circle version of the elementary transformation considered in~\cite{Pakovich}. 

We first prove equality of their spectra; it follows from \cite[Lemma 2.1]{Pakovich}, and we present the proof for completeness. The identities $b\circ f_t=g_t\circ b$ and $f_t\circ a_t=a_t\circ g_t$ show that the factors send periodic points to periodic points. Although $b$ is not invertible on ${\Circle}$, it induces a bijection
\[
 b\colon\mathbb P_n^{f_t}\longrightarrow\mathbb P_n^{g_t}
\]
whose inverse is $y\mapsto a_t(g_t^{n-1}(y))$. Indeed, for $y\in\mathbb P_n^{g_t}$ and $x\in\mathbb P_n^{f_t}$,
\[
 \begin{aligned}
 b(a_t(g_t^{n-1}(y)))&=g_t^n(y)=y,\\
 a_t(g_t^{n-1}(b(x)))&=a_t(b(f_t^{n-1}(x)))=f_t^n(x)=x.
 \end{aligned}
\]
This bijection intertwines the orbit permutations and therefore preserves least periods. Differentiating $b\circ f_t^n=g_t^n\circ b$ at a periodic point $x$, we get
\[
 b'(x)(f_t^n)'(x)=(g_t^n)'(b(x))b'(x).
\]
Since $b'(x)=2$, the multipliers agree. Consequently, $\Lyap_n(f_t)=\Lyap_n(g_t)$ for every $n$, even when multiplicities are recorded. 

We now show that the periodic bijection cannot be replaced by a smooth circle conjugacy. Both maps agree with $L_4$ on $\{j/15:0\le j<15\}$. Since $4^2j\equiv j\pmod{15}$, all fifteen points are fixed by their second iterates. Lemma~\ref{Lem:Grid} gives exactly $4^2-1=15$ such points, so this is their full second-periodic set.

On this grid,
\[
 f_t'(j/15)=4+2t\chi'(2j/15),\qquad
 g_t'(j/15)=4+2t\chi'(j/15).
\]
By~\eqref{eq:counterexample-jets}, $g_t'$ differs from $4$ only at $1/15$. For $f_t'$, the exceptional rank satisfies $2j\equiv1\pmod{15}$, hence $j=8$. The exceptional derivative is $4+2t$ in both cases. Thus the unique two-cycles whose multipliers differ from $16$ are
\[
 \begin{array}{c|c|c}
 &\text{exceptional two-cycle}&\text{multiplier}\\ \hline
 f_t&\{2/15,8/15\}&16+8t\\
 g_t&\{1/15,4/15\}&16+8t.
 \end{array}
\]
Indeed, their cyclic ranks follow $2\mapsto8\mapsto2$ and $1\mapsto4\mapsto1$, respectively, and the multiplier is $4(4+2t)$.

A circle conjugacy must permute the fifteen-point grid, acting on its cyclic ranks by $j\mapsto j+s$ or $j\mapsto-j+s$ modulo $15$. A $C^1$ diffeomorphic conjugacy must also match the exceptional cycles. This is impossible: their rank differences are $\pm6$ and $\pm3$ modulo $15$. Hence neither orientation admits a $C^1$ conjugacy. 
\end{proof}

\appendix

%%%%%%%%%%%%%%%%%%%%%%%%%%%%%%%%%%%%%%%%%%%%%%%%%%%%%%%%%%%%%%%%%%%%%%%%%%%%%
\section{Prevalence of exponentially sparse spectra \\ (Proof of Main Theorem~\ref{Thm:Prev})}
\label{App:Prevalence}
%%%%%%%%%%%%%%%%%%%%%%%%%%%%%%%%%%%%%%%%%%%%%%%%%%%%%%%%%%%%%%%%%%%%%%%%%%%%%

Throughout this appendix, $k\ge3$, $1<\Lambda<d<\Omega$, and $m\ge1$ are fixed. We first consider families satisfying
\begin{equation}\label{eq:uniform-family}
 \Lambda\le g'\le\Omega,\qquad \|g'\|_{C^2}\le M_3.
\end{equation}
Constants may depend on these bounds and on $k,m$, but not on the periods. We will keep track of the exponential rates; the remaining dependence is included in the multiplicative constants.

\subsection{A small cluster and variations of periodic lengths}

By Lemma~\ref{lem:marking}, distinct periodic points of periods at most $N$, including the fixed point $0$, satisfy
\begin{equation}\label{eq:point-separation}
 \operatorname{dist}(x,y)\ge c\Omega^{-2N}.
\end{equation}
We first find a point separated from all but a bounded number of points at a larger scale.

\begin{lemma}[Extreme neighbours lemma]
\label{lem:cluster}
Let $S$ be a $g$-invariant set of $M$ periodic points, where $2\le M\le2N$. There are a nonempty subset $C\subset S$, a point $z\in C$, and a number $\delta>0$ such that
\[
 |C|\le m,\qquad \operatorname{dist}(z,S\setminus C)\ge\delta,\qquad \delta\ge\frac{c}{N}\Omega^{-2N/m}.
\]
The distance condition is void if $C=S$.
\end{lemma}

We call the elements of $C$ the \emph{extreme neighbours} of $z$, including $z$ itself. Thus every point of $S$ at distance less than $\delta$ from $z$ belongs to $C$. The estimate \eqref{eq:point-separation} separates distinct extreme neighbours, whereas $\delta$ separates $z$ from the points outside $C$.

\begin{proof}
We choose an integer $J_0>2\Omega$ and partition ${\Circle}$ into $J_0$ equal intervals. We can translate this partition so that every point of $S$ is at distance at least $\rho:=(4J_0M)^{-1}$ from its boundaries. Indeed, in the interval $[0,1/J_0)$ of translation parameters, one point excludes a set of length at most $2\rho$. The total excluded length is therefore at most $2M\rho=1/(2J_0)<1/J_0$.

This partition distinguishes infinite itineraries. Indeed, if two points have the same itinerary, their distance at every iterate is at most $1/J_0$. The image of the short arc joining them has length at most $\Omega/J_0<1/2$. Consequently, their circle distance increases by a factor at least $\Lambda$ at every iterate. This is impossible for distinct points whose iterated distances remain bounded by $1/J_0$.

Let $p(L)$ be the number of length-$L$ words realized by $S$, with $p(0)=1$. We obtain
\[
 p(L)\ge\min\{L+1,M\}.
\]
Indeed, $p(L+1)\ge p(L)$. If equality holds, every realized length-$L$ word has a unique one-symbol extension. Deleting its first symbol gives another realized word, since $g(S)=S$. Repeating the unique extension therefore determines the whole infinite itinerary and hence a single point. Thus equality is possible only when $p(L)=M$. Before reaching $M$, the integer $p(L)$ increases by at least one at each step, proving the estimate.

We now take $L=\lceil M/m\rceil$. Some length-$L$ word is realized by at most $m$ points: if $p(L)=M$, every class is a singleton; otherwise the average class size is $M/p(L)\le M/(L+1)\le m$. Let $C$ be such a class and choose $z\in C$.

We set $\delta:=\rho\Omega^{-L}$. If $y\in S$ satisfies $\operatorname{dist}(y,z)<\delta$, then
\[
 \operatorname{dist}(g^j(y),g^j(z))\le\Omega^j\operatorname{dist}(y,z)<\rho,\qquad 0\le j<L.
\]
Since every $g^j(z)$ is at distance at least $\rho$ from the partition boundaries, $y$ and $z$ have the same length-$L$ word. Hence $y\in C$. Finally,
\[
 \delta=\frac{\Omega^{-\lceil M/m\rceil}}{4J_0M}\ge\frac{1}{8J_0\Omega N}\Omega^{-2N/m},
\]
which proves the lemma.
\end{proof}

For $x\in\mathbb P_p^g$ and $y\in\mathbb P_q^g$ on distinct orbits, with $p,q\le N$, we use Lemma~\ref{lem:cluster} with
\[
 S:=\{g^i(x):0\le i<p\}\cup\{g^j(y):0\le j<q\}.
\]
Repeated traversals introduce no new points, so $2\le|S|\le2N$. We do not add $0$ to $S$: the selected point $z$ must belong to one of the two compared orbits. The perturbations constructed below will satisfy the normalization $v(0)=0$ separately.

\begin{lemma}\label{lem:variation}
Let $v$ be a real periodic $C^3$ function with $v(0)=0$, and put $g_t=g+tv$. For the continuations $x_t,y_t$ of the periodic points above, write $\Delta(g_t):=\lambda_p^{g_t}(x_t)-\lambda_q^{g_t}(y_t)$. Under \eqref{eq:uniform-family},
\[
 \left|\left.\frac{d}{dt}\Delta(g+tv)\right|_{t=0}\right|\le CN\bigl(\max_S|v|+\max_S|v'|\bigr),\qquad
 \left|\frac{d^2}{dt^2}\Delta(g_t)\right|\le CN\|v\|_{C^2}^2.
\]
The second estimate holds at every parameter for which $g_t$ satisfies \eqref{eq:uniform-family}.
\end{lemma}

\begin{proof}
The periodic continuations are twice differentiable in $t$. Indeed, at a periodic point the derivative with respect to $x$ of $g_t^p(x)-x$ is $(g_t^p)'(x)-1>0$, so the implicit function theorem applies in local angular coordinates.

We put $x_i(t):=g_t^i(x_t)$, with cyclic indices modulo $p$. Differentiating $x_{i+1}(t)=g_t(x_i(t))$, we obtain
\[
 \dot x_{i+1}=g_t'(x_i)\dot x_i+v(x_i),\qquad
 \ddot x_{i+1}=g_t'(x_i)\ddot x_i+g_t''(x_i)\dot x_i^2+2v'(x_i)\dot x_i.
\]
These cyclic systems are invertible because $(g_t^p)'(x_t)>1$. We estimate their solutions without introducing a factor depending on $p$. If $V:=\max_i|\dot x_i|$, the first equation gives $\Lambda|\dot x_i|\le V+\max_i|v(x_i)|$. Taking the maximum in $i$ and rearranging, we get the first of the following bounds; the second follows from the same argument:
\[
 \max_i|\dot x_i|\le\frac{\max_i|v(x_i)|}{\Lambda-1},\qquad
 \max_i|\ddot x_i|\le\frac{M_3\max_i|\dot x_i|^2+2\|v'\|_{C^0}\max_i|\dot x_i|}{\Lambda-1}.
\]
In particular, the two quantities are at most $C\|v\|_{C^0}$ and $C\|v\|_{C^1}^2$, respectively.

We next differentiate the periodic length. At $t=0$, with $x_i=x_i(0)$, this gives
\begin{equation}\label{eq:length-variation}
 \left.\frac{d}{dt}\lambda_p^{g_t}(x_t)\right|_{t=0}
 =\sum_{i=0}^{p-1}\left(\frac{v'(x_i)}{g'(x_i)}+\frac{g''(x_i)}{g'(x_i)}\dot x_i\right).
\end{equation}
The first estimate follows from the velocity bound. Differentiating a second time, each summand contributes
\[
 \frac{g_t'''(x_i)\dot x_i^2+g_t''(x_i)\ddot x_i+2v''(x_i)\dot x_i}{g_t'(x_i)}
 -\left(\frac{g_t''(x_i)\dot x_i+v'(x_i)}{g_t'(x_i)}\right)^2.
\]
By \eqref{eq:uniform-family} and the preceding bounds, its absolute value is at most $C\|v\|_{C^2}^2$. The same calculation holds for $y_t$. Summing at most $p+q\le2N$ terms proves the lemma, including when the chosen periods are multiples of the least periods.
\end{proof}

\subsection{A Fourier direction for each comparison}

We first record the Fourier estimates used in the construction. If $v(x)=\sum_{|j|\le K}\widehat v(j)e^{2\pi ijx}$ is a Fourier polynomial, then
\begin{equation}\label{f:coefficient-bounds}
 \|v\|_{C^0}\le\sum_j|\widehat v(j)|,\qquad \|v'\|_{C^0}\le2\pi K\sum_j|\widehat v(j)|.
\end{equation}
Both inequalities follow by taking absolute values in the expansion and its derivative. The convolution formula also gives
\[
 \sum_j|\widehat{vw}(j)|\le\sum_{j,l}|\widehat v(l)|\,|\widehat w(j-l)|
 =\left(\sum_l|\widehat v(l)|\right)\left(\sum_j|\widehat w(j)|\right).
\]
Thus the coefficient sum of a product is at most the product of the coefficient sums.

When $v$ is real and $v(0)=0$, it has the normalized real expansion
\begin{equation}\label{f:normalized-expansion}
 \begin{split}
 v(x)&=\sum_{j=1}^K\bigl[a_j\sin(2\pi jx)+b_j(\cos(2\pi jx)-1)\bigr],\\
 \sum_{j=1}^K(|a_j|+|b_j|)&\le2\sum_{j=-K}^K|\widehat v(j)|.
 \end{split}
\end{equation}
Indeed, $\widehat v(-j)=\overline{\widehat v(j)}$, $a_j=-2\operatorname{Im}\widehat v(j)$, and $b_j=2\operatorname{Re}\widehat v(j)$. The equality $v(0)=0$ determines the constant coefficient, while
\[
 \sum_{j=1}^K(|a_j|+|b_j|)\le4\sum_{j=1}^K|\widehat v(j)|\le2\sum_{j=-K}^K|\widehat v(j)|.
\]

The functions $\sin(2\pi(x-z))$ and $\sin^2(\pi(x-y))$ have degree $1$ and coefficient sum $1$. We also use $W(x):=\cos^{2J}(\pi(x-z))$, where $J\ge1$ is an integer. Its degree is $J$ and its coefficient sum is $1$, as follows from
\[
 W(x)=2^{-2J}\sum_{l=0}^{2J}\binom{2J}{l}e^{2\pi i(l-J)(x-z)}.
\]
Furthermore, $W(z)=1$, $W'(z)=0$, and
\begin{equation}\label{f:kernel-tail-estimate}
 |W(x)|\le e^{-4J\delta^2},\qquad |W'(x)|\le2\pi J e^{-2J\delta^2}
 \quad\text{if }\operatorname{dist}(x,z)\ge\delta.
\end{equation}
To verify these estimates, let $t:=\operatorname{dist}(x,z)\le1/2$. Since $\sin(\pi t/2)\ge t$, we have $\cos(\pi t)=1-2\sin^2(\pi t/2)\le e^{-2t^2}$. Raising this bound to the power $2J$ gives the first estimate. For the derivative, $2J-1\ge J$ gives
\[
 |W'(x)|=2\pi J|\sin(\pi(x-z))|\,|\cos(\pi(x-z))|^{2J-1}\le2\pi J e^{-2Jt^2}.
\]

\begin{prop}\label{f:transversality}
There are constants $c_m,C_m>0$ such that, for every comparison of distinct orbits of periods $p,q\le N$, one of the functions $v_j(x)=\sin(2\pi jx)$ or $v_j(x)=\cos(2\pi jx)-1$, with
\begin{equation}\label{f:deterministic-cutoff}
 1\le j\le K_N:=\left\lceil C_mN^3\Omega^{4N/m}\right\rceil,
\end{equation}
satisfies
\begin{equation}\label{f:mode-transversality}
 \left|\left.\frac{d}{dt}\Delta(g+tv_j)\right|_{t=0}\right|\ge c_m\Omega^{-4(m+1)N}.
\end{equation}
The constants are uniform under \eqref{eq:uniform-family}.
\end{prop}

\begin{proof}
We first construct a perturbation adapted to the two compared orbits. Let $C,z,\delta$ be given by Lemma~\ref{lem:cluster}, with $\delta=\rho\Omega^{-\lceil|S|/m\rceil}$ as in its proof. We put
\[
 v_0(x):=\sin(2\pi(x-z))\prod_{y\in C\setminus\{z\}}\sin^2(\pi(x-y)),
\]
and, if $z\ne0$, multiply this expression by $\sin^2(\pi x)$. The resulting polynomial satisfies $v_0(0)=0$, has degree at most $m+1$, and has coefficient sum at most $1$. Its values vanish on $C$, and its derivatives vanish at every point of $C\setminus\{z\}$.

At $z$, only the derivative of the first factor contributes. Thus
\[
\begin{aligned}
v_0'(z)
&=2\pi\prod_{y\in\mathcal C\setminus\{z\}}
       \sin^2\bigl(\pi(z-y)\bigr)
  \cdot
  \begin{cases}
    1, & z=0,\\
    \sin^2(\pi z), & z\ne0,
  \end{cases}\\
&\ge c_m\Omega^{-4(m+1)N}.
\end{aligned}
\]
Indeed, there are at most $m$ squared factors, and $|\sin(\pi t)|\ge2\operatorname{dist}(t,\mathbb Z)$. Their arguments correspond to distinct periodic points, so \eqref{eq:point-separation} applies, including to the extra point $0$. The extra factor at $0$ causes no difficulty if $0$ already belongs to $C$.

We next make the perturbation small outside $C$. For a constant $A_m\ge1$ to be chosen, we set
\[
 J:=\lceil A_mN\delta^{-2}\rceil,\qquad v(x):=v_0(x)\cos^{2J}(\pi(x-z)).
\]
The values and derivatives prescribed on $C$ are unchanged, and $v(0)=0$. Its coefficient sum is still at most $1$. By \eqref{f:coefficient-bounds}, $\|v_0\|_{C^0}\le1$ and $\|v_0'\|_{C^0}\le2\pi(m+1)$. The product rule and \eqref{f:kernel-tail-estimate} therefore give
\[
 \max_{S\setminus C}(|v|+|v'|)\le C(J+m+1)e^{-2A_mN},
\]
since every point of $S\setminus C$ is at distance at least $\delta$ from $z$ and $J\delta^2\ge A_mN$.

We now estimate the first variation \eqref{eq:length-variation}. The direct derivative terms at the extreme neighbours vanish except at $z$. Since the two compared orbits are disjoint, the occurrences of $z$ belong to only one periodic sum and carry the same sign. Their contribution has absolute value at least $v_0'(z)/\Omega$.

The remaining terms are small, including the motion of the extreme neighbours. Indeed, although $v$ vanishes on $C$, their velocities depend on the values of $v$ along the whole orbit. The velocity estimate in Lemma~\ref{lem:variation} gives
\[
 \max_i|\dot x_i|+\max_j|\dot y_j|\le C\max_S|v|=C\max_{S\setminus C}|v|\le C(J+m+1)e^{-2A_mN},
\]
where a maximum over the empty complement is interpreted as zero. Summing the velocity terms and the direct derivative terms outside $C$ over at most $2N$ positions, we obtain
\[
 \left|\left.\frac{d}{dt}\Delta(g+tv)\right|_{t=0}\right|\ge\frac{v_0'(z)}{\Omega}-CN(J+m+1)e^{-2A_mN}.
\]

By Lemma~\ref{lem:cluster}, $J+m+1\le C_m(A_m+1)N^3\Omega^{4N/m}$. The ratio of the last error to $v_0'(z)/\Omega$ is consequently at most
\[
 C_m(A_m+1)N^4\exp\!\left[-\left(2A_m-\left(4(m+1)+\frac4m\right)\log\Omega\right)N\right].
\]
We choose $A_m$ so that this quantity is at most $1/2$ for every $N\ge1$. This choice is independent of $N$: when $A_m\ge\max\{4,(4(m+1)+4/m)\log\Omega\}$, the expression is bounded by $C_m(A_m+1)N^4e^{-A_mN}\le C_m(A_m+1)e^{-A_m}$, which tends to zero as $A_m\to\infty$. We have proved
\[
 \left|\left.\frac{d}{dt}\Delta(g+tv)\right|_{t=0}\right|\ge c_m\Omega^{-4(m+1)N}.
\]
After increasing $C_m$, the degree of $v$ is at most $K_N$ in \eqref{f:deterministic-cutoff}.

Finally, we pass from this adapted polynomial to a single Fourier mode. We keep $g$ and $v$ fixed while differentiating; the choices of $C,z,J$ and the Fourier coefficients are not differentiated. The first variation is linear in $v$: the cyclic equation for $\dot x_i$ has a fixed invertible linear part and forcing $v(x_i)$, so both terms in \eqref{eq:length-variation} depend linearly on $v$.

By \eqref{f:normalized-expansion}, we can write $v$ using the sine and normalized cosine modes of orders at most $K_N$, with $\sum_j(|a_j|+|b_j|)\le2$. Hence
\[
 \begin{aligned}
 \left|\left.\frac{d}{dt}\Delta(g+tv)\right|_{t=0}\right|
 &\le\left(\sum_j(|a_j|+|b_j|)\right)
 \max_w\left|\left.\frac{d}{dt}\Delta(g+tw)\right|_{t=0}\right|\\
 &\le2\max_w\left|\left.\frac{d}{dt}\Delta(g+tw)\right|_{t=0}\right|,
 \end{aligned}
\]
where $w$ ranges over those modes. At least one mode therefore satisfies \eqref{f:mode-transversality}, after decreasing $c_m$. There is no factor $K_N$ in this loss. The polynomial $v$ depends on the compared orbits, while the modes $v_j$ belong to a fixed list. The selected mode need not retain the zeros of $v$; only its first variation is used below.
\end{proof}

\subsection{Random perturbations and a sublevel estimate}
\label{SSec:Random}

Let $X_j,Y_j$, $j\ge1$, be independent random variables uniformly distributed on $[-1,1]$. For $0<\eta\le1$, we consider
\[
 v_\xi(x):=\eta\sum_{j\ge1}j^{-(k+5/4)}\bigl[X_j\sin(2\pi jx)+Y_j(\cos(2\pi jx)-1)\bigr].
\]
This defines a random element of the Banach space $\{v\in C^k({\Circle},\mathbb R):v(0)=0\}$. Indeed, every summand vanishes at $0$, and its derivatives of orders at most $k$ are bounded by $C_k\eta j^{-5/4}$. Thus the series converges uniformly in the parameters in $C^k$, with $\|v_\xi\|_{C^k}\le C_k\eta$.

Its distribution, denoted by $\mu_\eta$, has compact support. To see this, we give the parameter cube $[-1,1]^{\mathbb N}\times[-1,1]^{\mathbb N}$ its product topology. Finite partial sums depend continuously on finitely many coordinates, and
\[
 \sup_\xi\left\|\eta\sum_{j>L}j^{-(k+5/4)}\bigl[X_j\sin(2\pi j\,\cdot)+Y_j(\cos(2\pi j\,\cdot)-1)\bigr]\right\|_{C^k}
 \le C_k\eta\sum_{j>L}j^{-5/4}\longrightarrow0.
\]
The sum is therefore a continuous map from a compact space into $C^k$. Decreasing $\eta$ places its image in any prescribed ball about zero.

\begin{lemma}\label{lem:slicing}
Suppose $F\in C^2([-1,1]^D)$ and, at every point,
\[
 \max_{1\le i\le D}|\partial_iF|\ge\sigma>0,\qquad |\partial_i^2F|\le H\quad(1\le i\le D).
\]
For normalized Lebesgue measure and $u>0$,
\[
 \mathbb P\{|F|<u\}\le2D\frac{u}{\sigma}\left(1+\frac{4H}{\sigma}\right).
\]
\end{lemma}

\begin{proof}
We first fix all coordinates except the $i$th and denote the resulting one-variable function by $h$. We partition $[-1,1]$ into $M:=\max\{1,\lceil4H/\sigma\rceil\}$ equal intervals. On each interval the variation of $h'$ is at most $2H/M\le\sigma/2$. Consequently, if it contains a point with $|h'|\ge\sigma$, then $|h'|\ge\sigma/2$ throughout the interval and $h'$ has constant sign.

On such an interval, the set where $|h|<u$ has length at most $4u/\sigma$. Indeed, for two points $s<t$ in this set, monotonicity and the mean value theorem give $\sigma(t-s)/2\le|h(t)-h(s)|<2u$. Intervals containing no point with $|h'|\ge\sigma$ do not contribute to the set under consideration. Summing over the intervals and normalizing by the length $2$, we obtain
\[
 \mathbb P\{|F|<u,\ |\partial_iF|\ge\sigma\}\le\frac{2Mu}{\sigma}\le\frac{2u}{\sigma}\left(1+\frac{4H}{\sigma}\right)
\]
after integration over the other coordinates. At every point at least one coordinate derivative has absolute value at least $\sigma$. Taking the union over $i$ proves the claim. Thus the detecting coordinate is allowed to vary from point to point.
\end{proof}

We now fix a normalized $C^k$ map $h$ for which every $g=h+v_\xi$ satisfies \eqref{eq:uniform-family}. We label periodic points by their cyclic order starting at $0$. By Lemma~\ref{Lem:Grid}, the set $\mathbb P_n^g$ always has $d^n-1$ points, and Lemma~\ref{lem:marking} prevents distinct points of bounded periods from colliding. Their local continuations given by the implicit function theorem therefore preserve the labels. The permutation induced by $g$ cannot change without a collision, so the orbit relations between labels are fixed throughout this convex family.

For two labels on distinct orbits, with periods $p,q\le N$, we put $\Delta(g)=\lambda_p^g(x)-\lambda_q^g(y)$ and condition on all coordinates of orders greater than $K_N$. For every fixed tail, $\Delta$ is a $C^2$ function of the remaining $2K_N$ coordinates. Differentiation with respect to $X_j$ or $Y_j$ is variation in the direction $\eta j^{-(k+5/4)}v_j$. Proposition~\ref{f:transversality} therefore gives
\[
 \max_i|\partial_i\Delta|\ge\sigma_N:=c_m\eta K_N^{-(k+5/4)}\Omega^{-4(m+1)N}.
\]
Indeed, the detecting mode has order $j\le K_N$, and its weight is at least $\eta K_N^{-(k+5/4)}$. This mode may depend on the map, which is permitted in Lemma~\ref{lem:slicing}.

Since $k\ge3$, every weighted direction satisfies $\|\eta j^{-(k+5/4)}v_j\|_{C^2}\le C\eta$. Lemma~\ref{lem:variation} consequently gives $|\partial_i^2\Delta|\le CN$, uniformly in the conditioned tail. Since $\eta\le1$, we decrease $c_m$, if necessary, so that $\sigma_N\le1$. Lemma~\ref{lem:slicing} then yields
\[
 \begin{aligned}
 \mathbb P\{|\Delta|<u\mid\text{tail}\}
 &\le CNK_Nu\sigma_N^{-2}\\
 &\le C_{k,m,\eta}NK_N^{2k+7/2}u\,\Omega^{8(m+1)N}.
 \end{aligned}
\]
The periodic lengths are continuous functions of the map, so these events are measurable. Integrating over the tail and substituting \eqref{f:deterministic-cutoff}, we obtain
\begin{equation}\label{eq:sublevel-detailed}
 \mathbb P\{|\Delta(h+v_\xi)|<u\}
 \le C_{k,m,\eta}N^{6k+12}u\exp\!\left[\left(8(m+1)+\frac{8k+14}{m}\right)(\log\Omega)N\right].
\end{equation}
Here $K_N\sigma_N^{-2}$ contributes the power $K_N^{1+2(k+5/4)}=K_N^{2k+7/2}$. Substituting $K_N\le C_mN^3\Omega^{4N/m}$ gives the exponential factor $\Omega^{(8k+14)N/m}$ and the polynomial factor $N^{1+3(2k+7/2)}=N^{6k+23/2}\le N^{6k+12}$. The bound $M_3$ affects only the multiplicative constant in \eqref{eq:sublevel-detailed}.

\subsection{From probability bounds to spectral gaps}

By Lemma~\ref{Lem:Grid}, the number of comparisons with periods $p,q\le N$ is at most
\[
 \left(\sum_{p=1}^N(d^p-1)\right)^2\le C_dd^{2N}.
\]
We use \eqref{eq:sublevel-detailed} with $u=e^{-\beta(\log\Lambda)N}$ and take the union over labels on distinct orbits. The probability of at least one gap smaller than this threshold is at most $CN^{6k+12}e^{-\gamma N}$, where
\[
 \gamma:=\beta\log\Lambda-\left(8(m+1)+\frac{8k+14}{m}\right)\log\Omega-2\log d>0
\]
by \eqref{eq:existence-budget}. These probabilities are summable. The first Borel--Cantelli lemma implies that, almost surely, the threshold is respected for all sufficiently large $N$. No independence between these events is needed.

For a fixed comparison, letting $u\downarrow0$ in \eqref{eq:sublevel-detailed} also gives $\mathbb P\{\Delta=0\}=0$. Since there are only countably many labels, almost surely no two distinct orbits have equal compared lengths. For such a map $g$, we choose $N_0$ larger than its last bad value of $N$ and put
\[
 c(g):=\min\left\{1,\log\Lambda,
 \min_{\substack{1\le p,q<N_0\\x\in\mathbb P_p^g,\ y\in\mathbb P_q^g\\x,y\text{ on distinct orbits}}}
 e^{\beta(\log\Lambda)\max\{p,q\}}|\lambda_p^g(x)-\lambda_q^g(y)|\right\},
\]
where an empty inner minimum is omitted. This is a positive minimum of finitely many positive numbers. Consequently, for every comparison on distinct orbits,
\[
 |\lambda_p^g(x)-\lambda_q^g(y)|\ge c(g)e^{-\beta(\log\Lambda)\max\{p,q\}}
 \ge c(g)e^{-\beta\max\{\lambda_p^g(x),\lambda_q^g(y)\}}.
\]
The first inequality uses the probability estimate for $\max\{p,q\}\ge N_0$ and the choice of $c(g)$ for smaller periods. The second follows from $\lambda_p^g(x)\ge p\log\Lambda$ and $\lambda_q^g(y)\ge q\log\Lambda$, as stated in Lemma~\ref{Lem:SimpleProp}.

It remains to compare two points on the same orbit. Their common least period divides both $p$ and $q$, and each periodic sum consists of complete traversals of the same cycle. Therefore $\lambda_p^g(x)/p=\lambda_q^g(y)/q\ge\log\Lambda$. The entries agree when $p=q$; otherwise,
\[
 |\lambda_p^g(x)-\lambda_q^g(y)|\ge|p-q|\log\Lambda\ge\log\Lambda\ge c(g).
\]
This proves $\beta$-sparsity for the whole unmarked spectrum, including repeated traversals.

The same random family almost surely contains no map $C^1$-conjugate to the linear map. Indeed,
\[
 (h+v_\xi)'(0)=h'(0)+2\pi\eta\sum_{j\ge1}j^{-(k+1/4)}X_j.
\]
After conditioning on every coordinate except $X_1$, the equation $(h+v_\xi)'(0)=d$ specifies at most one value of $X_1$, so it has probability zero. On the other hand, if a $C^1$ diffeomorphism $H$ satisfies $H\circ g=L_d\circ H$, differentiation at $0$ gives $H'(0)g'(0)=dH'(0)$. Since $H'(0)\ne0$, this would force $g'(0)=d$.

\subsection{The translation argument}

We finally verify prevalence in $\mathcal U$ from \eqref{eq:family-U}, using the definition of Hunt--Sauer--Yorke~\cite{HSY}. The estimates above apply to every translated random family satisfying \eqref{eq:uniform-family}: its coordinate directions are the same weighted Fourier modes, and Lemma~\ref{lem:variation} and Proposition~\ref{f:transversality} hold uniformly there. Changing the base map changes only the multiplicative constants.

Let $\mathcal A\subset\mathcal U$ consist of maps with $g'(0)\ne d$ for which there is a positive constant in the spectral gap inequality for every pair of labels on distinct orbits. We have proved that every full random family satisfying \eqref{eq:uniform-family} belongs to $\mathcal A$ with probability one. By the same-orbit calculation above, $\mathcal A$ is a subset of the maps with $\beta$-sparse spectra.

The set $\mathcal A$ is Borel. Indeed, periodic lengths depend continuously on the map, and orbit relations between the ordered labels are fixed on the convex set $\mathcal U$. For a fixed integer $l\ge1$, the inequalities
\[
 |\lambda_p^g(x)-\lambda_q^g(y)|\ge\frac1l e^{-\beta\max\{\lambda_p^g(x),\lambda_q^g(y)\}},
 \qquad x,y\text{ on distinct orbits},
\]
are countably many closed conditions relative to $\mathcal U$. Their intersection, followed by the union over $l$ and intersection with the open condition $g'(0)\ne d$, is precisely $\mathcal A$. In particular, we use a Borel subset with no equalities between distinct orbits; the definition of sparsity itself allows equal spectral entries.

We now fix a $C^k$ ball $B(g_0,r)$ with $B(g_0,3r)\subset\mathcal U$. The larger ball has a uniform $C^3$ bound, since $k\ge3$. We choose $\eta>0$ so that $\operatorname{supp}\mu_\eta\subset\{v:\|v\|_{C^k}\le r\}$. For any normalized degree-$d$ map $h$, suppose first that $h+\operatorname{supp}\mu_\eta$ meets $B(g_0,r)$, and choose $v_0$ with $h+v_0\in B(g_0,r)$. For every $v$ in the support,
\[
 \|h+v-g_0\|_{C^k}\le\|h+v_0-g_0\|_{C^k}+\|v-v_0\|_{C^k}<3r.
\]
Thus the entire translated family lies in $B(g_0,3r)$ and satisfies \eqref{eq:uniform-family}. The probability-one conclusion above applies. If the family does not meet $B(g_0,r)$, there is nothing to estimate. In either case,
\[
 \mu_\eta\{v:h+v\in B(g_0,r)\setminus\mathcal A\}=0\qquad\text{for every }h.
\]
All additions are taken in the affine space of normalized degree-$d$ maps, identified with the Banach space of $C^k$ displacements vanishing at $0$. This is exactly the witness required for the Borel set $B(g_0,r)\setminus\mathcal A$ to be shy. Moreover, its witness can have arbitrarily small support, by decreasing $\eta$.

Since $C^k$ is separable, countably many such balls cover $\mathcal U$. We finish by verifying that the union of their exceptional sets is shy. Let $E_j$ be these Borel exceptional sets, and choose independent random displacements $w_j$ whose distributions witness the shyness of $E_j$ and have compact supports contained in $\{w:\|w\|_{C^k}\le2^{-j}\}$. The series $\sum_jw_j$ converges uniformly over the product of these supports because its tail norm is at most $\sum_{j>L}2^{-j}$. Its sum is a continuous map from a compact product into $C^k$, so its distribution has compact support.

For a fixed translation $h$ and an index $j$, we condition on all $w_i$ with $i\ne j$. Then $h+\sum_{i\ne j}w_i$ is one fixed translation, and the witnessing property of $w_j$ gives
\[
 \mathbb P\left\{h+\sum_iw_i\in E_j\,\middle|\,(w_i)_{i\ne j}\right\}=0.
\]
Integrating over the conditioned variables and taking the countable union in $j$, we obtain $\mathbb P\{h+\sum_iw_i\in\bigcup_jE_j\}=0$. This proves directly the countable-union property of shyness used here; see also \cite{HSY}. Hence $\mathcal U\setminus\mathcal A$ is shy, and $\mathcal A$ is the prevalent subset asserted in Theorem~\ref{Thm:Prev}.

\makeatletter
\begingroup
\let\addcontentsline\@gobblethree
  \section*{Statements and Declarations: Conflict of interest and Data availability}

On behalf of all authors, the corresponding author states that there is no conflict of interest and data sharing does not apply to this article as no datasets were generated or analysed during the current study.
  
\endgroup
\makeatother
	
%%%%%%%%%%%%%%%%%%%%%%%%%%%%%%%%%%%%%%%%%%%%%%%%%%%%%%%%%%
%%%%%%%%%%%%%%%%%%%%%%%%%%%%%%%%%%%%%%%%%%%%%%%%%%%%%%%%%%
%%%%%%%%%%%%%%%%%%%%%%%%%%%%%%%%%%%%%%%%%%%%%%%%%%%%%%%%%%
%%%%%%%%%%%%%%%%%%%%%%%%%%%%%%%%%%%%%%%%%%%%%%%%%%%%%%%%%% 

\end{document}